\documentclass{amsart}

\usepackage{graphicx}

\def\R{{\mathbb R}}

\def\N{{\mathbb N}}

\def\Z{{\mathbb Z}}
\def\1{{1\!\!\!1}}
\def\a{{\alpha}}
\def\E{{\mathbb E}}
\def\P{{\mathbb P}}

\def\cal{\mathcal}

\def\supp{{\rm{supp}}}
\def\eps{\varepsilon}

\newcommand{\be}{\begin{equation}}
\newcommand{\ee}{\end{equation}}
\numberwithin{equation}{section}

\newtheorem{theorem}{Theorem}
\newtheorem{prop}{Proposition}[section]
\newtheorem{cor}{Corollary}[section]
\newtheorem{defi}{Definition}[section]
\newtheorem{lemma}{Lemma}[section]

\title{Martin boundary of a killed random walk in $\Z^n_+$}
\author{Irina Ignatiouk-Robert}
\address{
{Universit\'e de Cergy-Pontoise,}
{D\'epartement de math\'ematiques,}
{2, Avenue Adolphe Chauvin,}
{95302 Cergy-Pontoise Cedex,}
{France}}
\date{\today}
\email{Irina.Ignatiouk@math.u-cergy.fr}

\begin{document}
\begin{abstract} The Martin compactification is investigated for a $d$-dimensional random
  walk which is killed when at least one of it's coordinates becomes zero or negative. The
  limits of the Martin kernel are represented in terms of the harmonic functions of the
  associated induced Markov chains. It is shown that any sequence of points $x_n\in\Z^d_+$
  with $\lim_n |x_n| = \infty$ and $\lim_n x_n/|x_n| = q$ 
  is fundamental in the Martin compactification of $\Z_+^d$ if up to the multiplication  by
  constants, the induced Markov chain
  corresponding to the direction $q$ has a unique positive harmonic
  function. The full Martin compactification is obtained  for Cartesian products of
  one-dimensional random walks. The methods involve  a ratio limit theorem and a large deviation 
  principle for sample paths of scaled processes leading to the logarithmic asymptotics of
  the Green function. 
\end{abstract}
\maketitle

\section{Introduction and main results}\label{sec1}
In the present paper, we investigate the Martin boundary of a random walk on $\Z^d$ which is
killed upon the first time when at least one of its coordinate becomes negative or
zero. Such a random walk $(Z(t))$ is a Markov chain on the state space $\Z_+^d~\dot=~
\{x=(x^1,\ldots,x^d)\in\Z^d: x^i > 0, \; \forall \; i=1,\ldots,d\}$ with a substochastic
transition matrix $(p(x,x') ~=~ \mu(x'-x), \; x,x'\in \Z_+^d)$. the Green function of
$(Z(t))$ is therefore given by 
\[
G(x,x') ~=~ \sum_{t=0}^\infty \P_x(X(t)~=~ x') ~=~ \sum_{t=0}^\infty \P_x(S(t)~=~ x', \;
\tau > t) 
\]
where $(S(t))$ is a homogeneous random walk on $\Z^d$ with transition probabilities
$p(x,x') ~=~ \mu(x'-x)$ and $\tau$ is the first time when the random walk $(S(t))$ exits
from $\Z_+^d$. 

For a transient discrete time Markov chain on a countable discrete  state space $E$ with
the Green function $G(x,x')$, the
Martin compactification of $E$  is the unique smallest
compactification of the discrete set $E$ for which the Martin kernels  
$$K(x,\cdot) ~\dot=~ G(x,\cdot)/G(x_0,\cdot)$$ 
extend continuously for all $x\in E$. An explicit description of the Martin compactification is usually a non-trivial problem and the
most of the existing results in this domain were obtained for so-called homogeneous random
walks, when the transition probabilities of the process are invariant with respect to the
translations over the state space $E$ (see the book of Woess~\cite{Woess}). For non-homogeneous Markov
chains, there are few examples where the Martin compactification was identified. 
Cartier~\cite{Cartier} identified the Martin compactification for random walks on
 non-homogeneous trees. Doney~\cite{Doney:02} described the harmonic functions and the Martin boundary of 
a random walk $(Z(n))$ on $\Z$ killed  on the negative half-line
$\{z : z<0\}$,  Alili and
Doney~\cite{Alili-Doney}  extend this result for the corresponding space-time random  walk $S(n)
=(Z(n),n)$. For Brownian motion on a half-space, the Martin boundary was
 obtained in the book of Doob~\cite{Doob}. Kurkova
and  Malyshev~\cite{Kurkova-Malyshev} described the  Martin boundary  for nearest
neighbor  random  walks on $\Z\times\N$ and  on $\Z^2_+$ with reflected conditions on the
boundary. The recent results of Raschel~\cite{Raschel} and Kurkova and
Raschel~\cite{Kurkova-Raschel} identify the  Martin compactification for random walks in
$\Z_+^2$ with  jumps at distance at most $1$ and absorbing boundary. 
All these results use the methods that seem to be unlikely to apply in a more
general situation. The methods of Doney~\cite{Doney:02} and  Alili and 
Doney~\cite{Alili-Doney} rely a one-dimensional structure of the process $(Z(n))$. For
Brownian motion, in ~\cite{Doob}, the explicit form of the Green function is used. Kurkova
and  Malyshev~\cite{Kurkova-Malyshev}, Raschel~\cite{Raschel} and Kurkova and
Raschel~\cite{Kurkova-Raschel} use an analytical method where the geometrical structure of  the elliptic
curve defined by the jump  generating function of the random walk is crucial : this method
work only if the corresponding elliptic
curve is homeomorphic to the torus. It seems therefore difficult to extend this method for
 higher dimensions  and for arbitrary jumps.  

In the present paper we develop large deviation method proposed in ~\cite{Ignatiouk:06},
where the  Martin compactification was identified for a random walks in a
half-space $\Z^{d-1}\times\Z_+$ killed on the boundary.  The main ideas of this method can be summarized as
follows~: 
\begin{itemize}
\item[--] The  ratio limit theorem allow to identify the limits of the Martin kernel 
$K(z,z_n)$ when the logarithmic asymptotic of the Green function for a given sequence $(x_n)$ is
zero. 
\item[--] The logarithmic asymptotics of the Green function were obtained form the sample
  path large deviation estimates of scaled processes, in terms of the corresponding
  quasipotential. 
\item[--] The ratio limit theorem is applied for a twisted Markov process, with an appropriated
  exponential change of the measure,  for which the corresponding logarithmic asymptotic
  of the Green  function is zero.  
\item[--] The limits of the Martin kernel of the original random walk are obtained from those of
  twisted random walk by using the inverse change of the measure. 
\end{itemize}
In \cite{Ignatiouk:07} this method was used to identify Martin compactification  for a random walks in a
half-space $\Z^{d-1}\times\Z_+$ with reflected boundary. The ratio limit theorem and the
large deviation estimates for the Green function were combined there with Pascal's method
applied with a suitable renewal equation. 

In the recent paper of Ignatiouk and Loree~\cite{Ignatiouk-Loree} the large deviation
method was applied to describe Martin compactification for a random walk in $\Z^2$ killed upon the first
exit from $\Z_+^2 =\{x=(x^1,x^2)\in\Z^2 : \; x^i > 0\}$.  The main ideas of this paper are
the following: for a sequence of points $x_n\in \Z_+^2$ with $\lim_n|x_n|=\infty$ and $\lim_n
x_n/|x_n| = q =(q^1,q^2)\in\R^2$   the limits of the Martin kernel were deduced from the
those of the corresponding local random walk on $\{x\in\Z^d : x^i > 0
\; \text{ for all } \; i\in\Lambda \}$ with $\Lambda=\{i\in\{1,2\} :
q^i = 0\}$. Such a local random walk is  obtained from the original random walk on
$\Z^2_+$ by removing the boundary $\{x\in\Z^2 : x^i = 0 \; \text{ for }\;
i\not\in\Lambda\}$, the transition probabilities of the original random walk are then
extended on the larger state space $\{x\in\Z^d : x^i > 0
\; \text{ for all } \; i\in\Lambda \}$ by homogeneity. For the local random walk, the
limiting behavior of the Martin kernel was identified by using the methods of the paper~\cite{Ignatiouk:06}. 

In the present paper, the results of Ignatiouk and Loree~\cite{Ignatiouk-Loree} are
generalized to higher dimensions and under  more general assumptions on the jumps
probabilities. In a difference to the paper ~\cite{Ignatiouk-Loree} where the jumps
probabilities were assume to decrease to zero at the infinity faster than any exponential
function,  here  the jumps
probabilities  on a distance $r>0$ can decrease as an exponential function $\exp(-\delta
r)$ with a given $\delta >0$. Such a more general setting require new careful estimates of the Martin
kernels for which the methods developed in ~\cite{Ignatiouk-Loree} are not sufficient (see
Section~\ref{sec2} below for more details).

The jumps probability measure $\mu$ is assumed to satisfy the following conditions: 

{\em 
\begin{itemize}
 \item[{\bf (A1)}] For any $\Lambda\subset\{1,\ldots,d\}$ and $x,x'\in\Z^{\Lambda,d}_+ ~\dot=~ \{x\in\Z^d : x_i > 0 \text{ for all }
  i\in\Lambda\}$, there is a sequence
  $x_0,\ldots,x_k\in\Z^{\Lambda,d}_+$ with $x_0=x$ and $x_k=x'$ such that
  $\mu(x_i-x_{i-1}) > 0$ for all $i=1,\ldots,k$. 
 \item[{\bf (A2)}] The  jump generating function 
\be\label{e1-2}
\varphi(a) ~\dot=~ \sum_{z\in \Z^d} \mu(z)
~\exp(a\cdot z)
\ee 
is finite in a neighborhood of the set $D~\dot=~\{a\in\R^d : \varphi(a)\leq 1\}$ in
$\R^d$ and 
\item[{\bf (A3)}] 
\be\label{e1-1}
M~\dot=~ \sum_{x\in\Z^d} \mu(x) \, x ~\not=~ 0. 
\ee
\end{itemize}
}
Remark that according to the assumption~(A1) the homogeneous random walk $(S(t))$ on
$\Z^d$ with transition probabilities $p(z,z')=\mu(z'-z)$  is irreducible.

For $i\in\{1,\ldots,
d\}$, we denote  by $x^i$ the $i$-th coordinate of $x\in\R^d$. Similarly, 
$S^i(t)$ denotes the $i$-th coordinate of $S(t)$ and  
\[
\tau_i ~\dot=~ \inf\{n \geq 0 :~ S^i(n) \leq 0\}
\]
is the first time when the $i$-th coordinate of the random walk $(S(t))$ becomes negative
or zero. 
For a subset $\Lambda$ of $\{1,\ldots ,d\}$, we denote  $\Lambda^c ~\dot=~
\{1,\ldots,d\}\setminus\Lambda$
and 
\[
\tau_\Lambda ~\dot=~ \min_{i\in\Lambda} \tau_i ~=~ \inf\{n\geq 0 :~ S^i(n)\leq 0 \;
  \text{for some $i\in\Lambda$}\}.
\]
Moreover, for $x\in\R^d$, we let $x^\Lambda ~\dot=~(x^i)_{i\in\Lambda} \in\R^\Lambda$ and similarly $(S^\Lambda(t)) = (S^i(t))_{i\in\Lambda}$.
For $\Lambda\not=\emptyset$,  the process $(S^\Lambda(t))$ is a random walk
on the lattice $$\Z^{\Lambda}~\dot=~ \left\{u=(u^i)_{i\in\Lambda} : u^i\in\Z\right\}$$ with
transition probabilities $p_\Lambda(u,u') ~\dot=~ \mu_\Lambda(u'-u)$ where 
\be\label{eq1-3}
\mu_\Lambda(u) ~\dot=~ \sum_{x\in \Z^d:~ x^\Lambda = u} \mu(x), \quad
u\in\Z^\Lambda. 
\ee
A substochastic random walk on $\Z_+^\Lambda = \{u\in\Z^\Lambda : u^i > 0 \; \text{ for
  all } \; i\in\Lambda\}$ with transition matrix $(p_\Lambda(u,u') = \mu_\Lambda(u'-u), \;
u,u'\in\Z^\Lambda_+)$ which is identical to $S^\Lambda(t)$ for $t <
\tau_\Lambda$ and killed at the time $\tau_\Lambda$ is denoted by $(X^\Lambda(t))$.
 We call  $(X^{\Lambda}(t))$  {\em the induced Markov chain} of the random walk
$(S(t))$ corresponding to a given subset $\Lambda$ of $\{1,\ldots, d\}$. 
For $\Lambda = \{1,\ldots,d\}$ we have therefore $
\tau_\Lambda ~=~ \tau$ and $X^\Lambda(t) = Z(t)$. It  is convenient
moreover to introduce the induced Markov chain for $\Lambda=\emptyset$. In the last case,
the random walk $(S^\Lambda(t))$ and the induced Markov chain $(X^\Lambda(t))$ are assumed
to be constant : almost surely $\tau_\emptyset = +\infty$ and  $S^\emptyset(t) ~=~
X^\emptyset(t) ~=~ \eta$ for all $t\geq 0$ with some additional state $\eta$.

For $q\in\R^d$ we consider
$\Lambda(q) ~=~ \{i\in\{1,\ldots,d\}~:~ q_i = 0\}$ and to simplify the notations, we let
$X^{\Lambda(M)}(t) = X_M(t)$.

Recall that  a sequence
$x_n\in \Z_+^d$ with $\lim_n|x_n|=\infty$ converges to a point on the Martin boundary 
$\partial_{\cal M}(\Z_+^d)$ of $\Z_+^d$ for a given Markov process
$(Z(t))$  if and
only if the sequence of functions $
K(x,\cdot) ~\dot=~ G(x,\cdot)/G(x_0,\cdot)$ 
converges point-wise on $\Z_+^d$. Another equivalent definition of the Martin
compactification is the following : letting  
\[
w_x ~=~ \exp\left(-\sum_{i=1}^d x^i\right)\times\P_{x_0}(Z(t) ~=~ x \; \text{ for some } t \geq 0) 
\]
define the metric 
\[
d_{\cal M}(y,y') ~=~ \sum_{x\in\Z^d_+} w_x \left(|K(x,y) - K(x,y')| + |\delta_{x,y} - \delta_{x,y'}|\right) 
\]
with $\delta_{x,x} ~=~ 1$ and $\delta_{x,y}=0$ for $x\not= y$. Then completion of the
metric space $(\Z^d_+, d_{\cal M})$ is the Martin compactification of $\Z_+^d$. A
fundamental sequence $(x_n)$ in the metric space $(\Z^d_+, d_{\cal M})$ is said to be fundamental in
the Martin compactification of $\Z_+^d$ for the Markov process
$(Z(t))$.  A sequence $x_n\in\Z_+^d$ is therefore fundamental for $(Z(t))$ if and only if
it converges to a point of the Martin boundary of $\Z_+^d$.  

Our main result is the following statement.

\begin{theorem}\label{th1} Suppose that the conditions (A1)-(A3) are satisfied and let the
  coordinates of the mean $M$ be non-negative. Then for any harmonic function $f>0$ of the
  induced Markov chain
  $(X_{M}(t))$, the function 
\be\label{e1-3p}
h(x) = f\left(x^{\Lambda(M)}\right) - \E_x\left(f(S^{\Lambda(M)}(\tau)), \;
\tau < \tau_{\Lambda(M)}\right)
\ee
is strictly positive and harmonic for $(Z(t))$ and for any fundamental sequence of points $x_n\in
  \Z_+^{d}$ with $\lim_n |x_n| = \infty$ and $\lim_n 
  x_n/|x_n| = M/|M|$ there is a harmonic function $f>0$ of the induced Markov chain
  $(X_{M}(t))$ such that for any $x\in\Z^d_+$, 
\be\label{e1-3}
\lim_n  {G(x,x_n)}/{G(x_0,x_n)} ~=~ \frac{f\left(x^{\Lambda(M)}\right) - \E_x\left(f(S^{\Lambda(M)}(\tau)), \;
\tau < \tau_{\Lambda(M)}\right)}{f\left(x_0^{\Lambda(M)}\right) - \E_{x_0}\left(f(S^{\Lambda(M)}(\tau)), \;
\tau < \tau_{\Lambda(M)}\right)}   
\ee
If moreover a harmonic function $f>0$ of the induced Markov chain $(X_{M}(t))$ is
  unique to constant multiples, then any sequence  $x_n\in
  \Z_+^{d}$ with $\lim_n |x_n| = \infty$ and $\lim_n 
  x_n/|x_n| = M/|M|$ is fundamental for $(Z(t))$ and satisfies \eqref{e1-3}.
\end{theorem}

Recall that in a particular case, when the coordinates of the mean vector $M$ are all
non-zero (i.e. when $\Lambda(M)=\emptyset$)  the induced Markov chain $(X_M(t))$ is
constant and $\tau_{\Lambda(M)} = \infty$. In this case, the harmonic functions of
$(X_M(t))$ are therefore constant and the function \eqref{e1-3p} is a constant multiple
of the function 
\[
x \to  \P_x(\tau = \infty). 
\]
Using therefore Theorem~\ref{th1} one gets 

\begin{cor}\label{cor1-0} Suppose that the conditions (A1)-(A3) are satisfied and let the
  coordinates of the mean $M$ be strictly positive. Then 
\begin{itemize}
\item[--] the function $h(x) ~=~ \P_x(\tau=\infty)$ is a strictly positive and harmonic
for the Markov chain $(Z(t))$, 
\item[--] any sequence  $x_n\in
  \Z_+^{d}$ with $\lim_n |x_n| = \infty$ and $\lim_n 
  x_n/|x_n| = M/|M|$ is fundamental for $(Z(t))$ and 
 for any $x\in\Z^d_+$,
\[
\lim_n  {G(x,x_n)}/{G(x_0,x_n)} ~=~ \P_x(\tau = \infty)/\P_{x_0}(\tau = \infty).
\]
\end{itemize} 
\end{cor} 

Consider now the case when at least one of the coordinates of the mean vector $M$ is
zero, i.e. when $\Lambda(M)\not=\emptyset$. In this case, the induced Markov chain
$(X_M(t))$ is identical to the homogeneous random walk $(S^{\Lambda(M)}(t))$ on
$\Z^{\Lambda(M)}$ before the time $\tau_{\Lambda(M)}$  and killed at the  time
$\tau_{\Lambda(M)}$, i.e. when at least one of the
coordinates of $(S^{\Lambda(M)}(t))$  becomes zero or negative. Remark moreover that the
mean jump of the random walk $(S^{\Lambda(M)}(t))$ is equal to zero because 
\[
\E_0(S^{\Lambda(M)}(1)) ~=~ \sum_{x\in\Z^d} \mu(x) \, x^{\Lambda(M)} ~=~ M^{\Lambda(M)} ~=~ 0. 
\]
Hence, to identify the limiting behavior of the Martin kernel $G(x,x_n)/G(x_0,x_n)$ for
a sequence of points $x_n\in\Z^d_+$ with $\lim_n|x_n| ~=~ \infty$ and $\lim_n x_n/|x_n| ~=~ M$, one has to identify the positive
harmonic functions of a random walk  on $\Z^l$ (with $l=|\Lambda(M)|$) which has zero mean
and 
is killed at the first exit
from $Z^l_+$. Unfortunately,  for $l>1$ there  are is now general results in this
domain. We hope that our paper will motivate the efforts in order to solve such a 
non-trivial problem.

If $|\Lambda(M)|=1$, i.e. when $\Lambda(M) =\{i\}$ for some $1\leq i\leq d$, the induced Markov
chain $(X_M(t))$ is a homogeneous random walk on $\Z$ killed when hitting the negative
half-line $\{k\in\Z : k\leq 0\}$. In this case,  the harmonic functions can be described by
using the results of Doney~\cite{Doney:02} (see also
  Example E 27.3 in Chapter VI of Spitzer~\cite{Spitzer}). Here, using   
Theorem~\ref{th1} one gets 

\begin{prop}\label{pr1-1} Suppose that the conditions (A1)-(A3) are satisfied and let the
  coordinates of the mean $M$ be non-negative. Suppose moreover that only one of the
  coordinates of $M$ is zero, i.e. $\Lambda(M)~=~\{i\}$ for some $1\leq i\leq d$. Then 
\begin{itemize}
\item[--] the function $
h_M(x) ~=~ x^i - \E_x\left(S^i(\tau)\right)$  is  strictly positive and
harmonic for $(Z(t))$, 
\item[--] any sequence  $x_n\in
  \Z_+^{d}$ with $\lim_n |x_n| = \infty$ and $\lim_n 
  x_n/|x_n| = M/|M|$ is fundamental for $(Z(t))$ and 
 for any $x\in\Z^d_+$,  
\[
\lim_n  {G(x,x_n)}/{G(x_0,x_n)} ~=~ h_M(x)/h_M(x_0).
\]
\end{itemize} 
\end{prop} 
The proof of this proposition is given in Section~\ref{sec9}. 

\medskip

While for $|\Lambda(M)| > 1$, the harmonic functions of the induced Markov chain
$(X_M(t))$ are not known in general, the results of Picardello and
  Woess~\cite{Picardello-Woess} allow us to identify them under the following additional
  assumption 

\begin{itemize}
 \item[{\bf (A4)}] {\em $\mu_{\Lambda(M)}(u) = 0$  if  $u^iu^j\not= 0$  for some \; $i <
   j$, $i,j\in\Lambda(M)$, 
i.e. only one coordinate of $X_M(t)$ can change during a transition
$X_M(t)\to X_M(t+1)$.} 
\end{itemize}
This condition is satisfied for a nearest neighbor random walk and more
generally, for Cartesian products of one-dimensional random
walks~(see~\cite{Picardello-Woess, Woess}). 
 Using  Theorem~\ref{th1} we obtain 
\begin{prop}\label{pr1-2} Under the hypotheses (A1)-(A4), the following assertions hold : 
\begin{itemize}
\item[--] the function 
\[
h_M(x) = \prod_{i\in\Lambda(M)} x^i - \E_x\left(\prod_{i\in\Lambda(M)} (S^i(\tau))\right)
\]
is strictly positive and harmonic for the Markov chain $(Z(t))$,
\item[--] any sequence $x_n\in\Z_+^d$ with $\lim_n |x_n| ~=~
  \infty$ and $\lim_n x_n/|x_n| ~=~ M/|M|$ is fundamental  for  $(Z(t))$ and for any $x\in\Z_+^d$, 
\[
\lim_n  {G(x,x_n)}/{G(x_0,x_n)} ~=~ h_M(x)/h_M(x_0). 
\]
\end{itemize}
\end{prop}
The proof of this proposition is given in Section~\ref{sec9}.

\medskip

To identify the limiting behavior of the Martin kernel ${G(x,x_n)}/{G(x_0,x_n)}$ for a
sequence of points $x_n\in\Z_+^d$ with $\lim_n |x_n| =\infty$ and 
$\lim_n x_n/|x_n| = q$ for an arbitrary vector 
\[
q\in{\cal S}_+^d ~\dot=~ \{x\in\R^d :~ |x|=1
\text{ and } x^i \geq 0 \text{ for all } i=1,\ldots,d\}
\]
the method of the exponential change of 
measure is used. Namely, for a given $a\in D~\dot=~ \{a\in\R^d :~\varphi(a)\leq 1\}$ we consider  
 a twisted random walk $(S_a(t))$ on $\Z^d$ with transition probabilities $
p_a(x,x') ~=~ \mu(x'-x) \exp(a\cdot(x'-x))$. 
Such a random walk is stochastic if and only if the point $a$ belongs to the boundary
$\partial D~\dot=~ \{a\in\R^d :~\varphi(a) = 1\}$ of $D$. For $a\in\partial D$, we let
\[
\tau^a ~\dot=~ \inf\bigl\{t\geq 0 : ~S^i_a(t) \leq 0 \; \text{ for some } \; i\in\{1,\ldots,d\}\bigr\}
\]
and we denote by $(Z_a(t))$ the random walk on $\Z^d$ which is identical to $(S_a(t))$ before the
time $\tau^a$ and killed at the time $\tau^a$.  The Green function  of the
twisted random walk $(Z_a(t))$ is denoted by $G_a(x,x')$. 

Furthermore, under the  assumptions (A1)-(A4), the set $D$ 
is compact and strictly convex, the gradient $\nabla\varphi(a)$ exists everywhere on
$\R^d$ and does not vanish on the
boundary $\partial D ~=~ \{a\in\R^d :~\varphi(a)= 1\}$, and the mapping 
\be\label{e1-6}
a\to q(a) ~\dot=~ \nabla\varphi(a)/|\nabla\varphi(a)|
\ee
determines a homeomorphism from $\partial D$ to the unit sphere ${\cal
  S}^d=\{q\in\R^d:~ |q| ~=~ 1\}$ (see~\cite{Hennequin}). We denote by $q\to a(q)$ the inverse mapping of 
\eqref{e1-6} and we let $a(q) = a(q/|q|)$ for a non-zero $q\in\R^d$. According to this
notation, $a(q)$ is the only point in $\partial D$ where the
vector $q$ is normal to the convex set $D$.  
For $q\in {\cal S}_+^d$ and $a=a(q)$, the mean  of the twisted random walk
$(S_{a(q)}(t))$ is given by 
\[
M(a(q)) ~=~ \sum_{x\in\Z^d} \mu(x) \, \exp(a(q)\cdot x) \, x ~=~
\left.\nabla\varphi(a)\right|_{a=a(q)} ~=~ \left|\left.\nabla\varphi(a)\right|_{a=a(q)}\right| q
\] 
and consequently,
\[
M(a(q))/|M(a(q))| ~=~ q. 
\]
Since clearly, 
\[
G_{a(q)}(x,x') ~=~ \exp(a(q)\cdot (x'-x)) G(x,x'), 
\]
the limiting behavior of the Martin kernel $G(x,x_n)/G(x_0,x_n)$ when $|x_n|\to\infty$ and
$x_n/|x_n|\to q$ can be obtained from Theorem~\ref{th1} applied for  the
twisted random walk $(S_{a(q)}(t))$. For instance, using the equality 
\[
\P_x(\tau_{a(q)} = \infty) = 1 - \P_x(\tau_{a(q)} < \infty) ~=~ 1 - \E_x(\exp(a(q)\cdot
(S(\tau) - x), \; \tau < \infty),
\]
Corollary~\ref{cor1-0} applied for the
twisted random walk $(S_{a(q)}(t))$ provides the 
following statement.

\begin{cor}\label{cor1-2} Suppose that the conditions (A1)-(A3) are satisfied and let the
  coordinates of $q\in{\cal S}_+^d$ be strictly positive. Then 
\begin{itemize}
\item[--] the function $h_q(x) ~=~ \exp(a(q)\cdot x ) - \P_x(\exp(a(q)\cdot S(\tau)), \tau <
  \infty)$ is strictly positive and harmonic for the Markov chain $(Z(t))$, 
\item[--] any sequence of points $x_n\in
  \Z_+^{d}$ with $\lim_n |x_n| = \infty$ and $\lim_n 
  x_n/|x_n| = q$ is fundamental for $(Z(t))$ and 
 for any $x\in\Z^d_+$,
\[
\lim_n  {G(x,x_n)}/{G(x_0,x_n)} ~=~ h_q(x)/h_q(x_0).
\]
\end{itemize} 
\end{cor} 

Similarly, from Proposition~\ref{pr1-1} it follows 
\begin{cor}\label{cor1-3} Suppose that the conditions (A1)-(A3) are satisfied and let only one of the
  coordinates of $q\in{\cal S}_+^d$ be zero : i.e. $\Lambda(q)~=~\{i\}$ for some $1\leq i\leq d$. Then 
\begin{itemize}
\item[--] the function $
h_q(x) ~=~ x^i\exp(a(q)\cdot x ) - \E_x\left(S^i(\tau)\exp(a(q)\cdot S(\tau)), \; \tau < \infty \right)$  is  strictly positive and
harmonic for $(Z(t))$, 
\item[--] any sequence  $x_n\in
  \Z_+^{d}$ with $\lim_n |x_n| = \infty$ and $\lim_n 
  x_n/|x_n| = q$ is fundamental for $(Z(t))$ and 
 for any $x\in\Z^d_+$,  
\[
\lim_n  {G(x,x_n)}/{G(x_0,x_n)} ~=~ h_q(x)/h_q(x_0).
\]
\end{itemize} 
\end{cor} 

Finally, with Proposition~\ref{pr1-2},  one gets the full Martin
compactification for $(Z(t))$ under the following additional assumption 

\begin{itemize}
 \item[{\bf (A4')}] {\em $\mu(x) = 0$ \; if \; $x^i x^j\not= 0$ \; for some \; $1\leq i <
   j\leq d$.} 
\end{itemize}

\begin{cor}\label{cor1-4}
Under the hypotheses (A1)-(A3) and (A4'), for any $q\in{\cal S}_+^d$, 
\begin{itemize}
\item[--] the function 
\[
h_q(x) =  \exp(a(q)\cdot x) \prod_{i\in\Lambda(q)}  x^i  ~-~ \E_x\left(\exp(a(q)\cdot S(\tau)
) \prod_{i\in\Lambda(q)} S^i(\tau)), \; \tau<\infty \right)
\] 
is strictly positive and harmonic for $(Z(t))$, 
\item[--]  any sequence of points $x_n\in
  \Z_+^{d}$ with $\lim_n |x_n| = \infty$ and $\lim_n 
  x_n/|x_n| = q$ is fundamental for the Markov chain $(Z(t))$ and  for any $x\in\Z^d_+$, 
\[
\lim_n{G(x,x_n)}/{G(x_0,x_n)} = h_q(x)/h_q(x_0) 
\]
\end{itemize}
\end{cor}

\medskip 

Our paper is organized as follows. In Section~\ref{sec2}, the main ideas of the
proof of our main result are given. We introduce there a local random walk 
$(Z_{\Lambda(M)}(t))$ which has asymptotically the same statistical behavior as the Markov
chain $(Z(t))$ on $\Z^d_+$ far from the boundary $\cup_{i\not\in\Lambda(M)}\{x^i = 0
\}$. For a sequence  $x_n\in\Z_+^d$ with $\lim_n|x_n| =\infty$
and $\lim_n x_n/|x_n| = M/|M|$,  the limiting behavior of the Martin kernel $G(x,x_n)/G(x_0,x_n)$
of the Markov chain $(Z(t))$ is obtained from the limiting behavior of the Martin kernel
$G_{\Lambda(M)}(x,x_n)/G_{\Lambda(M)}(x_0,x_n)$ of the local process
$(Z_{\Lambda(M)}(t))$. Our main tools are large deviation estimates of the Green function
and a ratio limit theorem for the local 
process $(Z_{\Lambda(M)}(t))$. The large deviation estimates of the Green functions are obtained from the sample
path large deviation estimates of the scaled processes $X_\eps(t) ~\dot=~ \eps
Z([t/\eps])$ and $Z_{\Lambda(M)}^\eps(t)~\dot=~ \eps Z_{\Lambda(M)}([t/\eps])$ in
Section~\ref{sec3}. Section~\ref{sec4} is devoted to the estimates of the hitting
probabilities of the induced Markov chain $(X_M(t))$. The ratio limit theorem and its
corollaries are given in Sections~\ref{sec5} and~\ref{sec6}. Section~\ref{sec7} is devoted
to the limiting behavior of the Martin kernel
$G_{\Lambda(M)}(x,x_n)/G_{\Lambda(M)}(x_0,x_n)$ of the local Markov-additive process
$(Z_{\Lambda(M)}(t))$ and in Section~\ref{sec8}, Theorem~\ref{th1} is proved. In
Section~\ref{sec9} we prove Propositions~\ref{pr1-1} and ~\ref{pr1-2}.

\section{Local Markov-additive processes and the main ideas of the proof}\label{sec2}

For $\Lambda\subset \{1,\ldots,d\}$ we introduce a random walk $(Z_\Lambda(t))$ on 
\[
\Z_+^{\Lambda,d} ~=~ \{x\in\Z^d :~ x_i > 0,\;  \forall i\in\Lambda\}, 
\]
with a substochastic transition matrix $\left(p(x,x') ~=~ \mu(x'-x), \; x,x'\in
\Z_+^{\Lambda,d}\right)$. This random walk is identical to $(S(t))$ for $t < \tau_\Lambda$
and killed at the time $\tau_\Lambda$. Recall that $\tau_\Lambda$ denotes the first time
when one of the coordinates of the vector $S^\Lambda(t)=(S^i(t))_{i\in\Lambda}$ becomes negative or
zero and the
induced Markov chain $X^\Lambda(t)$ is  killed at the time $\tau_\Lambda$ and identical to
$S^\Lambda(t)$ for $t<\tau_\Lambda$. 
Another equivalent representation of the induced Markov chain $(X^\Lambda(t))$ is therefore
the following : 
\[
X^\Lambda(t) ~=~ (Z_\Lambda^i(t), \, i\in\Lambda) ~\dot=~ Z_\Lambda^\Lambda(t).
\]
Remark moreover that the transition probabilities of the Markov process $(Z_\Lambda(t))$ are
invariant with respect to the translations on $x$ for all  $x\in\Z^d$ with $x^{\Lambda} =
0$:~
\[
P_\Lambda(x',x'') ~=~ P_\Lambda(x'+x, x''+x),\quad \forall x',x''\in \Z_+^{\Lambda,d}.
\]
 Hence, according to the usual terminology, our process $(Z_\Lambda(t))$ is 
Markov-additive with  an
additive part $Z_\Lambda^{\Lambda^c}(t) = (Z^i_\Lambda(t), \, i\in\Lambda^c)$ and a
Markovian part $Z^\Lambda_{\Lambda}(t) = X^\Lambda(t)$. To simplify 
the notation we denote $Y_\Lambda(t) ~\dot=~ Z_\Lambda^{\Lambda^c}(t)$ and we identify
$Z_\Lambda(t)$ with $(X^\Lambda(t), Y_\Lambda(t))$. Similarly, it is convenient to
identify the points $x\in  \Z_+^{\Lambda,d}$ and $(x^\Lambda, x^{\Lambda^c}) \in
\Z^{\Lambda}_+\times\Z^{\Lambda^c}$.   

The random walk $(Z_\Lambda(t))$ is called a
{\em local Markov-additive process } corresponding to the set
$\Lambda\subset\{1,\ldots,d\}$. Remark that for $\Lambda = \emptyset$, this is our
homogeneous random walk $(S(t))$ on
$\Z^d$, while for $\Lambda=\{1,\ldots,d\}$, $Z_\Lambda(t) ~=~ Z(t)$ is the random walk with the
same transition probabilities as $(S(t))$ on $\Z_+^d $ and killed upon the first time $\tau = \min_i
\tau_i$ when one of its coordinates  becomes negative or zero. 
The Green function 
\[
G(x,x') ~\dot=~ \sum_{n=0}^\infty \P_x(S(n) = x', \; \tau > n) 
\]
of the random walk $(Z(t))$ can be represented in terms of the Green function  
\[
G_\Lambda(x,x') ~\dot=~ \sum_{n=0}^\infty \P_x(Z_\Lambda(n) = x')
~=~ \sum_{n=0}^\infty \P_x(S(n) = x', \; \tau_{\Lambda} > n) 
\]
of the random walk $(Z_\Lambda(t))$   in the following way : for any $x,x'\in\Z^d_+$, 
\be\label{e2-1}
G(x,x') ~=~ G_\Lambda(x,x')  ~- \E_x\Bigl( G_\Lambda\bigl(S(\tau),x'\bigr), \; \tau < \tau_\Lambda\Bigr) 
\ee
The main steps of the proof of Theorem~\ref{th1} are the following :~

\medskip 
\noindent
{\em Step 1.} The transition probabilities of the local random walk $Z_\Lambda(t) =
(X^\Lambda(t),Y_\Lambda(t))$ are invariant with respect to the translations on $x\in\Z^d$
with $x^\Lambda=0$, and hence, a function $f>0$ on $\Z^{\Lambda}_+$ is harmonic for the induced Markov chain
$(X^\Lambda(t))$ is and only if the function $
h(x) ~=~ f(x^{\Lambda})$ 
is harmonic for the local process $(Z_\Lambda(t))$. 
 For $\Lambda = \Lambda(M) ~\dot=~ \{i\in\{1,\ldots,d\} :~
M_i=0\}$ we prove the following property : if a sequence of points $x_n\in
\Z_+^{\Lambda,d}$  with $\lim_n |x_n| = \infty$ and $\lim_n x_n/|x_n| = M/|M|$ is
fundamental for the local process $(Z_\Lambda(t))$ then 
\be\label{e2-2}
\lim_n  {G_{\Lambda}(x,x_n)}/{G_{\Lambda}(x_0,x_n)} ~=~ f(x^\Lambda)/f(x_0^\Lambda),
\quad \forall x\in \Z_+^{\Lambda,d},
\ee
for some harmonic function $f > 0$ of $(X_M(t))$.  This result is obtained
by using the method developed in \cite{Ignatiouk:06} :~ a ratio limit theorem is combined
with  large deviation estimates of the Green function $G_{\Lambda}(x,x_n)$. With the
large deviation estimates we prove that  
\[
\lim_n  \frac{1}{|x_n|} \log  G_{\Lambda}(x,x_n) ~=~ 0 
\]
and next, using the ratio limit theorem of the paper \cite{Ignatiouk:06} we deduce that 
\be\label{e2-3}
\lim_n  {G_{\Lambda}(x,x_n)}/{G_{\Lambda}(\hat{x},x_n)} ~=~ 1  
\ee 
for all $x,\hat{x}\in\Z_+^{\Lambda,d}$ with $x^{\Lambda} = \hat{x}^{\Lambda}$. Furthermore, we show that  the limit 
\be\label{e2-4}
h(x) ~=~ \lim_k  {G_{\Lambda}(x,x_{n_k})}/{G_{\Lambda}(x_0,x_{n_k})}  
\ee
is a harmonic function of $(Z_\Lambda(t))$ and from \eqref{e2-3} we get the equality $h(x)
~=~ h(\hat{x})$ for all $x,\hat{x}\in\Z^{\Lambda,d}_+$ with $
x^\Lambda ~=~ \hat{x}^\Lambda$. If the only harmonic functions of the induced Markov chain
$(X^{\Lambda(M)}(t))$ are the constant multiples
of the function $h(x) ~=~ f(x^\Lambda)$,  the above arguments prove that any sequence of
points $x_n\in
\Z_+^{\Lambda,d}$  with $\lim_n |x_n| = \infty$ and $\lim_n x_n/|x_n| = M/|M|$ is
fundamental for the local process $(Z_\Lambda(t))$ and satisfies the equality \eqref{e2-2}.

\medskip 
\noindent
{\em Step 2.} The renewal equation \eqref{e2-1} is next used to deduce the following result
: if a  sequence
points 
$x_n\in \Z_+^d$ with $\lim_n |x_n| = \infty$ 
and $\lim_n x_n/|x_n| = M/|M|$ is fundamental for the local process $(Z_{\Lambda(M)}(t))$ then
it is  also fundamental for the random walk $(Z(t))$ and satisfies the equality 
\eqref{e1-3} with some harmonic function $f > 0$ of $(X_M(t))$. 
The main ideas  are here the following~: 
 for $\Lambda = \Lambda(M)$, using  \eqref{e2-1} and \eqref{e2-2} one can get  the equality  
\[
\lim_n {G(x,x_n)}/{G_{\Lambda}(\hat{x},x_n)} ~=~ \Bigl(f\left(x^{\Lambda}\right) - \E_x\left(f(S^{\Lambda}(\tau)), \;
\tau < \tau_{\Lambda}\right)\Bigr)/{f\bigl(\hat{x}^{\Lambda}\bigr)} 
\]
if one can prove the exchange of the limits  
\[
\lim_n ~\E_x\left( \frac{G_{\Lambda}(S(\tau),x_n)}{G_{\Lambda}(\hat{x},x_n)}, \; \tau =
\tau_{\Lambda^c} < \tau_{\Lambda}\right)
=~ \E_x\left( \lim_n \frac{G_{\Lambda}(S(\tau),x_n)}{G_{\Lambda}(\hat{x},x_n)}, \; \tau =
\tau_{\Lambda^c} < \tau_{\Lambda}\right) 
\]
In a straightforward way, the last equality seems very difficult to obtain : the classical
convergence theorems do not work here because there are no known suitable uniform estimates of the Martin kernel
${G_{\Lambda}(w,x_n)}/{G_{\Lambda}(\hat{x},x_n)}$. We first show that the right hand side
of \eqref{e2-1} can be decomposed into a main part  
\[
G_\Lambda(x,x_n)  ~-~ \E_x\Bigl( G_\Lambda(S(\tau),x_n), \; \tau < \tau_\Lambda, \; |S(\tau)| <
\delta|x_n|\Bigr) 
\]
and a corresponding negligible part 
with an arbitrary $\delta > 0$, by using large deviation estimates of the Green functions. Next we prove that 
\begin{multline}\label{e2-6}
\lim_n \E_x\left( \frac{G_\Lambda(S(\tau),x_n)}{G_{\Lambda}(x',x_n)}, \; \tau < \tau_\Lambda, \; |S(\tau)| <
\delta|x_n|\right)  
\\ =~ \E_x\left( \lim_n \frac{G_{\Lambda}(S(\tau),x_n)}{G_{\Lambda}(\hat{x},x_n)}, \; \tau =
\tau_{\Lambda^c} < \tau_{\Lambda}\right)
\end{multline}
by using the dominated convergence theorem is used : it is shown that  
\be\label{e2-7}
\E_x\left(\exp(\eps |Z_\Lambda(\tau)|), \;
\tau < \tau_{\Lambda}\right) ~<~ \infty 
\ee 
for $\eps > 0$ small enough and it is proved that for any
$x'\in\Z^{\Lambda,d}_+$ and $\eps > 0$ there are $\delta
> 0$ and $C>0$ such that 
\be\label{e2-8}
\1_{\{|w| < \delta |x_n|\}} {G_{\Lambda}(w,x_n)}/{G_{\Lambda}(x',x_n)}  ~\leq~ C
\exp(\eps |w|), \quad \forall w\in\Z^{\Lambda,d}_+. 
\ee

A similar method was earlier used in \cite{Ignatiouk-Loree} where the Martin compactification was
identified for a random walk on $\Z^2$ killed upon the first exit from $\Z^2_+$. In our
setting, the main steps of the proof are quite similar to those of the paper
\cite{Ignatiouk-Loree} but the intermediate results are much more delicate to
get because of a higher dimension and a more general assumption on the transition
probabilities of the process, in a difference of the paper ~\cite{Ignatiouk-Loree}, we do
not assume the
jump generating function \eqref{e1-2}  to be 
finite everywhere in $\R^d$ but only in a neighborhood of the set $D$. Such a more general setting is important
in view of the applications to a large class of random walks where the probability of the 
jumps decreases exponentially when the size jumps  tends to infinity.  

If the
jump generating function \eqref{e1-2} is finite everywhere in $\R^d$, any exponential
function is integrable with respect of the measure $\mu$, and the fact that the 
limit \eqref{e2-4} is a harmonic function of the local Markov-additive process
$(S^{\Lambda(M)}(t))$  is a simple consequence of a rough estimate of the Martin kernel
$G_{\Lambda(M)}(x,x_{n_k})/G_{\Lambda(M)}(x',x_{n_k})$ by an exponential function $C(x') \exp(
\kappa |x|)$ with a large constant $\kappa >0$.  Such a  rough estimate easily follows from the Harnack inequality. 
In our setting, the only exponential functions $C(x') \exp(
\kappa |x|)$ which are integrable with respect to the measure $\mu$ are those with a small
$\kappa > 0$ and hence, we
need a more careful estimate of the Martin kernel
$G_{\Lambda(M)}(x,x_{n_k})/G_{\Lambda(M)}(x',x_{n_k})$.

Another point where the arguments of the paper \cite{Ignatiouk-Loree}  do not work is
the proof of the inequality \eqref{e2-8}. The most difficult is here the case  when at least one of
the coordinates of the mean vector $M$ is zero. In \cite{Ignatiouk-Loree}, the proof of
\eqref{e2-8} for such a vector $M$ heavily relies on the fact that the corresponding
induced Markov chain $(X_{M}(t))$ is a recurrent random walk on $\Z$ killed when
hitting the negative half-line $\{ k\in\Z : k \leq 0\}$. An important property of such a
random walk, which is essential for the proof of \eqref{e2-8} in \cite{Ignatiouk-Loree}, is that for any $e\in\Z$,
\be\label{e2-9}
\P_k( X_{M}(t) = k+e \; \text{ for some } t > 0) ~\to~ 1 \quad \text{ as } \quad
k\to\infty 
\ee
(see \cite{Ignatiouk-Loree}). In a more general case,
 the induced Markov chain $(X_{M}(t))$  is a random
walk on $\Z^{\Lambda(M)}$ having a finite variance, zero mean and killed upon the first exit from
$\Z_+^{\Lambda(M)}$. Remark that for $Card(\Lambda(M)) \geq 3$ a random walk with a finite
variance and zero drift on $\Z^{\Lambda(M)}$ is transient and moreover, for
$Card(\Lambda(M)) \geq 2$, the property \eqref{e2-9} fails to hold when
$k\in\Z^{\Lambda(M)}_+$ tends to infinity along one of the axis $\{x \in\Z^{\Lambda(M)}_+ :
x_i > 0 \; \text{ and } \; x_j = 0 \; \text{ for } \; j\not= i\}$,
$i\in\Lambda(M)$. 
In the present paper,  the inequality \eqref{e2-8} and the careful desired estimates of the Martin kernel
$G_{\Lambda(M)}(x,x_{n_k})/G_{\Lambda(M)}(x',x_{n_k})$ 
are obtained by using a new approach, from the ratio limit theorem and large deviation estimates of the induced Markov chain
$(X_{M}(t))$.

\section{Large deviation results}\label{sec3}
In this section we obtain large deviation estimates for scaled local Markov-additive processes and we
deduce from them the large deviation asymptotics of the Green functions. Before to prove the
large deviation results we show that the local Markov-additive processes satisfy the
following communication condition. 
\subsection{Communication condition.}
\begin{defi}  A discrete time Markov chain $(\cal{Z}(t))$ on a countable state space $E\subset\Z^d$ is
  said to satisfy the communication condition on $E_0\subset E$ if there exist  
$\theta >0$ and $C>0$ such that for any $x\not=x'$, $x,x'\in E_0$ there is a sequence of
  points $x_0, x_1, \ldots,x_n\in E_0$  with $x_0=x$, $x_n=x'$ and 
    $n\leq C|x'-x|$ such that  
\[
 |x_i-x_{i-1}| ~\leq~ C \quad \text{ and  } \quad \P_{x_{i-1}}({\cal Z}(1) = x_i) ~\geq~ \theta, \quad \quad \forall \;
 i=1,\ldots,n. 
\] 
\end{defi}
\begin{lemma}\label{lem3-1} Under the hypotheses (A1), for any
  $\Lambda\in\{1,\ldots,d\}$, the local Markov-additive process $(Z_\Lambda(t))$ 
  satisfies the communication condition  on $\Z_+^{\Lambda,d}$. 
\end{lemma}
\begin{proof} To prove this proposition it is sufficient to show that for any unit vector
  $e\in\Z^d$ there is a sequence of vectors $u_{e,1}, \ldots,u_{e,n(e)}\in \supp(\mu) ~\dot=~
  \{x\in\Z^d : \mu(x) > 0\}$ with $u_{e,1} +  \ldots +u_{e,n(e)} = e$ such that  
\be\label{e3-0}
 x + u_{e,1} + \cdots + u_{e,k} \in \Z^{\Lambda,d}_+, \quad \forall k=1,\ldots,n(e), \quad \text{whenever $x, x+e\in\Z^{\Lambda,d}_+$.}
\ee
The communication condition will  be satisfied then 
with 
\[
\theta = \min_e \min_k \mu(u_{e,k}) > 0 \quad \text{and} \quad C = \max_e \max \left\{d\, n(e), \,\max_{k}
|u_{e,k}|\right\}. 
\]

Suppose first that the coordinates $e^i$ of the unit vector
$e$  are non-negative for all $i\in\Lambda$
(i.e. either $e^i=0$ for all $i\in\Lambda$ or  $e^i = 1$ for some $i\in\Lambda$ and
$e^j=0$ for $j\not= i$) and let us consider  
$\hat{x}\in\Z_+^{\Lambda,d}$ with $\hat{x}^i = 1$ for
$i\in\Lambda$ and $x^i=0$ for $i\in\Lambda^c$. Then clearly, $\hat{x}+e\in
\Z_+^{\Lambda,d}$ and there are  $u_{e,1}, \ldots,u_{e,n(e)}\in \supp(\mu) ~\dot=~ \{x\in\Z^d : \mu(x) > 0\}$ with
$u_{e,1} +
  \ldots +u_{e,n(e)} = e$ and  
\[
\hat{x} + u_{e,1} + \cdots + u_{e,k} \in \Z^{\Lambda,d}_+, \quad \forall k=1,\ldots,n(e),
\]
because the Markov process $(Z_\Lambda(t))$ is irreducible on $\Z^{\Lambda,d}_+$. Since 
for any $x\in\Z_+^{\Lambda,d}$ and $i\in\Lambda$, the $i$-th coordinate $(x + u_{e,1} + \cdots + u_{e,k})^i$ of
the vector $x + u_{e,1} + \cdots + u_{e,k}$ is greater or equal to the  $i$-th coordinate
$(\hat{x} + u_{e,1} + \cdots + u_{e,k})^i$ of the vector $\hat{x} + u_{e,1} + \cdots + u_{e,k}$  then  we
get also \eqref{e3-0} for all $x\in\Z_+^{\Lambda,d}$. 

Similarly, when a unit vector $e$ has a negative non-zero coordinate  $e^i = -1$ for
some $i\in\Lambda$ (and consequently, $e^j=0$ for $j\not= i$), we consider $\hat{x}\in\Z_+^{\Lambda,d}$ with
$\hat{x}^i= 2$, \, $\hat{x}^j = 1$ for $j\in\Lambda, j\not= i$ and $\hat{x}^j = 0$ for
$j\in\Lambda^c$.  For such a point $\hat{x}$, one has $\hat{x} + e\in \Z_+^{\Lambda,d}$
and consequently, there is 
sequence of vectors $u_{e,1}, \ldots,u_{e,n(e)}\in \supp(\mu)$ with
$u_{e,1} +
  \ldots +u_{e,n(e)} = e$ and  
\[
\hat{x} + u_{e,1} + \cdots + u_{e,k} \in \Z^{\Lambda,d}_+, \quad \forall k=1,\ldots,n(e).
\]
Moreover, for any  point $x\in\Z^{\Lambda,d}_+$ for which $x + e\in\Z^{\Lambda,d}_+$, one gets 
$x^i \geq 2 = \hat{x}^i$ and $x^j \geq 1 = \hat{x}^j$ for $j\in\Lambda, j\not= i$. For all 
$x\in\Z_+^{\Lambda,d}$ for which $x + e\in\Z^{\Lambda,d}_+$, the $i$-th coordinate $(x + u_{e,1} + \cdots + u_{e,k})^i$ of
the vector $x + u_{e,1} + \cdots + u_{e,k}$ is therefore greater or equal to the  $i$-th coordinate
$(\hat{x} + u_{e,1} + \cdots + u_{e,k})^i$ of the vector $\hat{x} + u_{e,1} + \cdots + u_{e,k}$  and
consequently \eqref{e3-0} holds.
\end{proof}
\subsection{Large deviation properties of scaled processes.} To formulate the large
deviation result we need to introduce the following notations : 
$D([0,T],\R^d)$ denotes  the set of all right continuous functions with left
limits from $[0,T]$ to $\R^d$  endowed with Skorohod metric
(see Billingsley~\cite{Billingsley}). We let $$\R_+^{\Lambda,d} =\{x\in\R^d :~ x_i\geq 0,
\; \forall i\in\Lambda\}.$$ For $x\in\R^{\Lambda,d}$ 
we denote by  $[x]$ the nearest lattice
point to $x$ in $\Z_+^{\Lambda,d}$.  For $t\in\R_+$,   $[t]$ denotes the integer
part of $t$. 

The following proposition proves the lower large deviation bound for the family of scaled
random walks $Z_\Lambda^\eps(t)~\dot=~\eps Z_\Lambda([t/\eps])$ in $D([0,T],\R^d)$ with the rate
function 
\be\label{e3-1}
I_{[0,T]}^\Lambda(\phi) ~\dot=~ \begin{cases} \int_0^T (\log \varphi)^*(\dot\phi(t)) \, dt, &\text{ if
    $\phi$ is absolutely continuous and}\\
&\text{ $\phi(t)\in\R_+^{\Lambda,d}$ for all $t\in[0,T]$,}\\
+\infty &\text{ otherwise,}
\end{cases}
\ee
where 
$\varphi$ is the jump generating function defined by \eqref{e1-2} and $(\log \varphi)^*$
denotes  the convex conjugate of the function $\log\varphi$ defined by 
\[
(\log \varphi)^*(v) ~=~ \sup_{a\in\R^d} \Bigl(a\cdot v - \log\varphi(a)\Bigr), \quad
v\in\R^d.
\]
Recall that a continuous function $\phi :[0,T]\to \R_+^{\Lambda,d}$ is called absolutely
continuous if the derivative $\dot\phi(t)$ exists almost everywhere on $[0,T]$ (with
respect to the Lebesgue measure on $[0,T]$) and for any $t\in[0,T]$,
\[
\phi(t) ~=~ \phi(0) + \int_0^t\dot\phi(s) \, ds.
\]

\begin{prop}\label{pr3-1} Under the hypotheses (A1) and (A2), for any  $x\in\R^{\Lambda,d}_+$, $T>0$ and an open set ${\cal
O}\subset D([0,T],\R^d)$,
\begin{equation}\label{e3-2}
\lim_{\delta\to 0} \;\liminf_{\eps\to 0} \; \inf_{x'\in \eps E : |x'-x|<\delta} \eps
\log\P_{[x'/\eps]}\left( Z_\Lambda^\eps(\cdot)\in {\cal 
O}\right) \geq -\inf_{\phi\in{\cal O}:\phi(0)=x} I_{[0,T]}(\phi), 
\end{equation}
\end{prop}
The proof of \eqref{e3-2} uses the
communication condition of Proposition~\ref{pr3-1} together with 
the lower  large deviation bound of  Mogulskii's theorem~(see~\cite{D-Z}) and is 
quite similar to the proof of the corresponding lower bound of Proposition~4.1 of the
paper~\cite{Ignatiouk:06}. While the hole
Mogulskii's theorem was proved under a more restrictive condition, when the jump
generating function \eqref{e1-2} is finite everywhere on $\R^d$, for the proof of its 
lower bound  our Assumption (A3) is sufficient (see ~\cite{D-Z}).

\subsection{Large deviation estimates of the Green functions.}
The large deviation estimates of scaled processes $Z_\Lambda^\eps(t)~\dot=~\eps
Z_\Lambda([t/\eps])$ are now used to get the large deviation estimates of the Green
functions 
\[
G_\Lambda(x,x') ~\dot=~ \sum_{t=0}^\infty \P_x(Z_\Lambda(t)=x') ~=~ \sum_{t=0}^\infty
\P_x(S(t)=x', \; \tau_\Lambda > t).
\]

\begin{prop}\label{pr3-2} Under the hypotheses (A1)-(A3), for any 
$q,q'\in\R^{\Lambda,d}_+$ and any sequences $x_n,x_n'\in \Z_+^{\Lambda,d}$ and $\eps_n >0$
with $\lim_n \eps_n = 0$, $\lim_n \eps_nx'_n ~=~q'$ and $\lim_n \eps_nx_n ~=~q$, 
\be\label{e3-5}
\liminf_{n\to\infty} \eps_n \log G_\Lambda(x'_n,x_n) ~\geq~ - \sup_{a \in D} a\cdot (q-q').
\ee
\end{prop}
\begin{proof} The proof of this proposition is similar to the proof of Proposition~4.2 of
Ignatiouk-Robert~\cite{Ignatiouk:06}.  The main arguments of this proof are the following
:  

For any $r>0$ and $T>0$,  using the lower large deviation bound
\eqref{e3-2} with an open set ${\cal O} =\{\phi : |\phi(T) - q| < r\}$ one gets 
\begin{align}
\lim_{\delta\to 0}\liminf_{n\to\infty}~&\inf_{\substack{x'\in\Z^{\Lambda,q}_+:~|\eps_n x' - q'| < \delta}}~
\eps_n\log \sum_{\substack{x\in\Z^{\Lambda,q}_+:~ |\eps_n x - q| < r}} G_\Lambda(x',x) \nonumber\\
 &\geq~ \lim_{\delta\to 0}\liminf_{n\to\infty}\inf_{\substack{x'\in\Z^{\Lambda,q}_+:~
    |\eps_n x' -q'| < \delta}}
\eps_n\log  \P_{x'}(|Z_\Lambda^{\eps_n}(T)| < r)  \nonumber\\
&\geq~ -\inf_{\phi: \phi(0)=q', |\phi(T)- q| < r } I^\Lambda_{[0,T]}(\phi)\nonumber\\
&\geq~ -\inf_{\phi: \phi(0)=q', \phi(T) = q  } I^\Lambda_{[0,T]}(\phi) \label{e3-6}
\end{align}
Furthermore, by Lemma~\ref{lem3-1}, for any $x\not=x'$, $x,x'\in \Z^{\Lambda,d}_+$ there is a sequence of
  points $w_0, w_1, \ldots,w_k\in \Z^{\Lambda,d}_+$  with $w_0=x'$, $w_k=x$ and 
    $k\leq C|x'-x|$ such that  
\[
 |w_i-w_{i-1}| \leq C \quad \text{ and  } \quad \P_{w_{i-1}}(Z_\Lambda(1) = w_i) \geq \theta, \quad \quad \forall \;
 i=1,\ldots,k. 
\] 
From this it follows that  for any $x\not=x'$, $x,x'\in \Z^{\Lambda,d}_+$, 
  there is $0<t\leq C|x-x'|$ such that 
\[
\P_{x'}(Z_\Lambda(t) = x) ~\geq~ \theta^t ~\geq~ \theta^{C|x'-x|}
\]
and consequently, 
\begin{align*}
G_\Lambda\bigl(x',x_n\bigr) &~\geq~  G_\Lambda(x',x) \,\P_{x}(Z_\Lambda(t) = x_n)  ~\geq~ G_\Lambda\bigl(x',x\bigr)\theta^{C|x-x_n|}
\\&~\geq~G_\Lambda\bigl(x',x\bigr)\theta^{C|x-q/\eps_n| + C|x_n-q/\eps_n|}.
\end{align*}
Using moreover the inequality $
\text{Card}\{x\in\Z^d : |x - q/\eps_n|< R\} ~\leq~ (2 R + 1)^{d}$ 
with $R = r/\eps_n$  one obtains 
\[
G_\Lambda\bigl(x',x_n\bigr) ~\geq~ \frac{1}{(1 + 2r/\eps_n)^{d}} \theta^{2C r/\eps_n}
\sum_{\substack{x :~|q - \eps_n x| < r}}
G_\Lambda\bigl(x',x\bigr)
\]
for all those $n\in\N$ for which  $|q
- \eps_n x_n| < r$ and consequently, for any $r>0$, 
\begin{multline*}
\lim_{\delta\to 0}\liminf_{n\to\infty}~\inf_{\substack{x':~ |\eps_n x' -q'| <
    \delta}} ~\eps_n\log G_\Lambda\bigl(x',x_n\bigr) ~\geq~ - 2C r \log \theta \\  +
\lim_{\delta\to 0}\liminf_{n\to\infty}~\inf_{\substack{x':~ |\eps_n x'-q'| < \delta}}~\eps_n
~\log \sum_{\substack{x :~|q - \eps_n x|< r }} G_\Lambda\bigl(x',x\bigr). 
\end{multline*}
Letting at the last inequality $r\to 0$ and using  \eqref{e3-6} one gets 
\[
\liminf_{n\to\infty} \eps_n \log G_\Lambda(x'_n,x_n) ~\geq~ - \inf_{T > 0} 
~\inf_{\phi:~\phi(0)=q', \,\phi(T)=q} I^\Lambda_{[0,T]}(\phi) 
\]
Since for $\phi(t)= q' + (q-q')t/T$, 
\[
I^\Lambda_{[0,T]}(\phi) ~=~ \int_0^T (\log \varphi)^*(\dot\phi(t)) \, dt ~=~ T(\log\varphi)^*\left(\frac{q-q'}{T}\right)
\]
and by Theorem~13.5 of Rockafellar~\cite{R}, 
\[
\inf_{T > 0}
T(\log\varphi)^*\left(\frac{q-q'}{T}\right) ~=~ \sup_{a : \varphi(a)\leq 1} a\cdot (q-q'), 
\]
from the last inequality it follows that 
\[
\liminf_{n\to\infty} \eps_n \log G_\Lambda(x'_n,x_n) ~\geq~ - \inf_{T > 0}
T(\log\varphi)^*\left(\frac{q-q'}{T}\right) ~=~ - ~\sup_{a \in D} a\cdot (q-q')
\]
and consequently, \eqref{e3-5} holds. 
\end{proof} 

A straightforward consequence of this proposition is the following statement.

\begin{cor}\label{cor3-1} Under the hypotheses (A1)-(A4), for any 
$q\not=q'$, $q,q'\in\R^{\Lambda,d}_+$,  and any sequences $x_n,x_n'\in \Z_+^{\Lambda,d}$ and $\eps_n >0$
with $\lim_n \eps_n = 0$, $\lim_n \eps_nx'_n ~=~q'$ and $\lim_n \eps_nx_n ~=~q$, 
\be\label{e3-7}
\liminf_{n\to\infty} \eps_n \log G_\Lambda(x'_n,x_n) ~\geq~ - a(q-q')\cdot (q-q') 
\ee
\end{cor}
\begin{proof} Indeed, under the hypotheses (A1)-(A4), the point $a(q-q')$  is
  the only point on the boundary $\partial D$ of the set $D~\dot=~\{a\in\R^d :
  \varphi(a)\leq 1\}$ where the vector $q-q'$ is normal to the set $D$ and consequently,
  the right hand side of \eqref{e3-5} is equal to the right hand side of \eqref{e3-7}. 
\end{proof}

Another immediate consequence of Proposition~\ref{pr3-2} is the following statement. 
\begin{cor}\label{cor3-2} Suppose that the conditions (A1)-(A4) are satisfied and let $M^i \geq 0$ for all $i\in\Lambda$.
  Then for any sequences $x'_n,x_n\in \Z_+^{\Lambda,d}$ and $\eps_n >0$
with $\lim_n \eps_n = 0$, $\lim_n \eps_nx'_n ~=~0$ and $\lim_n \eps_nx_n ~=~ M$, 
\be\label{e3-8}
\liminf_{n\to\infty} \eps_n \log G_{\Lambda}(x'_n,x_n) ~=~ 0 
\ee
\end{cor}
\begin{proof} To deduce this estimate from \eqref{e3-7} it is sufficient to notice that
  for $q'=0$ and $q=M$, under the hypotheses (A1)-(A4), one has $a(q-q') = a(M) = 0$. 
\end{proof}

\section{Ratio limit theorem}\label{sec5}
In this section we identify the limiting behavior of the Martin kernel 
$${G_{\Lambda(M)}(x,x_n)}/{G_{\Lambda(M)}(\hat{x},x_n)}$$ when $\lim_n|x_n| = \infty$ and $\lim_n
x_n/|x_n| = M/|M|$ for $x,\hat{x}\in\Z^{\Lambda,d}_+$ with $x^\Lambda = \hat{x}^\Lambda$.  To get
this results, the large
deviation estimates of the Green function $G_{\Lambda(M)}(x,x_n)$ 
are combined with the results of the paper
~\cite{Ignatiouk-Loree}. Proposition~7.3 of ~\cite{Ignatiouk-Loree} applied for the
Markov-additive process $(Z_\Lambda(t))$ proves the following statement. 

\begin{prop}\label{pr5-1} Suppose that the conditions (A1) - (A3) are satisfied and let $\mu(0) > 0$.  Suppose moreover that 
  a sequence of
  points $x_n\in\Z^{\Lambda,d}_+$ is such that  $\lim_n|x_n| = \infty$ and 
\[
\lim_{\delta\to 0} ~\liminf_{n\to\infty}~\inf_{x\in \Z^{\Lambda,d}_+:~ |x| < \delta|x_n|}
~\frac{1}{|x_n|} \log G_\Lambda\bigl(x, x_n\bigr)  ~\geq~ 0.
\]
 Then 
\begin{multline}\label{e5-1}
\lim_{\delta\to 0} ~\liminf_{n\to\infty}~\inf_{x\in \Z^{\Lambda,d}_+:~ |x| < \delta|x_n|}
G_\Lambda(x+w,x_n)/G_\Lambda(x,x_n) \\~=~ \lim_{\delta\to 0}
~\limsup_{n\to\infty}~\sup_{x\in \Z^{\Lambda,d}_+:~ |x| < \delta|x_n|} G_\Lambda(x+w,x_n)/G_\Lambda(x,x_n)  ~=~ 1  
\end{multline}
for all $x\in\Z^{\Lambda,d}_+$ and $w\in\Z^d$ with $w^\Lambda = 0$.
\end{prop}
Corollary~\ref{cor3-2} combined with Proposition~\ref{pr5-1} provides the following
statement.

\begin{prop}\label{pr5-2} Suppose that the conditions  (A1)-(A3) are satisfied. Then for
  $\Lambda=\Lambda(M)$, relations \eqref{e5-1} hold for any sequence of points 
  $x_n\in\Z^{\Lambda(M),d}_+$ with $\lim_n |x_n| = \infty$ and $\lim_n \, x_n/|x_n| = M/|M|$.
\end{prop}
\begin{proof} Under the hypotheses (A1) - (A3), Corollary~\ref{cor3-2} and Proposition~\ref{pr5-1}
  imply this  statement for strongly aperiodic random walk $(Z_\Lambda(t))$, i.e. when $\mu(0)>0$.  Hence, to prove our proposition
  we need to show that the last assumption  can be omitted. For this we consider a modified
  substochastic random walk $(\tilde{Z}_\Lambda(t))$ on $\Z^{\Lambda,d}_+$ with transition
  probabilities 
\[
\P_x(\tilde{Z}_\Lambda(1) ~=~ x') ~=~ (1-\eps) \mu(x'-x) + \eps \delta_{x,x'}
\]
where $0 < \eps < 1$ and $\delta_{x,x'}$ denotes Kronecker's symbol : $\delta_{x,x}=1$
and $\delta_{x,x'}=0$ for $x\not=x'$. One can represent the random walk
$(\tilde{Z}_\Lambda(t))$ 
in terms of the random walk  $({Z}_\Lambda(t))$ as follows : let $(\theta_n)$ be a
sequence of independent identically 
distributed Bernoully random variables with $\P(\theta_n = 1) = 1-\eps$ which are
independent on the random walk $({Z}_\Lambda(t))$, then letting 
$
N(n) ~=~ \theta_1 + \cdots + \theta_n$, one gets $
\tilde{Z}_\Lambda(n)  ~=~ Z_\Lambda(N(n))$ for all $n\in\N$. For the Green function $\tilde{G}_\Lambda(x,x')$
of the modified random walk $(\tilde{Z}_\Lambda(t))$ one gets therefore
\begin{align}
\tilde{G}_\Lambda(x,x') &~=~  \sum_{n=0}^\infty\P_x( {Z}_\Lambda(N(n)) = x') ~=~ 
\sum_{n=0}^\infty\sum_{k=0}^n \P_x( {Z}_\Lambda(k) = x') \,\P(N(n)=k)   
\nonumber\\
&~=~   \sum_{k=0}^\infty\P_x( {Z}_\Lambda(k) = x') \sum_{n=k}^\infty C_n^k (1-\eps)^k
\eps^{n-k}  ~=~ (1-\theta)^{-1} G_\Lambda(x,x').\label{e5-2}
\end{align}
For $\Lambda = \Lambda(M)$, the last equality and Corollary~\ref{cor3-2} applied for the random walk
$(Z_\Lambda(t))$ show that  
\begin{align*}
\lim_{\delta\to 0} ~\liminf_{n\to\infty}~\inf_{x\in \Z^{\Lambda,d}_+:~|x| < \delta|x_n|}
~\frac{1}{|x_n|} \log \tilde{G}_\Lambda\bigl(x, x_n\bigr) &~=~\\
\lim_{\delta\to 0} ~\liminf_{n\to\infty}~\inf_{x\in \Z^{\Lambda,d}_+: ~|x| < \delta|x_n|}
~\frac{1}{|x_n|} \log G_\Lambda\bigl(x, x_n\bigr)  &~\geq~ 0.
\end{align*}
The new random walk $(\tilde{Z}_\Lambda(t))$ satisfies therefore the conditions
of Proposition~\ref{pr5-1}. Hence,  for $\Lambda = \Lambda(M)$ and
$w\in\Z^{\Lambda,d}_+$ with $w^\Lambda = 0$, 
\begin{multline*}
\lim_{\delta\to 0} ~\liminf_{n\to\infty}~\inf_{x\in \Z^{\Lambda,d}_+: ~|x| < \delta|x_n|}
\tilde{G}_\Lambda(x+w,x_n)/\tilde{G}_\Lambda(x,x_n) \\~=~ \lim_{\delta\to 0}
~\limsup_{n\to\infty}~\sup_{x\in \Z^{\Lambda,d}_+:~ |x| < \delta|x_n|} \tilde{G}_\Lambda(x+w,x_n)/\tilde{G}_\Lambda(x,x_n)  ~=~ 1
\end{multline*}
and using again the equality \eqref{e5-2} we get \eqref{e5-1}.
\end{proof}
An immediate consequence of Proposition~\ref{pr5-2} is the following statement.

\begin{cor}\label{cor5-1} Suppose that the conditions (A1) - (A3) are satisfied
  and  let  a sequence $x_n\in
  \Z_+^{\Lambda(M),d}$ with $\lim_n |x_n| = \infty$ and $\lim_n 
  x_n/|x_n| = M/|M|$ converge to a point of the Martin boundary $\partial_{\cal
    M}(\Z_+^d)$ for the random walk $(Z_{\Lambda(M)}(t))$. Then  the limit 
\be\label{e5-3}
h(x) ~=~ \lim_{n\to\infty} 
G_{\Lambda(M)}(x,x_n)/G_{\Lambda(M)}(x_0,x_n) 
\ee
satisfies the equality 
\be\label{e5-4}
h(x+w) ~=~ h(x) \quad \text{ for all $x\in\Z^{\Lambda,d}_+$ and $w\in\Z^d$ with $w^\Lambda = 0$.}
\ee 
\end{cor}
Remark moreover that under the hypotheses of Corollary~\ref{cor5-1}, by Fatou's lemma, the
function $h$ defined by \eqref{e5-3} is  super-harmonic 
for the random walk $(Z_{\Lambda(M)}(t))$ and hence, letting for $u\in\Z^{\Lambda(M)}_+$, 
\[
f(x^{\Lambda(M)}) ~=~ h(x) \quad \text{ with $x\in\Z_+^{\Lambda(M),d}$ such that
  $x^{\Lambda(M)} = u$} 
\] 
and using \eqref{e5-4} one gets a positive function $f$ on $\Z^{\Lambda(M)}_+$ satisfying the inequality 
\[
\E_{u}\left(f(X_M(t))\right)  ~=~ 
\E_x\left(h(Z_{\Lambda(M)}(t))\right) \leq h(x) ~=~ f(u). 
\]
The resulting function $f$ is therefore super-harmonic for the induced Markov chain
$(X_M(t))$. Moreover,  if the function
\eqref{e5-3} is harmonic for $(Z_{\Lambda(M)}(t))$ then the last inequality holds with the
equality and consequently, the function $f$ is
harmonic for the induced Markov chain
$(X_M(t))$. The next step of our proof shows that under the hypotheses of
Corollary~\ref{cor5-1}, 
the function \eqref{e5-3} is always harmonic for $(Z_{\Lambda(M)}(t))$. This result would
be a simple consequence of dominated convergence theorem and the Harnack inequality if
instead of the assumption (A2) we assume that the jump generating function $\varphi$ is
finite everywhere on $\R^d$ : indeed, in this case any exponential function $\exp(\eps
|x|)$ is integrable with respect to the probability measure $\mu$ and using the Harnack
inequality, one can easily show that for any $n \geq 0$ and $x\in\Z^{\Lambda(M),d}_+$ 
\[
G_{\Lambda(M)}(x,x_n)/G_{\Lambda(M)}(x_0,x_n) ~\leq~ C \exp(\eps|x|) 
\]
with some $C>0$ and $\eps >0$ do not depending on $x$ and $n$.  
In our setting, the exponential functions $\exp(\eps |x|)$ are integrable only if $\eps
>0$ is small enough and hence, we need to get this inequality with a suitable small $\eps >
0$. For this we need to estimate hitting probabilities of the induced Markov chain
$(X_M(t))$. This is a subject of the following section.

\section{Hitting probabilities of the induced Markov chain}\label{sec4} 

Recall that the induced Markov chain $X_M(t) ~\dot=~ X^{\Lambda(M)}(t)$ corresponding to
the set $\Lambda(M) ~=~ \{ i\in\{1,\ldots,d\} : M^i = 0\}$ is a
substochastic random walk on $\Z_+^{\Lambda(M)} ~\dot=~ \{u=(u^i)_{i\in\Lambda(M)}\in\Z^{\Lambda(M)} : u_i > 0 \; \text{ for all } \;
i\in\Lambda(M)\}$  with
transition probabilities 
\[
p_{\Lambda(M)}(u,u') ~\dot=~ \mu_{\Lambda(M)}(u'-u) ~\dot=~ \sum_{x\in \Z^d:~ x^{\Lambda(M)} = u'-u} \mu(x),
\]
It is identical to the random walk $S^{\Lambda(M)}(t) =
(S^i(t))_{i\in\Lambda(M)}$ on $\Z^{\Lambda(M)}$ before  it first exits from $\Z_+^{\Lambda(M)}$. 
The mean jump of the random walk
$S^{\Lambda(M)}(t)$ is equal to zero because according to the definition of the set $\Lambda(M)$, 
\be\label{e4-1}
\sum_{u\in\Z^{\Lambda(M)}} \mu_{\Lambda(M)}(u) \,u ~=~ \sum_{x\in \Z^{d}:~
  x^{\Lambda(M)}}  \mu(x) \,x^{\Lambda(M)} ~=~ M^{\Lambda(M)} ~=~ 0. 
\ee
 The main result of this section is the following statement.  
\begin{prop}\label{pr4-1} Under the hypotheses (A1)-(A3), for any
  $\hat{u}\in\Z^{\Lambda(M)}_+$, 
\[
\lim_{u \in \Z^{\Lambda(M)}_+, \; |u|\to\infty} \frac{1}{|u|} \log \P_{\hat{u}}( X_M(t) = u \; \text{ for some } \; t
\geq 0) ~=~ 0.
\]
\end{prop}
\begin{proof} Since clearly $\P_{\hat{u}}\left(X_M(t) = u \; \text{ for some } \; t >
0\right) ~\leq~ 1$, to prove this proposition it is sufficient to show that 
for any ${\hat{u}}\in\Z^{\Lambda(M)}_+$
\be\label{e4-6}
\liminf_{u \in \Z^{\Lambda(M)}_+, \; |u|\to\infty} \frac{1}{|u|} \log \P_{\hat{u}}( X_M(t) = u \; \text{ for some } \; t
\geq 0) ~\geq~ 0.
\ee
Moreover, since the Markov process $(X_M(t))$ is irreducible on $\Z^{\Lambda(M)}_+$ it
  is sufficient to prove \eqref{e4-6} for $\hat{u}=(\hat{u}^i)_{i\in\Lambda(M)}$ with $\hat{u}^i=1$ for all
 $i\in\Lambda(M)$. The proof of this inequality depends on the number of elements
  $|\Lambda(M)|$ of the set $\Lambda(M)$ and is different in each of the following cases :
\begin{itemize}
\item[]{\em case 1:} $|\Lambda(M)| = 1$, 
\item[]{\em case 2:} $|\Lambda(M)| \geq 3$, 
\item[]{\em case 3:} $|\Lambda(M)| = 2$. 
\end{itemize} 
In the first case, $(S^{\Lambda(M)}(t))$ is a random walk on $\Z$ with zero mean (see
\eqref{e4-1}) and a finite variance because its jump generating function 
\[
  \varphi_{\Lambda(M)}(\a) ~=~ \sum_{u\in\Z^{\Lambda(M)}} \mu_{\Lambda(M)}(u) \exp(\a\cdot
  u), \quad \a\in\R^{\Lambda(M)} 
\]
is finite in a neighborhood of zero in $\R$. In this case, the random walk
$(S^{\Lambda(M)}(t))$ is therefore recurrent and for any $u\in\Z$, 
\be\label{e4-3}
\P_u(S^{\Lambda(M)}(t) = u + 1 \; \text{ for some } \; t > 0) ~=~ 1. 
\ee
Since the induced Markov chain $(X_M(t))$ is identical to $(S^{\Lambda(M)}(t))$ for $t <
\tau_{\Lambda(M)}$ and killed upon the time $\tau_{\Lambda(M)}$, from this it follows that for any $u > 0$,
\begin{multline*}
\P_{u+k}(X_M(t) = u + k + 1 \; \text{ for some } \; t > 0) ~=~ \\
\P_{u+k}(S^{\Lambda(M)}(t) = u + k + 1 \; \text{ for some } \; 0 < t < \tau_{\Lambda(M)}) ~\to~ 1 \; \text{ as } \; k\to\infty  
\end{multline*}
(for more details, see Lemma~7.2 of the paper ~\cite{Ignatiouk-Loree}). 
Hence, by strong Markov property applied for a sequence
of stopping times $(T_n)$ with $T_0 = 0$  and
\[
T_n ~=~ \inf\{ t > T_{n-1} :
S^{\Lambda(M)}(t) = u + n\}, \; \text{ for $n \geq 1$,}
\]
one gets 
\begin{align*}
\liminf_{n\to\infty} \frac{1}{n} \log \P_{\hat{u}}&(X_M(t) = \hat{u} + n \; \text{ for some } \; t > 0) \\&\geq~
\liminf_{n\to\infty}  \frac{1}{n}
\log ~\prod_{k=0}^{n-1} ~\P_{\hat{u}+k}(X_M(t) = \hat{u} + k + 1 \; \text{ for some } \; t > 0)  \\ &=~
\lim_{n\to\infty} ~\frac{1}{n}
~\sum_{k=0}^{n-1}~\log \P_{\hat{u}+k}(X_M(t) = \hat{u} + k + 1 \; \text{ for some } \; t > 0) \\&=~ 0
\end{align*}
where the last relation follows from \eqref{e4-3} by Cezaro's theorem. In the case when $|\Lambda(M)|
= 1$, the inequality \eqref{e4-6} is therefore verified. 

\medskip 
\noindent
{\em Case 2.} Suppose now that $|\Lambda(M)| \geq 3$. Here, to get \eqref{e4-6} we use large deviation
estimates of the Green function of the induced Markov chain $(X_M(t))$ 
\[
g_M(u,u') ~=~ \sum_{t=0}^\infty \P_u(X_M(t) = u'), \quad u,u'\in\Z^{\Lambda(M)}_+
\]
and the equality 
\[
g_M(\hat{u},u) ~=~ \P_{\hat{u}}( X_M(t) = u \; \text{ for some } \; t
\geq 0) \times g_M(u,u) 
\]
In this case, the random walk $(S^{\Lambda(M)}(t))$ is transient, the function
$g_M(u,u)$ is bounded above by the Green function of $(S^{\Lambda(M)}(t))$ : 
\begin{align}
g_M(u,u) &~=~ \sum_{t=0}^\infty \P_u(X_M(t) = u) ~=~ \sum_{t=0}^\infty
\P_u(S^{\Lambda(M)}(t) = u, \; \tau_{\Lambda(M)} > t) \nonumber\\&~\leq~ \sum_{t=0}^\infty
\P_u(S^{\Lambda(M)}(t) = u) ~=~ \sum_{t=0}^\infty \P_0(S^{\Lambda(M)}(t) = 0) ~<~ \infty, \label{e4-4p} 
\end{align}
and consequently, the left hand side of \eqref{e4-6} is greater or equal to 
\[
\liminf_{u \in \Z^{\Lambda(M)}_+, \; |u|\to\infty} \frac{1}{|u|} ~\log ~g_M(\hat{u},u). 
\]
To get \eqref{e4-6} it is therefore sufficient to show that for any sequence
$u_n\in\Z^{\Lambda(M)}_+$ with $\lim_n|u_n| = \infty$, 
\be\label{e4-7}
\liminf_n \frac{1}{|u_n|} \log g_M(\hat{u}, u_n)~\geq~ 0.
\ee
Moreover, since the set 
\[
{\cal S} ~\dot=~  \{u\in\R^{\Lambda(M)} : |u| = 1 \; \text{ and
 } \; u^i\geq 0, \; \text{ for all } \; i\in\Lambda(M)\}
\]
is compact, it is sufficient to prove that  \eqref{e4-7} holds when
$\lim_n u_n/|u_n| = e$, for each $e\in{\cal S}$. 
The proof of this inequality uses the large deviation estimates similar to that of  
Section~\ref{sec3}. 
Namely, denote 
\[
\R^{\Lambda(M)}_+~\dot=~
\{u\in\R^{\Lambda(M)} : u^i \geq 0, \; \forall i\in\Lambda(M)\}
\]
and let 
\[
  \varphi_{\Lambda(M)}(\a) ~=~ \sum_{u\in\Z^{\Lambda(M)}} \mu_{\Lambda(M)}(u) \exp(\a\cdot
  u), \quad \a\in\R^{\Lambda(M)}, 
\]
be the jump generating function of the random walk $(S^{\Lambda(M)}(t))$.
 Then for any $q\not=q'$, $q,q'\in\R^{\Lambda(M)}_+$ and
any sequences $u_n', u_n\in \Z_+^{\Lambda(M)}$ and $\eps_n >0$ with $\lim_n \eps_n = 0$,
$\lim_n \eps_nu'_n ~=~q'$ and $\lim_n \eps_nu_n ~=~q$, the same arguments as in the proof
of Proposition~\ref{pr3-2} prove that 
\be\label{e4-4}
\liminf_{n\to\infty} \eps_n \log g_M(u'_n,u_n) ~\geq~ - \sup_{\a\in\R^{\Lambda(M)} :
  \varphi_{\Lambda(M)}(\a)\leq 1} ~a\cdot (q-q'). 
\ee
Moreover, since the mean of the random walk $(S^{\Lambda(M)}(t))$ is zero (see \eqref{e4-1}), then
the set $\{\a\in\R^{\Lambda(M)} : \varphi_{\Lambda(M)}(\a)\leq 1\}$ contains an only
point zero and consequently, the right hand side of \eqref{e4-4} is equal to zero. Using
therefore \eqref{e4-4} with $u_n'=\hat{u}$, $\eps_n =1/|u_n|$, $q'=0$ and
$q=e\in{\cal S}$ one gets \eqref{e4-7}
for any sequence of points $u_n\in\Z^{\Lambda(M)}_+$ with $\lim_n|u_n|=\infty$ and $\lim_n
u_n/|u_n|= e$. 

\medskip
\noindent
{\em Case 3.} Consider now the case when $|\Lambda(M)|=2$. Remark that in this case, the same arguments as above
prove the inequality \eqref{e4-7} for any sequence $u_n\in\Z^{\Lambda(M)}_+$ with
$\lim_n|u_n|=\infty$. However, the right hand side of \eqref{e4-4p} is in this case
infinite because the random walk $(S^{\Lambda(M)}(t))$ is  recurrent and hence, in order to
get \eqref{e4-6} one should be more careful. To prove \eqref{e4-6} in this case we first
notice that for any unit vector $e\in\R^d$ with $e^i = 1$ for some $i\in\Lambda(M)$
and $e^j = 0$ for $j\not= i$, 
\begin{align*}
g_M(\hat{u} + n e, \hat{u} + n e)  &=~ \sum_{t=0}^\infty \P_{u + n e}( X_M(t) ~=~ \hat{u} + n e) \\ &=~
\sum_{t=0}^\infty \P_{\hat{u} + n
  e}\left( S^{\Lambda(M)}(t) = \hat{u} + n e, \; \tau(\Lambda(M)) > t\right) \\
&=~  \sum_{t=0}^\infty \P_{\hat{u} + n
  e}\left( S^{\Lambda(M)}(t) = \hat{u} + n e, \; \tau_j > t \; \text{ for all } \; j\in\Lambda(M)\right) \\
&\leq \sum_{t=0}^\infty \P_{\hat{u} + n
  e}\left( S^{\Lambda(M)}(t) = \hat{u} + n e,  \; \tau_j > t \text{ for } 
j\in\Lambda(M)\!\setminus\!\{i\}\right) \\
&\leq~ \sum_{t=0}^\infty \P_{\hat{u}}\left( S^j(t) = \hat{u}^j , \; \tau_j > t
\; \text{ for } \; j\in\Lambda(M)\setminus\{i\}\right). 
\end{align*}
When  $|\Lambda(M)|= 2$, the right
hand side of the above inequality is finite because for $j\in\Lambda(M)\setminus\{i\}$, the random walk
$(S^j(t))$ killed upon the time $\tau_j$ is transient, and consequently,
\begin{align*}
g_M(\hat{u}, \hat{u} + n e) &=~ 
\P_{\hat{u}}(X_M(t) = \hat{u} + n e \; \text{ for some } t > 0 )  \times g_M(\hat{u} + n e, \hat{u} + n e) \\
&\leq~ C \, \P_{\hat{u}}(X_M(t) = \hat{u} + n e \; \text{ for some } t > 0 ) 
\end{align*}
with some $C > 0$ do not depending on $n$. The last relation combined with \eqref{e4-7} proves that for any
unit vector $e\in\R^d$ with $e^i = 1$ for some $i\in\Lambda(M)$
and $e^j = 0$ for $j\not= i$, the following inequality holds 
\be\label{e4-2}
\liminf_{n\to\infty} \frac{1}{n} \log \P_{\hat{u}}\left(X_M(t) = \hat{u} + n e \; \text{ for some } \; t >
0\right) ~\geq~ 0. 
\ee
For $\Lambda(M) = \{i,j\}$ with $1\leq i < j\leq d$, consider now the unit vectors
$e_i=(e_i^i,e_i^j)$ and $e_j=(e_j^i,e_j^j)$ in $\Z^{\Lambda(M)}$ with $e_i^i = e^j_j = 1$ and $e_i^j = e_j^i = 0$. Then $
\hat{u} ~=~ e_i + e_j$
and for any $u\in\Z^{\Lambda(M)}_+$, 
\[
u ~=~ \hat{u} + n e_i + m e_j 
\]
with some non-negative integers $n,m$. To prove \eqref{e4-6} it is sufficient to
show that 
\be\label{e4-8}
\liminf ~\frac{1}{n+m} \log \P_{\hat{u}}\left(X_M(t) = \hat{u} + n e_i + m e_j\; \text{ for some } \; t >
0\right) ~\geq~ 0
\ee
when $n+m\to\infty$. The last relation is a consequence of \eqref{e4-2}. Indeed,  for any
$u\in\Z^{\Lambda(M)}_+$, one has 
\[
\P_{u}\bigl(X_M(t) = u + m e_j\; \text{ for some } \; t >
0\bigr) ~\geq~ \P_{\hat{u}}\bigl(X_M(t) = \hat{u} + m e_j\; \text{ for some } \; t >
0\bigr).
\]
When applied with $u=\hat{u} + n e_i$, the last inequality shows that  
\begin{align*}
\P_{\hat{u}}\bigl(X_M(t) = \hat{u} + n e_i + &m e_j\; \text{ for some } \; t >
0\bigr)  \\&\geq~ \P_{\hat{u}}\left(X_M(t) = \hat{u} + n e_i\; \text{ for some } \; t >
0\right) \\ &\hspace{1cm}\times \P_{\hat{u}+ n e_i}\left(X_M(t) = \hat{u} + n e_i + m e_j\; \text{ for some } \; t >
0\right)  \\&\geq~ \P_{\hat{u}}\left(X_M(t) = \hat{u} + n e_i\; \text{ for some } \; t >
0\right) \\ &\hspace{2.6cm}\times \P_{\hat{u}}\left(X_M(t) = \hat{u} + m e_j\; \text{ for some } \; t >
0\right).   
\end{align*}
The left hand side of \eqref{e4-8} is therefore greater or equal to 
\begin{multline*}
\liminf ~\frac{1}{n+m} \log \P_{\hat{u}}\left(X_M(t) = \hat{u}  + n e_i\; \text{ for some } \; t >
0\right) \\ + \liminf ~\frac{1}{n+m} \log \P_{\hat{u}}\left(X_M(t) = \hat{u}  + m e_j\; \text{ for some } \; t >
0\right) 
\end{multline*}
and consequently, using \eqref{e4-2} one gets \eqref{e4-8}.
\end{proof}

When combined with the Harnack inequality, the above proposition implies the following
statement.  

\begin{cor}\label{cor4-1} Let $f > 0$ be a harmonic function of the  induced Markov chain
  $(X_M(t))$. Then under the hypotheses (A1)-(A3), 
\be\label{e4-9}
\limsup_{|u|\to\infty} \frac{1}{|u|} \log f(u) ~\leq~ 0. 
\ee
\end{cor}
\begin{proof} Indeed, by the Harnack inequality (see~\cite{Woess}), for any
  $u,\hat{u}\in\Z^{\Lambda(M)}_+$, 
\[
f(u) ~\leq~ f(\hat{u})/\P_{\hat{u}}( X_M(t) = u \; \text{ for some } \; t \geq 0) 
\]
and hence, \eqref{e4-6} implies \eqref{e4-9}. 
\end{proof}

\section{Estimates of the local Martin kernel $G_{\Lambda(M)}(x,x_n)/G_{\Lambda(M)}(x_0,x_n)$}\label{sec6}
Throughout this section, to simplify the notation, we let $\Lambda(M) = \Lambda$. 
Our main result is here the following proposition. 

\begin{prop}\label{pr6-1} Suppose that the conditions (A1)-(A3) are satisfied and let  a sequence of points 
  $x_n\in\Z^{\Lambda,d}_+$ be such that  $\lim_n |x_n| = \infty$ and $\lim_n \,
  x_n/|x_n| = M/|M|$. Then for
  any $\hat{x}\in\Z^{\Lambda,d}_+$ and $\eps > 0$ there are $N>0$, $\delta > 0$ and  $C > 0$ such that 
\be\label{e6-1}
\1_{\{|x|\leq\delta|x_n|\}}G_{\Lambda}(x,x_n)/G_{\Lambda}(\hat{x},x_n) ~\leq~ C \exp(\eps|x|) 
\ee
for all $n\geq N$ and $x\in\Z_+^{\Lambda,d}$. 
\end{prop}

To prove this proposition we need the following lemmas. 

\begin{lemma}\label{lem6-1} Under the hypotheses of Proposition~\ref{pr6-1}, for any $\eps > 0$
  there are $N > 0$ and $\delta > 0$  such that  
\be\label{e6-2}
\exp(-\eps|w|) ~\leq~ G_{\Lambda}(x+w,x_n)/G_{\Lambda}(x,x_n) ~\leq~ \exp(\eps|w|)
\ee
for all $n > N$, $x\in\Z^{\Lambda,d}_+$ and $w\in\Z^d$  with  $w^{\Lambda} = 0$ satisfying
the inequality 
$\max\{|w|, |x|\} \leq \delta |x_n|$. 
\end{lemma}
\begin{proof} By  
  Proposition~\ref{pr5-2},   for any $\eps > 0$ there are $N > 
  0$ and $\delta > 0$ such that  
\be\label{e6-3}
\exp(-\eps/d) ~\leq~ G_{\Lambda}(x'+e,x_n)/G_{\Lambda}(x',x_n)  ~\leq~ \exp(\eps/d)   
\ee
for  all $n > N$ and $x'\in\Z^{\Lambda,d}_+$ satisfying the inequality $|x'| < \delta(d+1)|x_n|$, and any unit
vector $e\in\Z^d$ with $e^{\Lambda} = 0$. For $w\in\Z^d$ with
$w^{\Lambda} = 0$ and $|w| = 1$ the inequalities \eqref{e6-2} are therefore verified. To
get these inequalities for an arbitrary $w\in\Z^d$ with  $w^{\Lambda} =
0$, it is sufficient  to consider a sequence of unit vectors
$e_1,\ldots, e_k\in\Z^d$ with  $e_i^{\Lambda} = 0$ and $k\leq d |w|$
such that $e_1 + \ldots + e_k =w$ and to apply the inequalities \eqref{e6-3}  in 
\[
\frac{G_{\Lambda}(x+w,x_n)}{G_{\Lambda}(x,x_n)} ~=~ \frac{G_{\Lambda}(x+e_1,x_n)}{G_{\Lambda}(x,x_n)}
\prod_{i=1}^{k-1} \frac{G_{\Lambda}(x+e_1+ \cdots + e_{i+1},x_n)}{G_{\Lambda}(x+ e_1+ \cdots +
  e_i,x_n)}
\]
first with $x'=x$
and $e=e_1$ and next with $x' = x + e_1 + \cdots + e_i$ and $e = e_{i+1}$ for every
$i=1,\ldots,k-1$. The resulting estimates  
\[
\exp(- k\eps/d) ~\leq~ G_{\Lambda}(x+w,x_n)/G_{\Lambda}(x,x_n)  ~\leq~ \exp(k\eps/d)
\]
and the inequality $k\leq d |w|$ provide \eqref{e6-2}.
\end{proof}

\begin{lemma}\label{lem6-2} Under the hypotheses of Proposition~\ref{pr6-1}, for any $\eps
  > 0$ there are $\delta > 0$, $N> 0$ and a 
positive integer $k > 0$  such that for any  $n >N$, 
$x\in\Z^{\Lambda,d}_+$ and a unit vector
$e\in\Z^{\Lambda,d}_+$  with $e^{\Lambda^c}=0$, 
\be\label{e6-4}
\1_{\{|x| < \delta|x_n|\}} {G_\Lambda(x+k e,x_n)}/{G_\Lambda(x,x_n)}  ~\leq~
\exp(\eps k).  
\ee
\end{lemma}
\begin{proof}  To prove this lemma we combine  Lemma~\ref{lem6-1} with
  Proposition~\ref{pr4-1}. By Lemma~\ref{lem6-1}, for any $\sigma > 0$ there are
  $N(\sigma)>0$ and  $\delta(\sigma) > 0$ such that  
\be\label{e6-5}
G_{\Lambda}(\hat{x},x_n) ~\geq~
\exp(- \sigma|\hat{x}-\tilde{x}|) \,G_{\Lambda}(\tilde{x},x_n)
\ee
for all $n > N(\sigma)$ and $\hat{x},\tilde{x}\in\Z^{\Lambda,d}_+$ with  $\hat{x}^{\Lambda} = \tilde{x}^\Lambda$
and satisfying the inequality $\max\{|\tilde{x}|,|\hat{x}-\tilde{x}|\} \leq \delta(\sigma)
|z_n|$ (we use here the first inequality of \eqref{e6-2} with 
$x = \tilde{x}$ and $w = \hat{x} - \tilde{x}$). Furthermore, for a vector
$u\in\Z^{\Lambda,d}_+$ denote 
\[
T_{u} ~\dot=~ \inf\{ t > 0 : Z^\Lambda_\Lambda(t) ~=~ Z^\Lambda_\Lambda(0) + u^\Lambda\}. 
\] 
Then for any $R > 0$, 
\begin{align*}
G_\Lambda(x,x_n) &\geq~ \E_{x}\Bigl(G_\Lambda\bigl(Z_\Lambda(T_u), x_n\bigr); \; T_u <
\infty\Bigr) \\&\geq~ \E_{x}\Bigl(G_\Lambda\bigl(Z_\Lambda(T_u), x_n\bigr); \; T_u <
\infty, \; |Z_\Lambda(T_u) - x - u| \leq R\Bigr).
\end{align*}
The inequality \eqref{e6-5} applied for the right hand side of the last inequality with
 $\hat{x}=Z_\Lambda(T_u)$ and  $\tilde{x} = x + u$  shows that 
\begin{multline*}
{G_\Lambda(x,x_n)}/{G_\Lambda(x+u, x_n) } ~\geq~ \\ \E_{x}\bigl( \exp(-\sigma | Z_\Lambda(T_u) - x-u|); \;
T_u < \infty, \; |Z_\Lambda(T_u)-x -u| \leq R\bigr) 
\end{multline*}
for all  $n\geq N(\sigma)$ and $R > 0$ satisfying the inequalities  $0 < R \leq
\delta(\sigma)|x_n|$ and $|x+u| \leq \delta(\sigma)|x_n|$. Moreover, let $\hat{x}\in\Z^{\Lambda,d}_+$ be such that
$\hat{x}^i = 1$ for $i\in\Lambda$ and $\hat{x}^i=0$ for $i\in\Lambda^c$. Then the right
hand side of the above inequality is greater or equal to 
\[
\E_{\hat{x}}\bigl( \exp(-\sigma | Z_\Lambda(T_u) - \hat{x}-u|); \;
T_u < \infty, \; |Z_\Lambda(T_u)-\hat{x} -u| \leq R\bigr) 
\]
and consequently, 
\begin{align}
&{G_\Lambda(x,x_n)}{G_\Lambda(x+u, x_n) } \nonumber\\&\hspace{1.2cm}\geq~  \E_{\hat{x}}\bigl( \exp(-\sigma | Z_\Lambda(T_u) - \hat{x}-u|); \;
T_u < \infty, \; |Z_\Lambda(T_u)-\hat{x} -u| \leq R\bigr) \label{e6-6}
\end{align}
for all  $R>0$, $n\geq N(\sigma)$  and $x,u\in\Z^{\Lambda,d}_+$ satisfying the
inequalities $|x+u| \leq \delta(\sigma)|x_n|$ and $R \leq \delta(\sigma)|x_n|$. 
Remark now that   
 by monotone
convergence theorem, the right hand side of \eqref{e6-6} tends to  $\P_{\hat{x}}(T_u < \infty)$
as $\sigma\to 0$ and $R\to\infty$. Since clearly,  
\[
\P_{\hat{x}}(T_u < \infty) ~<~ 1 
\]
from this it follows  that for any $u\in\Z^{\Lambda,d}_+$, there are
$\sigma(u) > 0$ and $R(u) > 0$ such that 
\be\label{e6-7}
{G_\Lambda(x,x_n)}/{G_\Lambda(x+u,x_n)}  ~\geq~ (\P_{\hat{x}}(T_u < \infty) )^2.
\ee 
for all  $n \geq
N(\sigma(u))$ and $x\in\Z^{\Lambda,d}_+$  satisfying the inequalities $R(u) < \delta(\sigma(u)) |x_n|
$ and  $|x+u| \leq \delta(\sigma(u))|x_n|$. 

Recall finally that for $\Lambda = \Lambda(M)$, $Z^\Lambda_\Lambda(t) = X_M(t)$ is the induced Markov chain on
$\Z^\Lambda_+$
corresponding to the set $\Lambda$. Another equivalent definition of the stopping time $T_{u}$ is therefore the
following : 
\[
T_{u} = \inf\{ t > 0 : X_M(t) = X_M(0) + u^\Lambda\}.
\]
Hence, by 
Proposition~\ref{pr4-1}, 
\[
\frac{1}{k}\log \P_{\hat{x}}(T_{ke}~<~ \infty)  ~\to~ 0 \quad \text{ as } \quad k \to
\infty 
\]
and consequently,  for any $\eps > 0$ there is a positive integer $k > 0$ such that 
\[
\P_{\hat{x}}(T_{ke}~<~ \infty) ~\geq~ 
~\exp(-\eps k/2) 
\] 
for any unit vector $e\in\Z^{\Lambda,d}_+$. Letting therefore
$\sigma = \sigma(k e)$ and $R = R(k e)$ we conclude  that  for any  $\eps > 0$
there are $k>0$, $N = N(\sigma)> 0$ and $\delta = \delta(\sigma) > 0$  such that 
\[
{G_\Lambda(x,x_n)}/{G_\Lambda(x+ke,x_n)}  ~\geq~ \exp(-\eps k)
\]
for any unit vector $e\in\Z^{\Lambda,d}_+$ and all $n \geq
N(\sigma)$ and $x\in\Z^{\Lambda,d}_+$  satisfying the inequalities $R < \delta |x_n|$ and
$|x+ke| \leq \delta|x_n|$. Since $\lim_n |x_n| = +\infty$ the last statement 
proves Lemma~\ref{lem6-2}.  
\end{proof}

\begin{lemma}\label{lem6-3} Under the hypotheses of Proposition~\ref{pr6-1}, for any $\eps
  > 0$ there are $\delta > 0$, $N> 0$ and a 
positive integer $k > 0$  such that for any  $n >N$, 
$x\in\Z^{\Lambda,d}_+$  and  $u\in\N^d$  with $u^{\Lambda^c} =0$, 
\be\label{e6-8}
\1_{\{|x| + k |u|< \delta|x_n|\}} {G_\Lambda(x+k u,x_n)}/{G_\Lambda(x,x_n)}  ~\leq~
\exp(\eps k|u|).  
\ee
\end{lemma}
\begin{proof} Indeed, by Lemma~\ref{lem6-2}, for any $\eps
  > 0$ there are $\delta > 0$, $N> 0$ and a 
positive integer $k > 0$ such that 
\[
\1_{\{|x| < 2\delta|x_n|\}} {G_\Lambda(x+k e,x_n)}/{G_\Lambda(x,x_n)}  ~\leq~
\exp(\eps k/d )
\]
for any unit vector
$e\in\Z^{\Lambda,d}_+$ with $e^{\Lambda^c}=0$, $n >N$ and
$x\in\Z^{\Lambda,d}_+$. Using this
 inequality at the right hand side of 
\[
\1_{\{|\tilde{x}| + k m < 2\delta|x_n|\}} \frac{G_\Lambda(\tilde{x}+k m e,x_n)}{G_\Lambda(\tilde{x},x_n)}  \leq~
\prod_{j=1}^m \1_{\{|\tilde{x} + k (j-1) e|< 2\delta|x_n|\}} \frac{G_\Lambda(\tilde{x}+k j e,x_n)}{G_\Lambda(\tilde{x}+
  k(j-1) e,x_n)}  
\]
with $x = \tilde{x} + k(j-1) e$ for every $j=1,\ldots,m$ one gets   
\be\label{e6-9}
\1_{\{|\tilde{x}| + k m < 2\delta|x_n|\}} {G_\Lambda(\tilde{x}+k m e,x_n)}/{G_\Lambda(\tilde{x},x_n)}  ~\leq~
\exp(\eps k m/d )
\ee
for any unit vector
$e\in\Z^{\Lambda,d}_+$ with $e^{\Lambda^c}=0$, $n >N$, $m \geq 0$ and  
$\tilde{x}\in\Z^{\Lambda,d}_+$.  
Suppose  now that $\Lambda =
\{j_1,\ldots,j_{|\Lambda|}\}$ with $j_1 < \cdots < j_{|\Lambda|}$ and let $e_j=(e_j^1,\ldots,e_j^d)$ denote the unit
vector in $\Z^d$ with $e_j^j = 1$ and $e_j^i = 0$ for $i\not= j$. Then  for any $u\in\N^d$
with $u^{\Lambda^c} = 0$ one has 
\[
u ~=~ \sum_{j\in\Lambda} u^j e_j \quad \text{ with } \quad u^j\in\N \quad \text{ for all }
\quad j\in\Lambda  
\]
and hence, letting  for $l=0,\ldots, |\Lambda|$ 
\[
 m_i=u^{j_i} \quad \text{ and } \quad z_l ~=~ x + k \sum_{i \leq l} m_i e_{j_i} 
\]
we obtain $z_0=x$, \, $z_{|\Lambda|} = u + k u$, \, $z_l = z_{l-1} + m_le_{j_l}$ and 
\[
 |z_{l-1}| + k m_{l} \leq |x| + k \Bigl|\sum_{i \leq l} m_i e_{j_i} \Bigr| + k m_{l} \leq |x| + k |u| + k m_{l} \leq 2(
|x| + k |u|) 
\]
for all $ l=1,\ldots,|\Lambda|$. From this it follows  that    
\begin{align*}
\1_{\{|x| + k|u| < \delta|x_n|\}} &\frac{G_\Lambda(x+k u,x_n)}{G_\Lambda(x,x_n)} ~=~
\1_{\{|x| + k|u| < \delta|x_n|\}}   \prod_{l=1}^{|\Lambda|} \frac{G_\Lambda(z_{l-1}+k
  m_{l}e_{j_{l}},x_n)}{G_\Lambda(z_{l-1},x_n)} \\
&\hspace{1.5cm} \leq~ \prod_{l=1}^{|\Lambda|} \1_{\{|z_{l-1}| + k
  m_{l} < 2\delta|x_n|\}}\frac{G_\Lambda(z_{l-1} +k
  m_{l}e_{j_l},x_n)}{G_\Lambda(z_{l-1},x_n)} 
\end{align*}
and consequently, using \eqref{e6-9} with $\tilde{x} = z_{l-1}$, $m=m_l$ and $e=e_{j_l}$  for
every $l=1,\ldots,|\Lambda|$ 
we conclude that 
\[
\1_{\{|x| + k|u| < \delta|x_n|\}} \frac{G_\Lambda(x+k
  u,x_n)}{G_\Lambda(x,x_n)} ~\leq~ \exp\left(\eps k \sum_{l=1}^{|\Lambda|}m_l
/d\right) ~\leq~ \exp\left(\eps k |u|\right).
\]
for all $n >N$. 
\end{proof}

When combined with Lemma~\ref{lem3-1}, Lemma~\ref{lem6-3} implies the following
estimates. 

\begin{lemma}\label{lem6-4} Under the hypotheses of Proposition~\ref{pr6-1}, for any $\eps
  > 0$ there are $\delta > 0$, $N> 0$ and $C>0$ such that for any  $n >N$,
  $x\in\Z^{\Lambda,d}_+$ and $u\in\N^d$ 
 with $u^{\Lambda^c} =0$, 
\be\label{e6-10}
\1_{\{|x| + |u|< \delta|x_n|\}} {G_\Lambda(x+
  u,x_n)}/{G_\Lambda(x,x_n)}  ~\leq~ C
\exp(\eps |u|).  
\ee
\end{lemma}
\begin{proof} Indeed, Lemma~\ref{lem6-3} proves that for any $\eps
  > 0$ there are $\delta > 0$, $N> 0$ and $k\geq 1$ such that for any  $n >N$, 
$x\in\Z^{\Lambda,d}_+$  and $\tilde{u}\in k\N^d$ with $\tilde{u}^{\Lambda^c} =0$, one has 
\be\label{e6-11}
\1_{\{|x| + |\tilde{u}|< 2\delta|x_n|\}} {G_\Lambda(x + \tilde{u}, x_n)}/{G_\Lambda(x, x_n)}  ~\leq~  
\exp(\eps |\tilde{u}|).
\ee 
Remark now that for any $u\in \N^d$ with $u^{\Lambda^c} = 0$ there is $\tilde{u}\in
  k\N^d$ with $\tilde{u}^{\Lambda^c}=0$ such that 
\be\label{e6-12}
\tilde{u}^i  < u^i \leq
  \tilde{u}^i + k\quad \text{ for all $i\in\Lambda$}
\ee
and by the Harnack inequality~(see~\cite{Woess}), 
\begin{multline*}
{G_\Lambda(x+
  u,x_n)}/{G_\Lambda(x,x_n)}  ~=~ \frac{G_\Lambda(x+
  u,x_n)}{G_\Lambda(x+\tilde{u},x_n)} \times\frac{G_\Lambda(x+
  \tilde{u},x_n)}{G_\Lambda(x,x_n)}\\
~\leq~ \frac{1}{\P_{x+\tilde{u}}( Z_\Lambda(t)
    = x + u\; \text{ for some } t \geq 0)} \times\frac{G_\Lambda(x+
  \tilde{u},x_n)}{G_\Lambda(x,x_n)}. 
\end{multline*}
Moreover since by Lemma~\ref {lem3-1}, the random walk $(Z_\Lambda(t))$ satisfies
the communication condition on $\Z^{\Lambda,d}_+$, then  there is $0< \theta <1$ such that 
\[
\P_{x+\tilde{u}}( Z_\Lambda(t)
    = x + u \; \text{ for some } t \geq 0) ~\geq~ \theta^{|u-\tilde{u}|} ~\geq~ \theta^{k d}
\]
and consequently, 
\[
{G_\Lambda(x+
  u,x_n)}/{G_\Lambda(x,x_n)} ~\leq~ \theta^{- k d} {G_\Lambda(x+
  \tilde{u},x_n)}/{G_\Lambda(x,x_n)}. 
\]
The last inequality combined with \eqref{e6-11} and \eqref{e6-12} proves \eqref{e6-10} for $n>N$
large enough with $C = \exp(\eps k d) \theta^{- k d}$. 
\end{proof}

\bigskip

\begin{proof}[Proof of Proposition~\ref{pr6-1}] Since the Markov chain $(Z_\Lambda(t))$ is
  irreducible on $\Z^{\Lambda,d}_+$ it is sufficient to prove this proposition for
  $\hat{x}\in\Z^{\Lambda,d}_+$ with $\hat{x}^i = 1$ for $i\in\Lambda$ and $\hat{x}^i = 0$
  for $i\in\Lambda^c$. For this we combine
  Lemma~\ref{lem6-1} with Lemma~\ref{lem6-4}.  For a given $\hat{x}\in\Z^{\Lambda,d}_+$, 
Lemma~\ref{lem6-1} proves that for any $\eps > 0$ there $N_1(\eps) > 0$ and $\delta_1(\eps) > 0$  such that  
\be\label{e6-13}
G_{\Lambda}(\hat{x}+w,x_n)/G_{\Lambda}(\hat{x},x_n) ~\leq~ \exp\left(\eps|w|\right)
\ee
for all $n > N_1(\eps)$ and $\hat{x},w\in\Z^{\Lambda,d}_+$ with  $w^{\Lambda} = 0$ such that 
$|w| \leq \delta_1(\eps) |x_n|$. While by  Lemma~\ref{lem6-4}, 
for any $\eps > 0$ there are $\delta_2(\eps) > 0$, $N_2(\eps)> 0$ and $C(\eps) >0$ such that 
\be\label{e6-14}
{G_\Lambda(\tilde{x}+
  u,x_n)}/{G_\Lambda(\tilde{x},x_n)}  ~\leq~ C(\eps) 
\exp(\eps |u|).  
\ee
for any  $n >N_2(\eps)$,
  $x\in\Z^{\Lambda,d}_+$ and $u\in\N^d$ 
 with $u^{\Lambda^c} =0$ satisfying the inequality $|\tilde{x}| + |u|< \delta_2(\eps)|x_n|$.
Letting for $x\in\Z^{\Lambda,d}_+$,
\[
w ~=~ \sum_{i\in\Lambda^c} x^i e_i \quad \text{ and } \quad u ~=~ \sum_{i\in\Lambda} (x^i-1) e_i
\]
one gets $u\in\N^d$ with $u^{\Lambda^c} = 0$ and $w\in\Z^{\Lambda,d}_+$ with $w^\Lambda =
0$ such that 
\[
x = \hat{x} + w + u \quad \text{ and } \quad \max\{|u|, |w|\} \leq |x-\hat{x}| \leq |x| + d.
\]
Hence, using the inequalities \eqref{e6-13} and  \eqref{e6-14} with $\tilde{x} = \hat{x} + w$ we obtain 
\begin{align*}
\frac{G_{\Lambda}(x,x_n)}{G_{\Lambda}(\hat{x},x_n)} &~=~
\frac{G_{\Lambda}(x,x_n)}{G_{\Lambda}(\tilde{x},x_n)} \times
\frac{G_{\Lambda}(\tilde{x},x_n)}{G_{\Lambda}(\hat{x},x_n)} ~=~ \frac{G_{\Lambda}(\tilde{x}+u,x_n)}{G_{\Lambda}(\tilde{x},x_n)} \times
\frac{G_{\Lambda}(\hat{x}+ w ,x_n)}{G_{\Lambda}(\hat{x},x_n)} \\
&~\leq~  C(\eps) \exp(\eps|w| + \eps|u|) ~\leq~  C(\eps) \exp(2\eps|x| + 2 d) 
\end{align*}
for all $n > \max\{N_1(\eps), N_2(\eps)\}$ and $x\in\Z^{\Lambda,d}_+$ satisfying the inequality 
\[
|x| + d\leq \delta_1(\eps) |x_n| \quad \text{ and } \quad 
2|x| + 3 d< \delta_2(\eps)|x_n|.  
\]
Since $\lim_n |x_n| = \infty$, the last relations prove \eqref{e6-1} with $C=C(\eps/2)\exp(
2 d)$ and $0 < \delta <
\min\{\delta_1(\eps/2), \delta_2(\eps/2)/2\}$ for   $n$ large enough. 

\end{proof}

\section{Limiting behavior of the local Martin kernel}\label{sec7} The estimates 
obtained in the previous section are now combined with  the results of Section~\ref{sec4}
in order to
investigate the limiting behavior of the Martin kernel
$G_{\Lambda(M)}(x,x_n)/G_{\Lambda(M)}(x_0,x_n)$ when $|x_n|\to\infty$ and $x_n/|x_n| \to
M/|M|$.  To simplify the notations, we denote throughout this section $\Lambda =
\Lambda(M)$.  
Our main result is here the following proposition. 

\begin{prop}\label{pr7-1} Suppose that the conditions (A1) - (A3) are satisfied
  and  let a sequence of points $x_n\in\Z^{\Lambda,d}_+$ with $\lim_n |x_n| = +\infty$ and $\lim_n
  x_n/|x_n| = M/|M|$ be fundamental for the Markov process $(Z_\Lambda(t))$. Then there is
  a harmonic function $f > 0$ of the induced Markov chain $(X_{M}(t))$ such that  for any $x\in\Z^{\Lambda,d}_+$, 
\be\label{e7-1}
\lim_{n\to\infty} 
G_{\Lambda}(x,x_n)/G_{\Lambda}(x_0,x_n) ~=~ f(x^{\Lambda})/f(x_0^{\Lambda}). 
\ee  
\end{prop}

As a consequence of this result one gets 

\begin{cor}\label{cor7-1} 
Suppose that the conditions (A1) - (A3) are satisfied
  and  let the only non-negative harmonic functions of the induced Markov chain $(X_M(t))$
  be the  constant multiples of $f>0$. Then any sequence  $x_n\in\Z^{\Lambda,d}_+$ with $\lim_n |x_n| = +\infty$ and $\lim_n
  x_n/|x_n| = M/|M|$ is fundamental for the Markov process $(Z_\Lambda(t))$ and the
  equality \eqref{e7-1} holds with a given function $f$. 
\end{cor}
\begin{proof} Indeed, if a sequence of points $x_n\in\Z^{\Lambda,d}_+$ is such that $\lim_n |x_n| = +\infty$ and $\lim_n
  x_n/|x_n| = M/|M|$, then  by Proposition~\ref{pr7-1}, for any fundamental subsequence
  $(x_{n_k})$, the sequence of functions $G_\Lambda(\cdot,x_{n_k})/G_\Lambda(x_0,x_{n_k})$
  converges point-wise to the same limit $f(\cdot)/f(x_0)$. By
  compactness, the sequence
  $(x_n)$ is therefore fundamental itself and satisfies the equality \eqref{e7-1}.
\end{proof}

The proof of Proposition~\ref{pr7-1} uses  the following preliminary results. 

\begin{lemma}\label{lem7-1} Under the hypotheses (A2), 
\be\label{e7-3}
\limsup_{r\to\infty} ~\frac{1}{r}~\log \!\sum_{x\in\Z^d : |x|\geq  r} \mu(x) ~<~ 0.
\ee
\end{lemma}
\begin{proof} Indeed, under the hypotheses (A2), there is $\theta > 0$ such that 
\[
C_\theta ~\dot=~ \sup_{a\in\R^d : |a|\leq \theta} \varphi(a) ~<~ \infty
\]
and hence, for any unit vector $e\in\Z^d$,
\[
C_\theta\geq \varphi(\theta e) ~\geq \sum_{x :~e\cdot x ~\geq~  r} \mu(x)
\exp(\theta e\cdot x) ~\geq~ \exp( \theta r ) \!\sum_{x:~e\cdot x ~\geq~  r}
\mu(x) 
\]
From the last inequality it follows that 
\[
\sum_{x\in\Z^d : |x|\geq  r} \mu(x) ~\leq~ 2d \max_{e\in\Z^d : |e|=1} \sum_{x\in\Z^d : ~e\cdot x ~\geq~  r}
\mu(x) ~\leq~ 2d C_\theta \exp( -\theta r ) 
\]
and consequently, \eqref{e7-3} holds. 
\end{proof}

Using this lemma together with Corollary~\ref{cor3-2} we obtain 

\begin{lemma}\label{lem7-2} Suppose that the conditions (A1)-(A3) are satisfied and let a
  sequence $x_n\in\Z^{\Lambda,d}_+$ be such that  $\lim_n |x_n| = +\infty$ and $\lim_n
  x_n/|x_n| = M/|M|$. Then for any $\delta > 0$ and $x\in\Z^{\Lambda, d}_+$, 
\be\label{e7-5}
\lim_{n\to\infty} ~\frac{1}{G_{\Lambda}(x,x_n)} \sum_{z\in\Z^{\Lambda,d}_+: |z| \geq \delta |x_n|} \mu(z-x)
G_{\Lambda}(z,x_n) ~=~ 0.  
\ee
\end{lemma} 
\begin{proof} Indeed,  for any $n\in\N$ and $z\in\Z^{\Lambda,d}_+$, 
\begin{align*}
G_{\Lambda}(z,x_n) &~\leq~ G_{\Lambda}(x_n,x_n) ~=~ \sum_{t=0}^\infty
\P_{x_n}(S(t)=x_n, \, \tau_{\Lambda} > t) \\&~\leq~ \sum_{t=0}^\infty
\P_{x_n}(S(t)=x_n) ~=~ \sum_{t=0}^\infty
\P_0(S(t)=0)  ~<~ \infty, 
\end{align*}
and hence, by Lemma~\ref{lem7-1}, 
\begin{multline*}
\limsup_{n\to\infty} ~\frac{1}{|x_n|} \log \sum_{z\in\Z^{\Lambda,d}_+: |z| \geq \delta |x_n|} \mu(z-x)
G_{\Lambda(M)}(z,x_n) \\ \leq~ \limsup_{n\to\infty} ~\frac{1}{|x_n|} \log
\sum_{z\in\Z^{\Lambda,d}_+: |z| \geq \delta |x_n|} \mu(z-x) ~<~ 0. 
\end{multline*}
The last inequality proves \eqref{e7-5} because by Corollary~\ref{cor3-2}, 
\[
\limsup_{n\to\infty} ~\frac{1}{|x_n|} \log G_{\Lambda(M)}(x,x_n) ~=~ 0.
\]
\end{proof}

Lemma~\ref{lem7-2} combined with Proposition~\ref{pr6-1} implies the following property. 

\begin{lemma}\label{lem7-3} Suppose that the conditions (A1)-(A3) are satisfied and let a
  sequence of points 
  $x_n\in\Z^{\Lambda,d}_+$ with  $\lim_n |x_n| = +\infty$ and $\lim_n
  x_n/|x_n| = M/|M|$ be fundamental for the Markov process $(Z_\Lambda(t))$. Then  the limit 
\be\label{e7-6}
h(x) ~=~ \lim_{n\to\infty} G_{\Lambda}(x,x_n)/G_{\Lambda}(x_0,x_n) 
\ee
is a harmonic function of  $(Z_\Lambda(t))$. 
\end{lemma}
\begin{proof} To prove this statement one has to show that 
\[
\lim_{n\to\infty} 
{G_{\Lambda}(x,x_n)}/{G_\Lambda(x_0,x_n)} \\ ~=~  \sum_{z\in\Z^{\Lambda,d}_+} \mu(z-x)
\lim_{n\to\infty} {G_{\Lambda}(z,x_n)}/{G_\Lambda(x_0,x_n)} 
\]
Since for $x\not=x_n$,
\[
G_{\Lambda}(x,x_n) ~=~ \sum_{z\in\Z^{\Lambda,d}_+} \mu(z-x)
G_{\Lambda}(z,x_n) 
\]
and by Lemma~\ref{lem7-2}, for any $\delta > 0$, 
\[
\lim_{n\to\infty} ~\frac{1}{G_{\Lambda}(x,x_n)} \sum_{z\in\Z^{\Lambda,d}_+: |z| > \delta |x_n|} \mu(z-x)
G_{\Lambda}(z,x_n) ~=~ 0,
\]
it is sufficient to show that for some $\delta > 0$, 
\be\label{e7-7}
\lim_{n\to\infty} \sum_{z\in\Z^{\Lambda,d}_+: |z| \leq \delta |x_n|} \mu(z-x)
\frac{G_{\Lambda}(z,x_n)}{G_\Lambda(x_0,x_n)} \\ ~=~  \sum_{z\in\Z^{\Lambda,d}_+} \mu(z-x)
\lim_{n\to\infty} \frac{G_{\Lambda}(z,x_n)}{G_\Lambda(x_0,x_n)}  
\ee
The proof of the last relation uses Proposition~\ref{pr6-1} and the dominated convergence theorem. By dominated convergence theorem,
\eqref{e7-7} holds if there exists  a positive and $\mu$-integrable  function $C(z)$ on $\Z^{\Lambda,d}_+$ such
that 
\be\label{e7-9}
\1_{|z| \leq \delta |x_n|} \frac{G_{\Lambda}(z,x_n)}{G_\Lambda(x_0,x_n)}  ~\leq~ C(z)
\quad \text{ for all $z\in\Z^{\Lambda,d}_+$. }
\ee
Because of the assumption (A2), an exponential function $C(z) = C\exp(\eps|z|)$ is
$\mu$-integrable for any $C>0$ if $\eps > 0$ is small
enough, and Proposition~\ref{pr6-1} proves that for any $\eps > 0$ there are $C>0$
and $\delta > 0$ for which \eqref{e7-9} also  holds with  $C(z) = C
\exp(\eps|z|)$. 
\end{proof}

\begin{proof} [Proof of Proposition~\ref{pr7-1} ] This proposition is a consequence of   
Lemma~\ref{lem7-3} and Corollary~\ref{pr5-1}. Indeed, let a sequence of points 
$x_n\in\Z^{\Lambda,d}_+$ with  $\lim_n |x_n| = +\infty$ and $\lim_n
  x_n/|x_n| = M/|M|$ be  fundamental  for the Markov chain $(Z_{\Lambda}(t))$ and let 
\be\label{e7-10} 
h(x) ~=~ \lim_{n\to\infty} {G_{\Lambda}(x,x_{n})}/{G_\Lambda(x_0,x_{n})}. 
\ee 
Then by irreducibility, 
\[
h(x) ~=~ \lim_{n\to\infty} \frac{G_{\Lambda}(x,x_{n})}{G_\Lambda(x_0,x_{n})} ~\geq~
\P_x(Z_\Lambda(t) = x_0, \; \text{ for some } \; t \geq 0) ~>~ 0 
\]
and by Lemma~\ref{lem7-3}, the function $h$  harmonic for $(Z_\Lambda(t))$. Moreover, by Corollary~\ref{pr5-1}, 
$h(x+w) ~=~ h(x)$ for all $x\in\Z^{\Lambda,d}_+$ and $w\in\Z^d$ with $w^\Lambda
  = 0$
and consequently, the function $h(x)=h(x^1,\ldots,x^n)$ does not depend on $x^i$ for  
$i\not\in\Lambda$. Letting therefore 
\[
f(x^\Lambda) ~\dot=~ h(x) 
\]
one gets a function $f>0$ on $\Z^\Lambda_+$ satisfying \eqref{e7-1} with
$f(x_0^\Lambda) = h(x_0) = 1$ and such that 
\[
\E_{x^\Lambda}(f(X_M(1))) ~=~ \E_x(f(Z^\Lambda_\Lambda(1))) ~=~ \E_x(h(Z_\Lambda(1))) ~=~
h(x) ~=~ f(x^\Lambda). 
\]
The last relation shows that the function $f$ is harmonic for the induced Markov chain
$(X_M(t))$. 
\end{proof}

\section{From local to the original process : proof of Theorem~\ref{th1}}\label{sec8}
The results of the previous sections are now used to obtain the asymptotic behavior of
the Martin kernel of the original Markov chain $(Z(t))$.

\subsection{Principal part of the renewal equation}
The first result of this section proves that for a sequence of points $x_n\in\Z_+^d$ with
$\lim_n|x_n|=\infty$ and $\lim_n x_n/|x_n| = M/|M|$ the right hand side of the renewal
equation \eqref{e2-1} with $x'=x_n$ and $\Lambda=\Lambda(M)$ 
can be decomposed into the main part 
\be\label{e8-2}
\Xi_{\delta,\Lambda}(x,x_n) ~=~ G_\Lambda(x,x_n)  ~-~ \E_x( G_\Lambda(S(\tau),x_n), \; \tau
< \tau_\Lambda, \; |S(\tau)| < \delta |x_n|) 
\ee
and the corresponding negligible part 
\[
\Xi_{\delta,\Lambda}(x,x_n) - G(x,x_n) ~=~ \E_x( G_\Lambda(S(\tau),x_n), \; \tau
< \tau_\Lambda, \; |S(\tau)| \geq \delta |x_n|)  
\]
with an arbitrary $\delta > 0$. This is a subject of the following proposition. 

\begin{prop}\label{pr8-1} Suppose that the conditions (A1)-(A3) are satisfied and let the
  coordinates of the mean $M$ be non-negative.  Suppose moreover that a sequence of points $x_n\in
  \Z_+^{d}$ is such that $\lim_n |x_n| = \infty$ and $\lim_n 
  x_n/|x_n| = M/|M|$. Then for $\Lambda=\Lambda(M)$  and any $\delta > 0$,
\[
\lim_{n\to\infty}~ \Xi_{\delta,\Lambda}(x,x_n)/G(x,x_n) ~=~ 1
\]
\end{prop}
The proof of this proposition uses the following lemmas. 

\begin{lemma}\label{lem8-1} Under the hypotheses (A1)-(A3), there is $C > 0$ such that for any
  $\Lambda\subset\{1,\ldots,d\}$, $a\in D$ and
  $x,x'\in\Z^{\Lambda,d}_+$, 
\be\label{eq8-3}
G_\Lambda(x,x') ~\leq~ C \exp(a\cdot (x-x')). 
\ee
\end{lemma}
\begin{proof} Indeed, for any $a\in D$, the exponential function $
f(x) = \exp(a\cdot x)$ 
is super-harmonic for the random walk $(S(t))$ and hence, by the Harnack inequality
(see~\cite{Woess}),  
\[
\P_x(S(t)=x' \; \text{ for some } \; t \geq 0) ~\leq~ \exp(a\cdot(x-x')), \quad \forall
x,x'\in\Z^d. 
\]  
the Green function 
\[
G_S(x,x') ~=~ \sum_{t=0}^\infty \P_x(S(t)=x') 
\]
of the random walk $(S(t))$ satisfies therefore the inequality 
\begin{align*}
G_S(x,x') &~=~ G_S(x',x') \P_x(S(t)=x' \; \text{ for some } \; t \geq 0) \\&~\leq~ G_S(x',x')
\exp(a\cdot(x-x')) ~=~ G_S(0,0)\exp(a\cdot(x-x')).  
\end{align*}
Since clearly, $G_\Lambda(x,x')\leq G_S(x,x')$ for all $x,x'\in\Z^{\Lambda,d}_+$, the last
relation proves \eqref{eq8-3} with $C=G_S(0,0)$. 
\end{proof}

Recall that for a non-zero vector $q\in\R^d$, the only
  point in $D$ where the linear function $a\to a\cdot q$ achieves its maximum over the set
  $D$ is denoted  $a(q)$.  For $q=0$ it is convenient to let $a(0) = 0\in D$. Then 
\[
\max_{a\in D} a\cdot q ~=~ a(q)\cdot q 
\]
for any $q\in\R^d$ and from Lemma~\ref{lem7-1} it follows 
\begin{cor}\label{cor8-1}Under the hypotheses (A1)-(A3), there is $C > 0$ such that for any
  $\Lambda\subset\{1,\ldots,d\}$ and
  $x,x'\in\Z^{\Lambda,d}_+$, 
\[
G_\Lambda(x,x') ~\leq~ C \exp(-a(x'-x)\cdot (x'-x)). 
\]
\end{cor}

\begin{lemma}\label{lem8-2} Suppose that the conditions (A1)-(A3) are satisfied and let a
  sequence $x_n\in\Z^d$ be such that $\lim_n |x_n| = +\infty$. Then there is $R>0$ such that  for any $x\in\Z^d$, 
\be\label{e8-4}
\limsup_{n\to\infty} \frac{1}{|x_n|} \log \sum_{z\in\Z^d: |z| \geq R |x_n|}
\exp\bigl( - a(z-x)\cdot (z-x) - a(x_n-z)\cdot(x_n-z)\bigr) ~<~ 0
\ee
\end{lemma}
\begin{proof} Indeed, for any  $n\in\N$
and $x,z\in\Z^d$  
\begin{align*}
a(z-x)\!\cdot\!(z-x) + a(x_n-z)\!\cdot\!(x_n-z) &= \sup_{a\in D} a\cdot (z-x) + \sup_{a\in D}
a\cdot(x_n-z) \\
&\geq~ a(z)\cdot(z-x) + a(-z)\cdot(x_n-z) \\&\geq a(z)\!\cdot\!z + a(-z)\!\cdot\!(-z) - (|x| +
|x_n|) \sup_{a\in D} |a|
\end{align*}
and consequently,  the left hand side of \eqref{e8-4} does not
exceed
\begin{align*}
\limsup_{n\to\infty} \frac{1}{|x_n|} \log \sum_{z\in\Z^d: |z| \geq R |x_n|}
\exp\Bigl( - a(z)\cdot z - a(-z)\!\cdot\!(-z) + (|x| +
|x_n|) \sup_{a\in D} |a| \Bigr) \\ =~ \sup_{a\in D} |a| +  \limsup_{n\to\infty} \frac{1}{|x_n|} \log \sum_{z\in\Z^d: |z| \geq R |x_n|}
\exp\bigl( - a(z)\cdot z - a(-z)\!\cdot\!(-z)  \bigr). 
\end{align*}
To prove Lemma~\ref{lem8-2}, it is therefore sufficient to show that for $R>0$ large enough,
\be\label{e8-5}
\limsup_{n\to\infty} \frac{1}{|x_n|} \log \sum_{z\in\Z^d: |z| \geq R |x_n|}
\exp\bigl( - a(z)\cdot z - a(-z)\cdot(-z)  \bigr) ~<~ -\sup_{a\in D} |a|. 
\ee
For this we have to investigate the function 
\[
 \lambda(q) ~\dot=~ a(q)\cdot q + a(-q)\cdot(-q) ~=~ \sup_{a\in D} a\cdot q + \sup_{a\in D} a\cdot (-q).
\] 
As a supremum of a collection of convex
functions $q\to a\cdot q - a'\cdot q$ over a compact set $(a,a')\in D\times D$, this
function is finite, convex and therefore continuous on $\R^d$. 
Moreover, recall that under the hypotheses (A1)-(A3), the mapping $q\to a(q)$ determines a homeomorphism from the unit sphere
${\cal S}^d ~\dot=~ \{q\in\R^d: |q|=1\}$ to the boundary $\partial D$ of the set $D$ and
for any non-zero $q\in\R^d$, the point $a(q) ~\dot=~ a(q/|q|)$ is the only point in $D$
where the supremum of the linear function $a\to a\cdot q$ over $a\in D$ is
attained. Hence, for any $q\not=0$, one has $a(-q) ~\not=~ a(q)$, 
\[
\lambda(q) ~=~ |q| \lambda(q/|q|), 
\]
and 
\[
\lambda(q) ~=~ a(q)\cdot q + \sup_{a\in D} a\cdot (-q) ~>~ a(q)\cdot q + a(q)\cdot(-q) ~=~
0,
\]
from which it follows that 
\[
\lambda(q) ~\geq~ \lambda^*|q|
\]
with 
\[
\lambda^* ~\dot=~ \inf_{q\in\R^d: |q| =1} a(q)\cdot q + a(-q)\cdot(-q) ~>~ 0.   
\]
The last relations show that the left hand side of \eqref{e8-5} does not exceed 
\begin{align*}
\limsup_{n\to\infty} \frac{1}{|x_n|} &\log \sum_{z\in\Z^d: |z| \geq R |x_n|}
\exp\bigl( - \lambda^*|z| \bigr) \\&\leq~ \limsup_{n\to\infty} \frac{1}{|x_n|} \log
~\left(\exp(-\lambda^* R|x_n|/2) \!\sum_{z\in\Z^d: |z| \geq R |x_n|}
\exp\bigl( - \lambda^*|z|/2 \bigr) \right)\\ &\leq~ - \lambda^* R/2  + \limsup_{n\to\infty} \frac{1}{|x_n|} \log
~\sum_{z\in\Z^d: |z| \geq R |x_n|}
\exp\bigl( - \lambda^*|z|/2 \bigr) 
 \end{align*}
where 
\begin{align*}
\limsup_{n\to\infty} \frac{1}{|x_n|} \log
~\sum_{z\in\Z^d: |z| \geq R |x_n|}
\exp\bigl( - \lambda^*|z|/2 \bigr)  &\leq~ \limsup_{n\to\infty} \frac{1}{|x_n|} \log \sum_{z\in\Z^d}
\exp\bigl( - \lambda^*|z|/2 \bigr) \\&\leq~ 0 
\end{align*}
because for $\lambda^* > 0$, the series 
\[
\sum_{z\in\Z^d}
\exp\bigl( - \lambda^*|z|/2 \bigr) 
\]
converge. The inequality \eqref{e8-5} holds therefore for $
R > 2 \sup_{a\in D} |a|/\lambda^*$. 
\end{proof}

\begin{lemma}\label{lem8-3} Suppose that the conditions (A1)-(A3) are satisfied and let a 
compact set  $V\subset\R^d$ be such that $V\cap \{c M : c \geq 0\}  =\emptyset$. Then for
any $x\in\Z^d$ and any sequence $x_n\in\Z^d$ with $\lim_n |x_n| = +\infty$ and $\lim_n
  x_n/|x_n| = M/|M|$,  one has    
\be\label{e8-6}
\limsup_{n\to\infty} \frac{1}{|x_n|} \log \sum_{\substack{z\in\Z^d: z\in|x_n| V}}
\exp\bigl( - a(z-x)\cdot (z-x) - a(x_n-z)\cdot(x_n-z)\bigr) ~<~ 0
\ee
\end{lemma}
\begin{proof} Indeed, let a compact set  $V\subset\R^d$ be such that $V\cap \{c M : c \geq
  0\}  =\emptyset$. Then letting $c=\sup_{a\in D}|a|$, for any $x\in\Z^d$, $n\in\N$ and 
  $z\in\Z^d\cap(|x_n|V)$   one gets 
\begin{align*}
\sup_{a\in D} a\cdot(x_n-z) &~=~ \sup_{a\in D} \left(a\cdot\left(x_n- \frac{|x_n|}{|M|}
M\right) + a\cdot \left(\frac{|x_n|}{|M|}
M - z \right)\right) \\ 
&~\geq~ - c \left|x_n- \frac{|x_n|}{|M|}
M\right|  +  \sup_{a\in D} a\cdot \left(\frac{|x_n|}{|M|}
M - z \right) 
\end{align*}
and 
\[
\sup_{a\in D} a\cdot (z-x) ~\geq~ a(z)\cdot(z-x) ~\geq~ \sup_{a\in D} a\cdot  z
- c |x| 
\]
from which it follows that 
\begin{align*}
a(z-x)\cdot &(z-x) + a(x_n-z)\cdot(x_n-z) ~=~ \sup_{a\in D} a\cdot (z-x) + \sup_{a\in D}
a\cdot(x_n-z)\\
&\geq~ \sup_{a\in D} a\cdot z + \sup_{a\in D} a\cdot \left(\frac{|x_n|}{|M|}
M - z \right)  - c|x| - c\left|x_n- \frac{|x_n|}{|M|}
M\right|   
\end{align*} 
and consequently,
\begin{align}
\limsup_{n\to\infty}&\frac{1}{|x_n|} \log \sum_{\substack{z\in\Z^d: z\in|x_n|V}}
\exp\bigl( - a(z-x)\cdot (z-x) - a(x_n-z)\cdot(x_n-z)\bigr) \nonumber\\&\leq 
\limsup_{n\to\infty} \frac{1}{|x_n|} \log \sum_{\substack{z\in\Z^d: z\in|x_n|V}}
\exp\left( - \sup_{a\in D} a\cdot z - \sup_{a\in D} a\cdot \left(\frac{|x_n|}{|M|}
M - z \right) \right) \nonumber\\&\hspace{6cm} + c\lim_{n\to\infty} \left(\frac{|x|}{|x_n|} +
\left|\frac{x_n}{|x_n|}- \frac{M}{|M|}\right| \right)\nonumber
\\&\leq 
\limsup_{n\to\infty} \frac{1}{|x_n|} \log \sum_{\substack{z\in\Z^d: z\in |x_n| V}}
\exp\left( - \sup_{a\in D} a\cdot z - \sup_{a\in D} a\cdot \left(\frac{|x_n|}{|M|}
M - z \right) \right) \label{e8-7}
\end{align}  
where the last relation holds because $|x_n|\to \infty$ and $x_n/|x_n|\to M/|M|$ as
$n\to\infty$.  Moreover, letting 
\[
\lambda_M(q) ~=~ \sup_{a\in D} a\cdot q + \sup_{a\in D} a\cdot \left(\frac{M}{|M|} - q \right)
\]
one gets  
\begin{align*}
\sup_{a\in D} a\cdot z + \sup_{a\in D} a\cdot \left(\frac{|x_n|}{|M|}
M - z \right) &~=~ |x_n|\left(\sup_{a\in D} a\cdot \frac{z}{|x_n|} + \sup_{a\in D} a\cdot
\left(\frac{M}{|M|} - \frac{z}{|x_n|} \right)  \right)\\
&~=~ |x_n| \lambda_M\left(\frac{z}{|x_n|} \right) ~\geq~ |x_n| \inf_{q\in V} \lambda_M(q) 
\end{align*}
for all $z\in\Z^d\cap(|x_n|V)$. The last inequality combined with \eqref{e8-7} shows
that the left hand side of \eqref{e8-6} does not exceed 
\[
- \inf_{q\in V} \lambda_M(q) + \limsup_{n\to\infty} \frac{1}{|x_n|} \log \text{Card}(\{z\in\Z^d: z \in |x_n|V\}). 
\]
Since for a compact set $V\subset\R^d$, the number of points $\text{Card}(\{z\in\Z^d:
z \in rV\})$ of the set $\{z\in\Z^d: z \in rV\}$ tends to infinity polynomially
with respect to $r$ as $r\to\infty$, we conclude  that 
\[
\limsup_{n\to\infty} \frac{1}{|x_n|} \log \text{Card}(\{z\in\Z^d: z \in |x_n|V\}) = 0
\]
and 
\begin{align} 
\limsup_{n\to\infty} \frac{1}{|x_n|} \log\sum_{\substack{z\in\Z^d: z\in |x_n|V}}
\exp\bigl( - a(z-x)\cdot (z-x) - a(x_n-z)\cdot(x_n-z)\bigr) \nonumber\\\leq - \inf_{q\in V}
\lambda_M(q). \label{e8-8}
\end{align}
To complete the proof of \eqref{e8-6} it is now sufficient to show that 
\be\label{e8-?}
\inf_{q\in V}\lambda_M(q) ~>~ 0. 
\ee
For this we investigate the function $\lambda_M(\cdot)$. As a supremum of a collection of
convex functions 
\[
q ~\to~ a\cdot q + a'\cdot \left(\frac{M}{|M|} - q\right)
\] 
over the compact set $(a,a')\in D\times D$, this function is finite, convex and therefore
continuous on $\R^d$. Moreover, recall  that the mapping $q\to a(q)$ determines a homeomorphism from the unit sphere
${\cal S}^d ~\dot=~ \{q\in\R^d: |q|=1\}$ to the boundary $\partial D$ of the set $D$ and
for any non-zero $q\in\R^d$, the point $a(q) ~\dot=~ a(q/|q|)$ is the only point in $D$
where the supremum of the linear function $a\to a\cdot q$ over $a\in D$ is
attained. Since $V\cap \{c M : c \geq 0\}  =\emptyset$, from this it follows that for 
any $q\in V$, 
\[
a(M)\not= a(q), \quad a(M)\not= a\left(\frac{M}{|M|} - q
\right)  
\]
and 
\begin{align*}
\lambda_M(q) &~=~ \sup_{a\in D} a\cdot q + \sup_{a\in D} a\cdot \left(\frac{M}{|M|} - q
\right) \\&~>~ a\left(M\right) \cdot q +
a\left(M\right)\cdot\left(\frac{M}{|M|} - q\right) ~=~ a(M)\cdot \frac{M}{|M|}.  
\end{align*} 
 Since according to the definition of the mapping $q\to a(q)$ (see Section~\ref{sec1})
\[
a(M) = a(\nabla\varphi(0)) = 0
\]
this proves that $\lambda_M(q) > 0$ for all $q\in V$ and consequently \eqref{e8-?} holds. 
The  inequality \eqref{e8-?} combined with \eqref{e8-8} provides \eqref{e8-6}.
\end{proof}

\begin{proof}[Proof of Proposition~\ref{pr8-1}] Remark first of all that by Corollary~\ref{cor8-1}, for any $\delta > 0$,
\begin{multline*}
0 < \Xi_{\delta,\Lambda}(x,x_n) - G(x,x_n) ~=~ \sum_{w\in\Z^{\Lambda,d}_+\setminus \Z^d_+: |w| \geq \delta |x_n|} \P_x( S(\tau) =
w, \; \tau < \tau_\Lambda) G_\Lambda(w,x_n)\\\leq~ \sum_{w\in\Z^{\Lambda,d}_+\setminus
  \Z^d_+: |w| \geq \delta |x_n|}C^2 \exp\bigl( - a(z-x)\cdot (z-x) -
a(x_n-z)\cdot(x_n-z)\bigr) 
\end{multline*}
and recall that by Proposition~\ref{pr6-1},
\[
\liminf_{n\to\infty} \frac{1}{|x_n|} \log G(x,x_n) = 0.
\]
To prove Proposition~\ref{pr8-1} it is therefore sufficient to show that  
\be\label{e8-9p}
\limsup_{n\to\infty} \frac{1}{|x_n|} \log\sum_{\substack{z\in\Z^{\Lambda,d}_+\setminus\Z^d_+: \\|z| \geq \delta |x_n|}}
\exp\bigl( - a(z-x)\cdot (z-x) -
a(x_n-z)\cdot(x_n-z)\bigr)  ~<~ 0. 
\ee
For this we use 
Lemma~\ref{lem8-2} with $R > 0$ large enough, Lemma~\ref{lem8-3} with a
compact set 
\[
V = \{q\in\R^d : \delta \leq |q|\leq R \; \text{ and } q^i \leq 0 \; \text{ for all } \;
i\in\Lambda^c(M)\}
\]
and Lemma~1.2.15 of Dembo and Zeitouni~\cite{D-Z}. 
Indeed, denote for $0 < \delta < R$, 
\[
\Sigma^1_{\delta,R} (x,x_n) ~\dot= \sum_{z\in\Z^{\Lambda,d}_+\setminus\Z^d_+: ~\delta |x_n| \leq |z| \leq R|x_n|}
\hspace{-0.3cm}\exp\bigl( - a(z-x)\cdot (z-x) - a(x_n-z)\cdot(x_n-z)\bigr)
\]
and let 
\[
\Sigma^2_{R} (x,x_n) ~\dot= \sum_{z\in\Z^d: |z| > R|x_n|}
\exp\bigl( - a(z-x)\cdot (z-x) - a(x_n-z)\cdot(x_n-z)\bigr).
\]
Then by  Lemma~1.2.15 of Dembo and Zeitouni~\cite{D-Z},   the left hand side of \eqref{e8-9p} is equal to 
\begin{multline*}
\limsup_{n\to\infty} \frac{1}{|x_n|} \log \left(\Sigma^1_{\delta,R} (x, x_n) + \Sigma^2_{R}
(x, x_n) \right)\\ ~=~ \max\left\{\limsup_{n\to\infty} \frac{1}{|x_n|} \log \Sigma^1_{\delta,R}
(x, x_n), \;\limsup_{n\to\infty} \frac{1}{|x_n|} \log \Sigma^2_{R}
(x, x_n)\right\},
\end{multline*}
where by Lemma~\ref{lem8-3}, for any $0 < \delta < R$, 
\[
\limsup_{n\to\infty} \frac{1}{|x_n|} \log \Sigma^2_{\delta, R}
(x, x_n) ~<~ 0 
\]
and by Lemma~\ref{lem8-2}, 
\[
\limsup_{n\to\infty} \frac{1}{|x_n|} \log \Sigma^2_{R}(x,x_n) ~<~ 0 
\]
if $R>0$ is large enough. Using these relations with  $R> 0$ large
enough and $0 < \delta < R$ one gets therefore \eqref{e8-9p}.
\end{proof} 

\subsection{Generating functions of hitting probabilities} 
\begin{prop}\label{pr8-2} Suppose that the conditions (A1)-(A3) are satisfied, the
  coordinates of the mean vector $M$ are non-negative and let
  $\Lambda = \Lambda(M)$. Then there is $\eps_0>0$ such that for any $0 < \eps < \eps_0$, 
\[
\E_x(\exp(\eps|S(\tau)|), \; \tau < \tau_\Lambda) ~<~ \infty, \quad \forall \;
x\in\Z_+^d, 
\]
and for any $x^\Lambda\in\Z^\Lambda_+$, 
\[
\E_x(\exp(\eps|S^\Lambda(\tau)|), \; \tau < \tau_\Lambda) ~\to~ 0 \quad \text{ as } \quad
\min_{i\in\Lambda^c}|x^i|\to\infty. 
\]
 \end{prop}

The proof of this proposition uses the following lemmas

\begin{lemma}\label{lem8-4} Under the hypotheses (A1)-(A3), for any $a\in D$ and
  $x\in\Z^d_+$, 
\[
\E_x(\exp(a\cdot S(\tau)), \; \tau < \tau_\Lambda)  ~\leq~ \exp(a\cdot x)
\] 
\end{lemma} 
\begin{proof}
To get this inequality it is sufficient to notice that for any $a\in D$, the exponential
function $a\to \exp(a\cdot x)$ is super-harmonic for the Random walk $(S(t))$ and the
quantity 
\[
\E_x(\exp(a\cdot S(\tau)), \; \tau < \tau_\Lambda)\times \exp(-a\cdot x)
\] 
is equal to the probability that the twisted substochastic random walk $(S_a(t))$ having transition
probabilities $\tilde{p}(x,x') ~=~ p(x,x')\exp(a\cdot(a'-a))$ exits from $\Z^d_+$ before
the first time when at least one of its coordinates $(S_a^i(t))$ with $i\in\Lambda$
becomes negative or zero. 
\end{proof}

Let $e_i=(e_i^1,\cdots,e_i^d)$ denote the unit vector in $\R^d$ with $e_i^i=1$ and
$e_i^j=0$ for $j\not= i$. 

\begin{lemma}\label{lem8-5} Under the hypotheses of Proposition~\ref{pr8-2}, 
  for any $\Lambda'\subset\{1,\ldots,d\}$ such that
  $\Lambda(M)\subset\Lambda'\not=\{1,\ldots,d\}$,  there are $\delta > 0$ and $\sigma > 0$
  for which the points 
\[
\tilde{a}_{\Lambda'} = -\delta \sum_{i\in\Lambda'^c} M^i e_i + \sigma \sum_{i\in\Lambda(M)}
e_i\quad \text{ and } \quad \hat{a}_{\Lambda'} ~=~ -\delta \sum_{i\in\Lambda'^c} M^i e_i + \sigma \sum_{i\in\Lambda'} e_i
\]
belong to the set  $D$. 
\end{lemma}
\begin{proof} Indeed, recall that $\Lambda(M) ~\dot=~ \{i\in\{1,\ldots,d\} : M^i =
  0\}$. Hence, for $i\not\in \Lambda(M)$, under the hypotheses of Proposition~\ref{pr8-2}, 
\[
M^i ~\dot=~ \bigl(\nabla\varphi(0)\bigr)^i ~\dot=~ \left.\frac{\partial}{\partial
  a^i}\varphi(a) \right|_{a=0} ~>~0
\]
and consequently, for some $\delta > 0$ small enough and any $\Lambda'\subset\{1,\ldots,d\}$
such that $\Lambda(M)\subset\Lambda'\not=\{1,\ldots,d\}$, 
\[
\varphi\Bigl( - \delta\sum_{i\in\Lambda'^c} M^ie_i\Bigr) ~<~ \varphi(0) ~=~ 1
\]
 The point 
\[
a_{\Lambda'} ~\dot=~ - \delta\sum_{i\in\Lambda'^c} M^ie_i 
\]
belongs therefore to the interior of the set $D$ and consequently there is $\sigma >0$
for which the points 
\[
\tilde{a}_{\Lambda'} = a_{\Lambda'} + \sigma \sum_{i\in\Lambda(M)}
e_i \quad \text{ and } \quad \hat{a}_{\Lambda'} ~=~ a_{\Lambda'} + \sigma\sum_{i\in\Lambda'} e_i
\]
also belong to the interior of the set $D$. 
\end{proof}

\begin{proof}[Proof of Proposition~\ref{pr8-2}] Recall that $\tau_i$ denotes the first
  time when the $i$-th coordinate of the random walk $(S(t))$ becomes negative or zero, 
\[
\tau ~\dot=~ \min_{i=1,\ldots, d} \tau_i,   \quad \quad \text{ and } \quad \quad \tau_{\Lambda}
~\dot=~ \min_{i\in\Lambda} \tau_i 
\]
Hence, on the event $\{\tau < \tau_\Lambda\}$, one has 
$S^i(\tau) > 0$ for all $i\in\Lambda$ and $S^i(\tau) \leq 0$ for some $i\not\in\Lambda$.
Letting   $\Lambda_+(x) ~=~ \{i\in\{1,\ldots, d\} : x^i > 0\}$,  
we get  therefore 
\be\label{eq8-10}
\E_x(\exp(\eps|S(\tau)|), \, \tau < \tau_\Lambda) ~= \sum_{\substack{\Lambda' : ~\Lambda \subset \Lambda'\\\Lambda'\not=\{1,\ldots,d\}}} \E_x\bigl(
\exp(\eps|S(\tau)|), \, \Lambda_+(S(\tau)) = \Lambda', \, \tau <
\tau_\Lambda\bigr).
\ee
Furthermore,  by Lemma~\ref{lem8-5}, for $\Lambda=\Lambda(M)$ and any $\Lambda'\subset\{1,\ldots,d\}$ such that
$\Lambda\subset\Lambda'\not=\{1,\ldots,d\}$ there are $\delta > 0$ and $\sigma > 0$ for
which 
\[
\hat{a}_{\Lambda'} ~\dot=~ -\delta \sum_{i\in\Lambda'^c} M^i e_i + \sigma \sum_{i\in\Lambda'}
e_i ~\in~ D.
\]
For such a point $\hat{a}_{\Lambda'}=(\hat{a}_{\Lambda'}^1,\ldots,\hat{a}_{\Lambda'}^d)$, on the event $\{\Lambda_+(S(\tau)) = \Lambda'\}$, 
\begin{align*}
\hat{a}_{\Lambda'}\cdot S(\tau) &~=~ -\delta \sum_{i\in\Lambda'^c} M^i S^i(\tau) + \sigma \sum_{i\in\Lambda'}
S^i(\tau) \\&~=~  \delta \sum_{i\in\Lambda'^c} M^i |S^i(\tau)| + \sigma \sum_{i\in\Lambda'}
|S^i(\tau)| ~\geq~  |S(\tau)|  \min\{\sigma, \delta\min_{i\in\Lambda^c} M^i\}.  
\end{align*}
Using this inequality at the right hand side of \eqref{eq8-10} with  
\[
0 < \eps < \min\{\sigma, \delta\min_{i\in\Lambda^c} M^i\}
\]
we obtain 
\[
\E_x(\exp(\eps|S(\tau)|), \; \tau < \tau_\Lambda) \leq \sum_{\substack{\Lambda' :~ \Lambda \subset \Lambda'\\\Lambda'\not=\{1,\ldots,d\}}} \!\E_x\left(
\exp(\hat{a}_{\Lambda'}\cdot S(\tau)), \; \Lambda_+(S(\tau)) = \Lambda', \; \tau <
\tau_\Lambda\right) 
\]
where by Lemma~\ref{lem8-4},
\[
\E_x\left(
\exp(\hat{a}_{\Lambda'}\cdot S(\tau)), \; \Lambda_+(S(\tau)) = \Lambda', \; \tau <
\tau_\Lambda\right)  ~\leq~ \exp(\hat{a}_{\Lambda'}\cdot x)
\]
and consequently, 
\[
\E_x(\exp(\eps|S(\tau)|), \; \tau < \tau_\Lambda)  ~\leq~ \sum_{\substack{\Lambda' :~
    \Lambda \subset \Lambda'\\\Lambda'\not=\{1,\ldots,d\}}}  \exp(\hat{a}_{\Lambda'}\cdot
x) ~<~ \infty.  
\]
The first assertion of Proposition~\ref{pr8-2} is therefore proved. To prove the second
assertion of this proposition we use again Lemmas~\ref{lem8-4} and ~\ref{lem8-5}  but with the points
\[
\tilde{a}_{\Lambda'} = -\delta \sum_{i\in\Lambda'^c} M^i e_i + \sigma \sum_{i\in\Lambda(M)}
e_i.
\]
The same arguments as above shows that on the event $\{\Lambda\subset \Lambda_+(S(\tau)) = \Lambda'\}$, 
\begin{align*}
\tilde{a}_{\Lambda'}\cdot S(\tau) &~=~ -\delta \sum_{i\in\Lambda'^c} M^i S^i(\tau) + \sigma \sum_{i\in\Lambda}
S^i(\tau) \\&~=~  \delta \sum_{i\in\Lambda'^c} M^i |S^i(\tau)| + \sigma \sum_{i\in\Lambda}
|S^i(\tau)| ~\geq~ \sigma \sum_{i\in\Lambda}
|S^i(\tau)|  ~=~ \sigma |S^\Lambda(\tau)|  
\end{align*}
and consequently, for $0 < \eps \leq \sigma$, 
\begin{align*}
\E_x(\exp(\eps|S^\Lambda(\tau)|), \; \tau < \tau_\Lambda)  &\leq \!\sum_{\substack{\Lambda'
    :~ \Lambda \subset \Lambda'\\\Lambda'\not=\{1,\ldots,d\}}} \!\E_x\left( 
\exp(\tilde{a}_{\Lambda'}\cdot S(\tau)), \, \Lambda_+(S(\tau)) = \Lambda', \, \tau < 
\tau_\Lambda\right) \\  
&\leq~ \sum_{\substack{\Lambda' :~
    \Lambda \subset \Lambda'\\\Lambda'\not=\{1,\ldots,d\}}}  \exp(\tilde{a}_{\Lambda'}\cdot
x) ~<~ \infty.  
\end{align*}
Since for any $x^\Lambda\in\Z_+^{\Lambda}$, and $\Lambda'\subset\{1,\ldots,d\}$ such that
$\Lambda \subset \Lambda'\not=\{1,\ldots, d\}$,
\[
\tilde{a}_{\Lambda'}\cdot x ~=~ - \delta \sum_{i\in\Lambda'^c} M^i x^i + \sigma
\sum_{i\in\Lambda} x^i ~\to~ - \infty \quad \text{ as } \quad \min_{i\in\Lambda^c} x^i
\to \infty, 
\]
the last inequality proves the second
assertion of Proposition~\ref{lem8-2}. 
\end{proof}

\subsection{Harmonic functions}

When combined with Corollary~\ref{cor4-1}, Proposition~\ref{pr8-2} implies the following
statement.

\begin{prop}\label{pr8-3} Suppose that the conditions (A1)-(A3) are satisfied, the
  coordinates of the vector $M$ are non-negative and let $\Lambda=\Lambda(M)$. 
  Then for any  harmonic function $f>0$ of the induced Markov chain $(X_M(t))$, the
  function 
\be\label{e8-9}
h(x) ~=~ f(x^\Lambda) -
\E_x\bigl(f(S^\Lambda(\tau)), \; \tau < \tau_\Lambda)\bigr) 
\ee
is finite, strictly positive and harmonic for the Markov chain $(Z(t))$. 
\end{prop}
\begin{proof} If $f>0$ is a harmonic function of the induced Markov chain
  $(X_M(t))$ then the function $x\to f(x^\Lambda)$ is harmonic for the local Markov
  process $Z_\Lambda(t)=Z_{\Lambda(M)}(t) = (X_M(t),Y_M(t))$ because according to the definition of the induced Markov chain
  $(X_M(t))$ (see Section~\ref{sec2}), 
\[
X_M(t) ~=~ Z_\Lambda^\Lambda(t), \quad \forall t\geq 0. 
\]
Since the local Markov chain $Z_\Lambda(t)$ is identical to the random walk $S(t)$ for $t <
\tau_\Lambda$ and is killed upon the time $\tau_\Lambda$, from this it follows that 
\[
f(x^\Lambda) ~\geq~ \E_x(f(Z_\Lambda^\Lambda(\tau_{\Lambda^c}))) ~=~
\E_x(f(S^\Lambda(\tau_{\Lambda^c})), \; \tau_{\Lambda^c} < \tau_\Lambda) ~=~
\E_x(f(S^\Lambda(\tau)), \; \tau < \tau_\Lambda) 
\]
where the last equality holds because $\tau ~=~ \min\{\tau_\Lambda,
\tau_{\Lambda^c}\}$. The function \eqref{e8-9} is therefore finite and
non-negative. Furthermore, for any $x\in\Z^d_+$, using again the fact that the function
$x\to f(x^\Lambda)$ is harmonic for $(Z_\Lambda(t))$, one gets 
\begin{align*}
f(x^\Lambda) &~=~ \E_x(f(Z^\Lambda_\Lambda(1))) ~=~ \E_x(f(Z^\Lambda_\Lambda(1)), \; \tau =
1) + \E_x(f(Z^\Lambda_\Lambda(1)), \; \tau > 1)\\
&~=~  \E_x(f(S^\Lambda(1)), \; \tau =
1 < \tau_\Lambda) + \E_x(f(S^\Lambda(1)),  \; \tau > 1).
\end{align*}
Since moreover, 
\[
\E_x\bigl(f(S^\Lambda(\tau)), \; \tau < \tau_\Lambda\bigr) ~=~
\E_x\bigl(f(S^\Lambda(\tau)), \; \tau = 1 < \tau_\Lambda\bigr) +
\E_x\bigl(f(S^\Lambda(\tau)), \; 1 <\tau < \tau_\Lambda\bigr)
\]
then for any $x\in\Z^d_+$, 
\begin{align*}
h(x) &~=~ f(x^\Lambda) ~-~ \E_x\bigl(f(S^\Lambda(\tau)), \; \tau < \tau_\Lambda\bigr) \\
&~=~ \E_x(f(S^\Lambda(1)),  \; \tau > 1) - \E_x\bigl(f(S^\Lambda(\tau)), \; 1 <\tau <
\tau_\Lambda\bigr)\\
&~=~ \sum_{w\in\Z^d_+} \mu(w-x) f(w^\Lambda) - \sum_{w\in\Z^d_+}
\mu(w-x)\E_w\bigl(f(S^\Lambda(\tau)), \; \tau < \tau_\Lambda\bigr) \\
&~=~ \sum_{w\in\Z^d_+} \mu(w-x) h(w)
\end{align*}
and consequently, the function $h$ is harmonic for the Markov chain $(Z(t))$. Now, to complete
the proof of this proposition we have to show that the function $h$ is strictly
positive on $\Z^d_+$. For this we use Propositions~\ref{pr8-2} and
Corollary~\ref{cor4-1}. Recall that by Corollary~\ref{cor4-1}, 
\[
\limsup_{|u|\to\infty} \frac{1}{|u|} \log f(u) ~\leq~ 0. 
\]
Hence, for any $\eps > 0$ there is $C > 0$ such that 
\[
0 < f(u) ~\leq~ C\exp(\eps|u|), \quad \forall u\in\Z^\Lambda_+ 
\]
and consequently, for any $x\in\Z^d_+$, 
\[
h(x) ~\geq~ f(x^\Lambda) - C ~\E_x\bigl(\exp(\eps|S^\Lambda(\tau)|), \; \tau < \tau_\Lambda\bigr).
\]
Since by Proposition~\ref{pr8-2}, for $\eps > 0$ small enough and any
$x^\Lambda\in\Z_+^\Lambda$, 
\[
\E_x\bigl(\exp(\eps|S^\Lambda(\tau)|), \; \tau < \tau_\Lambda\bigr) ~\to~ 0 \quad \text{
  as } \; \min_{i\in\Lambda} x^i\to +\infty,
\]
from this it follows that there is $\tilde{x}\in\Z^d_+$ ( with a large $\min_{i\in\Lambda} \tilde{x}^i$) for
which $h(\tilde{x}) > 0$.  
Using finally the Harnack inequality 
\[
h(x) ~\geq~ h(\tilde{x}) ~\P_x(Z(t)=\tilde{x} \quad \text{ for some } \; t\geq 0) 
\]
and the fact that under the hypotheses (A2), the Markov chain $(Z(t))$ is
irreducible on $\Z^d_+$, we conclude that $h(x) > 0$ for all $x\in\Z^d_+$. 
\end{proof}

\subsection{Proof of Theorem~\ref{th1}} We are ready now to complete the proof of
Theorem~\ref{th1}. Indeed, suppose that the conditions (A1)-(A3) are satisfied and the
coordinates of the mean vector $M$ are non-negative. Then by Proposition~\ref{pr8-3}, for
any  harmonic function $f>0$ of the induced Markov chain $(X_M(t))$, the 
  function 
\[
h(x) ~=~ f(x^\Lambda) -
\E_x\bigl(f(S^\Lambda(\tau)), \; \tau < \tau_\Lambda)\bigr) 
\]
with $\Lambda=\Lambda(M)~\dot=~ \{i : M^i = 0\}$ is finite, strictly positive and harmonic for the Markov chain $(Z(t))$. 
Hence, the first assertion of Theorem~\ref{th1} is already proved.  

Suppose now that the sequence of points 
$z_n\in\Z^d_+$ with $\lim_n |z_n| = \infty$ and $\lim_n z_n/|z_n| = M/|M|$ is fundamental
for the local random walk $(Z_{\Lambda}(t))$ with 
$\Lambda = \Lambda(M)~\dot=~ \{i : M^i = 0\}$. Then by Proposition~\ref{pr7-1}, there is a
harmonic function $f>0$ of the induce Markov chain $(X_M(t))$ such that 
\be\label{e8-10}
\lim_{n\to\infty} G_{\Lambda}(x,z_n)/G_{\Lambda}(x_0,z_n) ~=~
f(x^{\Lambda})/f(x_0^{\Lambda}), \quad \forall x\in\Z^{\Lambda,d}_+
\ee
That was the first step of our proof. Now, we use the renewal equation \eqref{e2-1} 
to prove that for such a sequence $(z_n)$, for any $x\in\Z_+^d$, 
\be\label{e8-11}
\lim_{n\to\infty} G(x,z_n)/G_\Lambda(x_0,z_n) ~=~ \frac{1}{f(x_0^\Lambda)} \Bigl(f(x^\Lambda) -
\E_x\bigl(f(S^\Lambda(\tau), \; \tau < \tau_\Lambda)\bigr)\Bigr). 
\ee
By Proposition~\ref{pr8-1}, the right hand side of the above renewal equation can be
decomposed into a main part 
\[
\Xi_{\delta,\Lambda}(x,z_n) ~\dot=~ G_\Lambda(x,z_n) - \E_x( G_\Lambda(S(\tau),z_n), \;
\tau < \tau_\Lambda, \; |S(\tau)| < \delta|z_n|) 
\]
and the corresponding negligible part $
\Xi_{\delta,\Lambda}(x,z_n) - G(x,z_n)$ 
so that for any $\delta > 0$ and $x\in\Z^d_+$, 
\[
\lim_{n\to\infty} {\Xi_{\delta,\Lambda}(x,z_n)}/{G(x,z_n)} ~=~ 1.
\]
From this it follows that 
\begin{align}
\lim_{n\to\infty} &{G(x,z_n)}/{G_{\Lambda}(x_0,z_n)} ~=~ 
\lim_{n\to\infty} {\Xi_{\delta,\Lambda}(x,z_n)}/{G_{\Lambda}(x_0,z_n)} \nonumber\\&=~
\lim_{n\to\infty}  \left(\frac{G_\Lambda(x,z_n)}{G_{\Lambda}(x_0,z_n)}  - \E_x\left(
\frac{G_\Lambda(S(\tau) ,z_n)}{G_\Lambda(x_0,z_n)}, \; \tau < \tau_\Lambda, \; |S(\tau)| < \delta|z_n|\right)  \right)\label{e8-12}
\end{align}
 whenever the last limit exists. By Proposition~\ref{pr6-1}, for any $\eps > 0$
there are $N>0$, $C>0$ and $\delta >0$ such that   for any $n \geq N$,
\be\label{e8-13}
\1_{\{|w| <\delta |z_n|\}} {G_\Lambda(w,z_n)}/{G_\Lambda(x_0,z_n)} ~\leq~ C\exp(\eps|w|), \quad
\forall w\in\Z^{\Lambda,d}_+. 
\ee
Since  by Proposition~~\ref{pr8-1},  for $\eps>0$ small enough,
\[
\E_x( \exp(\eps|S(\tau)|), \; \tau < \tau_\Lambda)   ~<~ \infty,
\]
using dominated convergence theorem
from \eqref{e8-10} and \eqref{e8-13} it follows that for some $\delta > 0$, the limit
at the right hand side of \eqref{e8-12} exists and is equal to the right hand side of \eqref{e8-11} 
Relation \eqref{e8-11} is therefore proved. Since by Proposition~\ref{pr8-3}, the right
hand side of \eqref{e8-11} is non-zero, we conclude that 
\be\label{e8-14}
\lim_{n\to\infty} \frac{G(x,z_n)}{G(x_0,z_n)} ~=~ \frac{f(x^\Lambda) -
\E_x\bigl(f(S^\Lambda(\tau), \; \tau < \tau_\Lambda)\bigr)}{f(x_0^\Lambda) -
\E_{x_0}\bigl(f(S^\Lambda(\tau), \; \tau < \tau_\Lambda)\bigr)},\quad \forall x\in\Z^d_+. 
\ee
Any fundamental sequence $z_n\in\Z_+^d$ of the local random walk $(Z_{\Lambda(M)}(t))$,
with $\lim_n|z_n|=\infty$ and $\lim_n z_n/|z_n| = M/|M|$, is
therefore fundamental for the random walk $(Z(t))$. 

Consider now a sequence $x_n\in\Z_+^d$ with $\lim_n|x_n|=\infty$ and $\lim_n x_n/|x_n| =
M/|M|$ which is fundamental for the random walk $(Z(t))$ and let 
\be\label{e8-15}
h(x) ~=~ \lim_{n\to\infty} \frac{G(x,x_n)}{G(x_0,x_n)}, \quad \forall x\in\Z^d_+.
\ee
Then by compactness, there is a
subsequence $(x_{n_k})$ which is also fundamental for the local random walk
$(Z_{\Lambda(M)}(t))$ and consequently, there exist a harmonic function $f>0$ of the
induced Markov chain $(X_M(t))$ such that 
\[
\lim_{k\to\infty} \frac{G(x,x_{n_k})}{G(x_0,x_{n_k})} ~=~ \frac{f(x^\Lambda) -
\E_x\bigl(f(S^\Lambda(\tau), \; \tau < \tau_\Lambda)\bigr)}{f(x_0^\Lambda) -
\E_{x_0}\bigl(f(S^\Lambda(\tau), \; \tau < \tau_\Lambda)\bigr)},\quad \forall x\in\Z^d_+
\]
(we use here \eqref{e8-14} with $z_k = x_{n_k}$). Comparison of the last relation with
\eqref{e8-15} shows that 
\[
h(x) ~=~ \frac{f(x^\Lambda) -
\E_x\bigl(f(S^\Lambda(\tau), \; \tau < \tau_\Lambda)\bigr)}{f(x_0^\Lambda) -
\E_{x_0}\bigl(f(S^\Lambda(\tau), \; \tau < \tau_\Lambda)\bigr)},\quad \forall x\in\Z^d_+
\]
and consequently, 
\be\label{e8-16}
\lim_{n\to\infty} \frac{G(x,x_{n})}{G(x_0,x_{n})} ~=~ \frac{f(x^\Lambda) -
\E_x\bigl(f(S^\Lambda(\tau), \; \tau < \tau_\Lambda)\bigr)}{f(x_0^\Lambda) -
\E_{x_0}\bigl(f(S^\Lambda(\tau), \; \tau < \tau_\Lambda)\bigr)},\quad \forall x\in\Z^d_+.
\ee
The second assertion of Theorem~\ref{th1} is therefore also proved. 

\medskip 
Suppose finally that a harmonic function $f>0$ of the induced Markov chain $(X_{M}(t))$ is
  unique to constant multiples and let a sequence  $x_n\in
  \Z_+^{d}$ be such that $\lim_n |x_n| = \infty$ and $\lim_n 
  x_n/|x_n| = M/|M|$. Then for any subsequence  $z_k = x_{n_k}$ which is fundamental for the Markov
  chain $(Z(t))$, one has 
\[
\lim_{n\to\infty} \frac{G(x,x_{n_k})}{G(x_0,x_{n_k})} ~=~ \frac{f(x^\Lambda) -
\E_x\bigl(f(S^\Lambda(\tau), \; \tau < \tau_\Lambda)\bigr)}{f(x_0^\Lambda) -
\E_{x_0}\bigl(f(S^\Lambda(\tau), \; \tau < \tau_\Lambda)\bigr)},\quad \forall x\in\Z^d_+.
\]
 with the same function $f$. By
  compactness, from this it follows that the sequence $(x_n)$ is fundamental itself and
  satisfies \eqref{e8-16}. Theorem~\ref{th1} is therefore proved.

\section{Proofs of Proposition~\ref{pr1-1} and Proposition~\ref{pr1-2}}\label{sec9} 
Remark first of all that Proposition~\ref{pr1-1} is a particular case of
Proposition~\ref{pr1-2}, when the set $\Lambda(M) ~=~ \{i : M^i = 0\}$ contains only one
point. The proof of Proposition~\ref{pr1-1} is therefore the same and even simpler than
the proof of Proposition~\ref{pr1-2}. To prove Proposition~\ref{pr1-2} we use the results of of Picardello and
  Woess~\cite{Picardello-Woess}. Before 
formulating these results we recall  some useful properties of minimal $t$-harmonic functions
  and the convergence norm of transition kernels. 

\subsection{Convergence norm and $t$-harmonic functions} 

\begin{defi} 
Let $P=(p(x,x'), \; x,x'\in E)$ be a transition kernel of a time-homogeneous, irreducible Markov
chains $\xi=(\xi(t))$  on a countable, discrete state spaces $E$. 
\begin{enumerate}
\item For $t > 0$, a positive function $f : E\to\R_+$ is
said to be $t$-harmonic for $P$ if it satisfies the equality $Pf = tf$ (i.e. if it is an
eigenvector of the transition operator $P$ with respect to the eigenvalue $t$).   
\item A $t$-harmonic function $f>0$ is said to be minimal if for any  $t$-harmonic
  function $\tilde{f}>0$ the inequality $\tilde{f}\leq f$ implies the equality $\tilde{f}
  = c f$ with some $c>0$. 
\item The convergence norm $\rho(P)$ of $P$ is defined by 
\[
\rho(P) ~=~ \limsup_{n\to\infty} \left(\P_x(\xi(n)=x')\right)^{1/n}.
\]
By irreducibility, $\rho(P)$ does not depend on $x,x'\in E$ (see~Seneta~\cite{Seneta}). 
\end{enumerate}
\end{defi}
By Perron-Frobenius theorem~(see~\cite{Seneta}), for finite state space $E$, the quantity
$\rho(P)$ is equal to the maximal real eigenvalue of the matrix $P$. When the state space
$E$ is infinite (and countable), Theorem~6.3 of Seneta~\cite{Seneta} gives another equivalent representation of the convergence norm
of an irreducible transition kernel $P$ : 
\be\label{e9-1}
\rho(P) ~=~ \sup_{K\subset E} \rho_K(P) 
\ee
where the supremum is taken over all finite subsets $K\subset E$, and for any finite set
$K\subset E$, $\rho_K(P)$ is the maximal real eigenvalue of the truncated transition
matrix $\bigl(p(x,x'), \; x,x'\in K\bigr)$.

Recall finally that for $t > 0$, the set of $t$-harmonic functions of an irreducible Markov
kernel $P$ on a countable state space $E$ is nonvoid only if $t\geq\rho(P)$,
see~Pruitt~\cite{Pruitt}. For $t=1$, the $t$-harmonic functions are called harmonic. 

\medskip

Consider now a probability measure $\nu$ on $\Z$, let $(\xi(t))$ be a random walk on $\Z$ with transition probabilities 
$p(k,k') ~=~ \nu(k'-k)$ and let $T$ denote the first time when the random walk $(\xi(t))$
becomes negative or zero :
\[
T = \inf\{n\geq 0 : \xi(n) \leq 0\}.
\]
To prove Propositions~\ref{pr1-1} and ~\ref{pr1-2} we need to identify the convergence norm and the
harmonic functions of the substochastic Markov kernel 
\[
P_+ = (p(k,k') = \nu(k'-k),
\;  k,k' > 0)
\] 
on $\Z_+~\dot=~\{k\in\Z : k > 0\}$. 
The assumptions we need on the measure $\nu$ are the following 
{\em 
\begin{itemize}
\item[(B1)] the substochastic matrix $P_+ = (p(k,k') = \nu(k'-k), \;  k,k' > 0)$ is
  irreducible,  
\item[(B2)] $\sum_{k\in\Z} \nu(k) \, k ~=~ 0$ and 
\item[(B3)] for some $\eps > 0$,
$$\sum_{k\in\Z} \nu(k) \, \exp(\eps|k|) ~<~ \infty.$$
\end{itemize}
}

\begin{prop}\label{pr9-1} Under the hypotheses (B1)-(B3), $\rho(P_+) ~=~ 1$ and the only
  positive harmonic functions of $P_+$ are the constant multiples of 
\[
f(k) ~=~ k - \E_k(\xi(T)), \quad k\in\Z_+. 
\]
\end{prop} 
\begin{proof} The inequality $\rho(P_+)\leq 1$ is clearly satisfied because the matrix
  $P_+$ is substochastic. To prove that $\rho(P_+)\geq 1$ we  consider a transition kernel 
\[
P ~=~ \bigl(\nu(k'-k), \; k,k'\in\Z\bigr) 
\]
of the homogeneous random walk $(\xi(t))$. Under the hypotheses (H1-(H2), this is
an irreducible random walk on $\Z$ with zero mean and a finite variance. The random walk
$(\xi(t))$ is therefore recurrent and consequently, $\rho(P) ~=~ 1$. Moreover, using
\eqref{e9-1} one gets 
\be\label{e9-1p}
\rho(P) ~=~ \sup_{K\subset \Z} \rho_K(P) \quad \text{ and } \quad \rho(P_+) ~=~
\sup_{K\subset \Z_+} \rho_K(P) 
\ee
where the supremums are taken over  finite sets $K$ and 
$\rho_K(P)$ is the maximal real eigenvalue of the truncated matrix $(\nu(k'-k), \; k,k'\in
K)$. Since the components of the matrix $P$ are invariant with respect to the translations
on $k\in\Z$,  
\[
\rho_K(P) ~=~ \rho_{k+K}(P), \quad \forall k\in\Z
\] 
for any finite set $K\subset\Z$. Hence, the right hand sides of the equalities
\eqref{e9-1p} are equal to each other and  consequently, $\rho(P_+) ~=~ \rho(P) ~=~ 1$. The first assertion of
Proposition~\ref{pr9-1} is therefore proved. 

When $\nu(0) ~>~ 0$, 
the second assertion follows from 
Theorem~1 of Doney~\cite{Doney:02} (see also Example E 27.3 in Chapter VI of
Spitzer~\cite{Spitzer}) and is proved in Lemma~5.3 of the paper ~\cite{Ignatiouk:06}. To
prove  the second assertion for a probability measure $\nu$ with $\nu(0)=0$, it is
sufficient to notice that the  transition kernel $P_+$ has the
same harmonic functions as the modified transition
kernel $\tilde{P}_+ = (\tilde{p}(k,k') ~=~ \tilde{\nu}(k'-k), \;
k,k'\in\Z_+)$ with  
\[
\tilde{\nu}(k) ~=~ \begin{cases}(1-\theta) \nu(k) &\text{if $k\not= 0$},\\
\theta &\text{for $k=0$}, \quad \text{ where \; $0<\theta < 1$,}
\end{cases}
\]
and that the modified random walk $(\tilde{\xi}(t))$ with transition
probabilities $\tilde{p}(k,k') = \tilde\nu(k'-k)$ and $\tilde{T} = \inf\{ n \geq 0 : \; \tilde{\xi}(t)
\leq 0\}$ satisfy  the equality $\E_k(\tilde{\xi}(\tilde{T})) ~=~ \E_k(\xi(T))$ for all
$k\in\Z_+$. \end{proof}

\subsection{Cartesian products of Markov chains}
Let $P=(P(u,u'), \; u,u'\in E_1)$ and $Q=(Q(u,u'), \; u,u'\in E_2)$ be two transition kernels of two  time-homogeneous Markov
chains $\xi(t)$ and $\eta(t)$ on countable, discrete state spaces $E_1$ and $E_2$ respectively. A Markov chain on the
Cartesian product $E_1\times E_2$ having transition kernel 
\be\label{e9-2}
R_a = a\, P\otimes\1_{E_2} + (1-a) \, \1_{E_1}\otimes Q
\ee
with some $0 < a < 1$, where $\1_{E_i}$ denotes the identity operator on $E_i$, is usually
called a {\em Cartesian product of Markov chains} $\xi$ and $\eta$. Transition
probabilities of such a Markov chain are given by 
\[
R_a((u^1,u^2), (v^1,v^2)) ~=~ \begin{cases} a~P(u^1, v^1) &\text{ if $u^2 = v^2$,}\\
(1-a)~Q(u^2,v^2) &\text{ if $u^1=v^1$,}\\
0 &\text{ otherwise.}
\end{cases}
\]
The kernel
$R_a=(R_a(u,v), \; u,v\in E\times E_2) $ is a {\em Cartesian
  product of transition kernels} $P$ and $Q$. By Lemma~3.1 of 
Picardello and Woess~\cite{Picardello-Woess}, the
convergence norm of the Cartesian product $R_a$ is related to those of $P$ and $Q$ as
follows 
\be\label{e9-3}
\rho(R_a) ~=~ a\rho(P) + (1-a)\rho(Q)
\ee
and it is clear that for any $r\geq \rho(P)$ and $s\geq\rho(Q)$, if $f>0$  is a $r$-harmonic
function of $P$ and $g>0$ is a $s$-harmonic function  of $Q$ then the function 
\be\label{e9-4}
h(u^1,u^2) ~=~ f\otimes g ~(u^1,u^2) ~\dot=~ f(u^1) g(u^2), \quad (u^1,u^2)\in E_1\times E_2
\ee
is a $t$-harmonic function of $R_a$ with $t= a r + (1-a) s$.   Conversely, for minimal $t$-harmonic
functions of the Cartesian product $R_a$, Theorem~3.2 of Picardello and Woess~\cite{Picardello-Woess} proves
the following property. 

\medskip
\noindent
{\bf Theorem~[Picardello and Woess~\cite{Picardello-Woess}].} {\em If the transition kernels $P$
  and $Q$ are stochastic  and irreducible  respectively on $E_1$ and $E_2$, then  for any
  $0< a < 1$ and  $t\geq \rho(R_a)$,
  every minimal $t$-harmonic function $h>0$ of the Cartesian product $R_a$ is of the form
  $f\otimes g$ with some minimal $r$-harmonic function $f>0$ of $P$, a minimal $s$-harmonic
  function $g>0$ of $Q$ and $r\geq \rho(P)$, $s\geq\rho(Q)$ satisfying the equality $a r +
  (1-a) s = t$.}

\medskip 

In a particular case, when $\rho(P) = \rho(Q) = 1$, for minimal harmonic
(i.e. $t$-harmonic with $t=1$) functions, this result implies the following statement. 

\begin{cor}\label{cor9-1} Suppose that the transition kernels $P$
  and $Q$ are stochastic  and irreducible   respectively on $E_1$ and $E_2$ and let $\rho(P)
  = \rho(Q) = 1$. Then for any $0< a < 1$, every minimal harmonic function $h > 0$ of $R_a$ is of the form
  $f\otimes g$ with some minimal harmonic functions $f>0$ and $g>0$ of $P$ and $Q$
  respectively. 
\end{cor}

For sub-stochastic transition kernels, one gets therefore 

\begin{prop}\label{pr9-2} Suppose that the sub-stochastic  transition kernels $P$
  and $Q$ are irreducible   respectively on $E_1$ and $E_2$ and let $\rho(P)
  = \rho(Q) = 1$. Suppose moreover that every positive harmonic function of $P$ is a
  constant multiple of $f > 0$ and every positive harmonic function of $Q$ is a
  constant multiple of $g > 0$. Then for any $0< a < 1$, every positive harmonic function  of $R_a$ is
  a constant multiple of the function $f\otimes g$.  
\end{prop}
\begin{proof} It is sufficient to apply the above statement for twisted 
  transition kernels $\tilde{P} =(\tilde{P}(u,u'), \; u,u'\in\E_1)$ and
  $\tilde{Q}=(\tilde{Q}(u,u'), \; u,u'\in\E_2) $ defined by 
\[
\tilde{P}(u,u') ~=~ P(u,u') f(u')/f(u), \quad \text{ for } \; u,u'\in E_1 
\]
and 
\[
\tilde{Q}(u,u') ~=~ Q(u,u') g(u')/g(u), \quad \text{ for } \; u,u'\in E_2
\]
Under the hypotheses of Proposition~\ref{pr9-2}, these transition kernels are clearly
stochastic and irreducible, the only positive harmonic functions of $\tilde{P}$ (resp. $\tilde{Q}$) are
constant and 
\[
\rho(\tilde{P}) ~=~ \rho(P) ~=~ \rho(Q) ~=~ \rho(\tilde{Q}) ~=~ 1.
\]
By Corollary~\ref{cor9-1}, from this it follows that every minimal harmonic function of
the Cartesian product $\tilde{R}_a = a \, \tilde{P}\otimes \1_{E_2} + (1-a)\,
\1_{E_1}\times\tilde{Q}$ is also constant. To compete the proof of our proposition it is
now sufficient to compare the transitions kernels $\tilde{R}_a=(\tilde{R}_a(u,v), \; u,v\in E\times E_2)$ and $R_a=(R_a(u, v), \; u, v\in E\times
E_2)$. Since clearly  
\[
\tilde{R}_a(u,v) ~=~ R_a(u,v) ~ f\otimes g(v)/ f\otimes g(u), \quad \forall u, v \in E_1\times E_2,
\]
this proves that the only minimal harmonic functions (and consequently, also the only positive
harmonic functions) of $R_a$ are the constant multiples of $f\otimes g$. 
\end{proof}

For $a=(a_1,\ldots,a_m)\in ]0,1[^m$ with $a_1+\cdots + a_m = 1$,  a Cartesian product
    $R_a$ of
    $m$ transition kernels $P_1,\ldots, P_m$ of   time-homogeneous Markov 
chains $\xi_1(t)$, \ldots, $\xi_m(t)$ on countable, discrete state spaces
$E_1,\ldots, E_m$, is defined by 
\[
R_a = a_1P_1\otimes \1_{E_2} \otimes \cdots \otimes \1_{E_m} + a_2\1_{E_1} \otimes P_2 \otimes \1_{E_3}\!
\cdots \!\otimes \1_{E_m} + \cdots + a_m\1_{E_1}\otimes\cdots \otimes \1_{E_{m-1}}\otimes P_m.
\]
By induction with respect to $m$ and using \eqref{e9-3}, from
Proposition~\ref{pr9-2} it follows  

\begin{cor}\label{cor9-2} Suppose that the sub-stochastic  transition kernels $P_1$,\ldots, $P_m$
  are irreducible   respectively on $E_1$,\ldots ,$E_k$ and let
  $\rho(P_i)= 1$ for all $i=1,\ldots,m$. Suppose moreover that for any $i=1,\ldots, m$, every positive harmonic function of $P_i$ is a
  constant multiple of $f_i > 0$. Then for any $a=(a_1,\ldots,a_m)\in ]0,1[^m$ with 
    $a_1+\cdots + a_m = 1$, every positive harmonic function of $R_a$ is 
  a constant multiple of the function $f_1\otimes \cdots \otimes f_m$. 
\end{cor}

\subsection{Application for a homogeneous random walk on $\Z^m$.} 

Let  $(\xi(t))$ be a homogeneous random walk on $\Z^m$ with transition probabilities 
\[
P(u,u') ~=~ \nu(u'-u), \quad u,u'\in\Z^m
\]
and let $\tau_i$ denote the first time when the $i$-th coordinate $\xi^i(t)$ of $\xi(t)$ becomes
negative or zero. Suppose moreover that the probability measure $\nu$ on $\Z^m$ satisfies the following
conditions. 

\begin{itemize}
\item[(C1)] For $\Z_+^m ~\dot=~\{u\in\Z^m : u^i > 0, \; \forall i=1,\ldots,m\}$ the
  sub-stochastic matrix $(P(u,u') ~=~ \nu(u'-u), \; u,u'\in\Z^m_+\}$ is irreducible. 
\item[(C2)] The function
\[
\varphi_\xi(\a) ~\dot=~ \sum_{u\in\Z^m} \nu(u) \exp(\a\cdot u) 
\]
is finite in a neighborhood of zero in $\R^m$.
\item[(C3)] $\sum_{u\in\Z^m} u \nu(u) ~=~ 0$. 
\item[(C4)] $\nu(u) ~=~ 0$ if $u^i u^j \not= 0$ for some $1 \leq i < j \leq m$. 
\end{itemize}
Under the hypotheses (C1)-(C4), for every $1\leq i\leq m$, the $i$-th
coordinate $(\xi^i(t))$ of $(\xi(t))$ is a recurrent random walk on $\Z$ and consequently,
almost surely $\tau_i < \infty$.  
If the hypotheses (C1)-(C4) are satisfied and moreover 
\[
\nu(0) = 0 
\]
then the transition 
  kernel $P_+$ is a Cartesian
  product of $m$  irreducible  transition kernels $P_i = (P_i(u,u') ~=~
  \nu_i(k'-k), \; k,k'\in\Z_+)$: 
\[
P_+ ~=~ a_1~P_1\otimes \1_{\Z_+} \otimes \cdots \otimes \1_{\Z_+} + a_2~\1_{\Z_+} \otimes P_2 \otimes \1_{\Z_+}
\cdots \otimes \1_{\Z_+} + \cdots + a_m~\1_{\Z_+}\otimes \cdots \otimes \1_{\Z_+}\otimes
P_m
\]
with 
\be\label{e9-5}
a_i ~=~ \sum_{k\in\Z} \nu(k e_i) 
\ee
and 
\be\label{e9-6}
\nu_i(k) ~=~ \frac{1}{a_i} \nu(k e_i), \quad k\in\Z 
\ee
where $e_i=(e_i^1,\ldots,e_i^m)$ denotes the unit vector in $\Z^m$ with $e_i^i=1$ and
$e_i^j=0$ for $j\not= i$. Because of the assumptions (C1)-(C3), the probability measures $\nu_1$,\ldots,$\nu_m$ satisfy the
conditions (B1)-(B3) of Proposition~\ref{pr9-1} and hence, using Corollary~\ref{cor9-2}
one gets 

\begin{cor}\label{cor9-3} If the conditions (C1)-(C4) are satisfied and 
  $\nu(0)=0$ then 
the only positive harmonic
functions of the transition kernel $P_+$ are
the constant multiples of the function 
\be\label{e9-7}
f(u) ~=~ \prod_{i=1}^m \left(u^i ~-~ \E_{u^i}( \eta_i(T_i))\right),
\ee 
where $(\eta_i(t))$ is a random walk on $\Z$ with transition probabilities
$p_i(k,k')=\nu_i(k'-k)$ and $T_i = \inf\{n\geq 0 : \eta_i(n)\leq 0\}$. 
\end{cor}

The main result of this section is the following statement. 

\begin{prop}\label{pr9-3} Under the hypotheses (C1)-(C4), the only positive harmonic
  functions of the substochastic kernel $P_+ $ are the constant multiples of the function 
\be\label{e9-8}
f(u) ~=~ \prod_{i=1}^m u^i ~-~ \E_u\left( \prod_{i=1}^m \xi^i\left(\min_{1\leq i\leq m}\tau_i\right) \right), \quad u\in\Z_+^m 
\ee
\end{prop}

Corollary~\ref{cor9-3} proves this proposition when $\nu(0)=0$ and the coordinates $(\xi^i(t))$ are lower-semicontinuous, 
i.e. if $\nu(u) ~=~ 0$ whenever $u^i < -1$ for some $1\leq i\leq m$. Indeed, in this 
case, almost surely, $\xi^i(\tau_i) = \eta^i(T_i) = 0$ and consequently, the
right hand sides of \eqref{e9-7} and \eqref{e9-8} are equal to $u^1\times \cdots \times u^m$. 
To prove Proposition~\ref{pr9-3} in a general case, we need the
following preliminary results. As above, for a given $\Lambda\subset\{1,\ldots,m\}$, we denote 
\[
\tau_\Lambda ~\dot=~ \min_{i\in\Lambda}\tau_i
\]

\begin{lemma}\label{lem9-1} Under the hypotheses (C1)-(C4), for any $\eps > 0$ there is
  $\delta_\eps >0$ such that  for any $0<\delta < \delta_\eps$ and any non-empty subset
  $\Lambda\subset\{1,\ldots,m\}$, 
\be\label{eq9-9}
\E_u\left( \exp\Bigl(-\eps \tau_\Lambda + \delta \sum_{i=1}^{m}\left|
\xi^i\left(\tau_\Lambda\right)\right|\Bigr) \right) < \infty.  
\ee
\end{lemma}
\begin{proof} To prove this lemma we first notice that for any $\delta >0$, 
\begin{multline*}
\exp\left( \delta \sum_{i=1}^{m}\left|
\xi^i\left(\tau_\Lambda\right)\right|\right)  \leq \sup_{\substack{\a=(\a^1,\ldots,\a^m)\in\R^m :\\
  |\a^i|=\delta, \;\forall i=1,\ldots,m}}
~\exp\left(\a\cdot\xi\left(\tau_\Lambda\right)\right) \\\leq~ \sum_{\substack{\a=(\a^1,\ldots,\a^m)\in\R^m :\\
  |\a^i|=\delta, \;\forall i=1,\ldots,m}}
~\exp\left(\a\cdot\xi\left(\tau_\Lambda\right)\right) 
\end{multline*}
from which it follows that 
\begin{align*}
\E_u\left( \exp\Bigl(-\eps \tau_\Lambda + \delta \sum_{i=1}^{m}\left|
\xi^i\left(\tau_\Lambda\right)\right|\Bigr) \right) \leq \sum_{\substack{\a=(\a^1,\ldots,\a^m)\in\R^m :\\
  |\a^i|=\delta, \;\forall i=1,\ldots,m}} \!\E_u\Bigl( \exp\bigl(-\eps \tau_\Lambda + \a\cdot \xi\left(\tau_\Lambda\right) \bigr)\Bigr) \\
\leq~ 2^m ~\sup_{\a\in\R^m : |\a|_\infty \leq m \delta} ~\E_u\Bigl(
\exp\bigl(-\eps \tau_\Lambda + \a\cdot \xi\left(\tau_\Lambda\right) \bigr)\Bigr) 
\end{align*}
where 
\begin{align*}
\E_u\Bigl( \exp\bigl(-\eps \tau_\Lambda + \a\cdot \xi\left(\tau_\Lambda\right) \bigr)\Bigr)  &=~ \sum_{n=0}^\infty \exp(-\eps n) ~\E_u\left( 
\exp\left( \a\cdot \xi\left(n\right) \right), \; \tau_\Lambda 
= n\right)\\
&\leq~
\sum_{n=0}^\infty \exp(-\eps n) ~\E_u\left( 
\exp\left( \a\cdot \xi\left(n\right) \right)\right) \\&= \exp(\a\cdot u) \!\sum_{n=0}^\infty \exp(-\eps n)
 \varphi_\xi(\a)^n = \frac{\exp(\a\cdot u)}{1-\exp(-\eps)
 \varphi_\xi(\a)}  
\end{align*}
whenever $\varphi_\xi(\a) < \exp(\eps)$. Since the jump generating function
$\varphi_\xi$ is continuous and $\varphi_\xi(0) = 1$, from
this it follows that the left hand side of \eqref{eq9-9} is finite whenever $\delta > 0$
is small enough. Lemma~\ref{lem9-1} is therefore proved. 
\end{proof}

\begin{lemma}\label{lem9-3} Suppose that the conditions (C1)-(C4) are satisfied and let
  $\nu(0)=0$. Then for any $u\in\Z^m_+$, 
\begin{multline}\label{e9-12}
\E_u\left(\xi^m(\tau_m) \prod_{1\leq i\leq m-1} \xi^i\left(\min_{1\leq i\leq
  m-1}\tau_i\right) \right) \\~=~ \E_u(\xi^m(\tau_m)) \times E_u\left(\prod_{1\leq i\leq m-1} \xi^i\left(\min_{1\leq i\leq
  m-1}\tau_i\right) \right).  
\end{multline} 
\end{lemma}
\begin{proof} To prove this lemma we consider a continuous-time version of the random walk
  $(\xi(t))$ defined by 
\[
\hat{\xi}(t) ~=~ \xi({\cal N}(t)), \quad t\in[0,+\infty[,
\]
where ${\cal N}(t)$ is the Poisson process on $[0,+\infty[$ with rate $1$. Because of the
    assumption (C4), the coordinates
    $(\hat{\xi}^1(t))$, \ldots $(\hat{\xi}^m(t))$ of such a random walk $(\hat{\xi}(t))$
    perform independent random walks on $\Z$: for any $t \geq 0$, ${\cal N}(t) ~=~ {\cal
      N}_1(a_1t) + \cdots + {\cal N}_m(a_mt)$ and 
\[
\hat{\xi}^i(t) ~=~ \eta_i({\cal N}_{i}(a_i t))
\]
where $({\cal N}_1(a_1t)), \ldots, ({\cal N}_m(a_mt))$ are independent Poisson processes
on $[0,+\infty[$ with rate $1$. The stopping times $\hat\tau_i ~\dot=~ \inf\{t \geq 0 :
    \hat\xi^i(t) \leq 0\}$ for $i=1,\ldots,m$, are therefore independent and consequently, the random variables 
\[
\hat\xi^m(\hat\tau_m) \quad \text{ and } \quad   \prod_{1\leq i\leq m-1}
\hat\xi^i\left(\min_{1\leq i\leq m-1}\hat\tau_i\right)  
\]
are also independent. Since almost surely, $\hat\xi^m(\hat\tau_m) = \xi^m(\tau_m)$ and 
\[
 \prod_{1\leq i\leq m-1}
\hat\xi^i\left(\min_{1\leq i\leq m-1}\hat\tau_i\right) ~=~  \prod_{1\leq i\leq m-1}
\xi^i\left(\min_{1\leq i\leq m-1}\tau_i\right),
\]
from this it follows \eqref{e9-12}. 
\end{proof}

\begin{lemma}\label{lem9-2} Suppose that the conditions (C1)-(C4) are satisfied and let
  $\nu(0)=0$. Then for any $1\leq i\leq m$ and $k\in\Z_+$, 
\be\label{e9-9p}
\E_k(\xi(\tau_i)) ~=~ \E_k(\eta^i(T_i))
\ee
and for any non-empty subset $\Lambda\subset\{1,\ldots,m\}$, $\eps > 0$ and
  $u=(u^1,\ldots,u^m)\in\Z^m_+$, 
\be\label{e9-9}
\E_u\Bigl( \exp\bigl(-\eps \tau_\Lambda\bigr)
\prod_{i=1}^{m} \xi^i\left(\tau_\Lambda\right)\Bigr) ~=~ \E_{u}\Bigl( \exp\bigl(-\eps
\tau_\Lambda\bigr)
\prod_{i\in\Lambda}\xi^i\left(\tau_\Lambda\right)\Bigr) \times
\prod_{i\not\in\Lambda} u^i, 
\ee
\end{lemma}
\begin{proof} Remark first that by Lemma~\ref{lem9-1}, the left and the right
  hand sides of these equalities are well defined, 
because clearly,  
\[
\left| \xi^i\left(\tau_\Lambda\right)\right| ~\leq~ \frac{1}{\delta}\exp(\delta \left|
\xi^i\left(\tau_\Lambda\right)\right|), \quad \forall \delta >0.
\]
Consider now a sequence of independent identically distributed random variables
  $(\theta_n)$ taking the values in the set  $\{1,\ldots,m\}$ with $
\P\left(\theta_n = i \right) ~=~ a_i$ for all $1\leq i\leq m$ 
with the quantities $a_i$ defined  by \eqref{e9-5}. Denote $
N_i(t) ~=~ \1_{\{\theta_1 = i\}} + \cdots + \1_{\{\theta_t = i\}}$ 
and let $(\eta_1(t)), \ldots (\eta_m(t))$ be independent random walks on $\Z$ which are independent on the
sequence $(\theta_n)$ and have transition probabilities
$p_i(k,k')=\nu_i(k'-k)$ defined by \eqref{e9-6} respectively for $i=1, \ldots,m$.  Then
our random walk $\xi(t)=(\xi^1(t),\ldots,\xi^m(t))$ can be represented in the 
following way : $\xi^i(t) = \eta^i(N_i(t))$ for any  $i\in\{1,\ldots,m\}$
and $t\in\N$. According to this representation, 
the stopping times $\tau_i~\dot=~ \inf\{n\geq 0 :
\xi^i(n)\leq 0\}$ and $T_i ~\dot=~ \inf\{
n\geq 0 : \eta^i(n) \leq 0\}$ are related as follows : 
\[
T_ i~=~ N_i(\tau_i) \quad \text{and} \quad  \tau_i = \inf\{n\geq 0 : N_i(n) ~=~ T_i\}.
\]
Since clearly, $\xi^i(\tau_i) = \eta^i(T_i)$, one gets therefore 
\eqref{e9-9p}. Furthermore, for given $\tau_\Lambda$ and $N_1(\tau_\Lambda),
\ldots, N_m(\tau_\Lambda)$, the random vectors $(\xi^i(\tau_\Lambda) =
\eta^i(N_i(\tau_\Lambda)), \; i\in\Lambda)$ and $(\xi^i(\tau_\Lambda) =
\eta^i(N_i(\tau_\Lambda)), \; i\not\in\Lambda)$ are conditionally independent and the
conditional expectation 
\[
\E_u\Bigl(\;\prod_{i\in\{1,\ldots,m\}\setminus\Lambda} \xi^i\left(\tau_\Lambda\right)
\; \Bigr| \; \tau_\Lambda, N_1(\tau_\Lambda),
\ldots, N_m(\tau_\Lambda)\Bigr) 
\]
is equal to 
\begin{multline*}
\E_u\Bigl(\; \prod_{i\in\{1,\ldots,m\}\setminus\Lambda} \eta^i\left(N_i(\tau_\Lambda)\right)
\; \Bigr| \; \tau_\Lambda, N_1(\tau_\Lambda),
\ldots, N_m(\tau_\Lambda)\Bigr) \\~=~ \prod_{i\in\{1,\ldots,m\}\setminus\Lambda}
E_{u^i}\left(\left.\eta^i\left(N_i(\tau_\Lambda)\right)\;\right| N_i(\tau_\Lambda)\right).
\end{multline*} 
Moreover, because of the assumption (C3),
$E_{u^i}\left(\left.\eta^i\left(N_i(\tau_\Lambda)\right)\;\right| N_i(\tau_\Lambda)\right)
= u^i$ for any $i\not\in\Lambda$. The left hand side of
\eqref{e9-9} is therefore equal to 
\[
\E_u\Bigl( \exp\bigl(-\eps \tau_\Lambda\bigr)
\prod_{i\in\Lambda} \xi^i\left(\tau_\Lambda\right)\Bigr) \prod_{i\in\{1,\ldots,m\}\setminus\Lambda}
u^i. 
\]
and consequently, \eqref{e9-9} holds. 
\end{proof}

By strong Markov property, from \eqref{e9-9} it follows  

\begin{cor}\label{cor9-4}
Under the hypotheses of Lemma~\ref{lem9-2}, for any non-empty subset $\Lambda\subset\{1,\ldots,m\}$, $\eps > 0$ and
  $u=(u^1,\ldots,u^m)\in\Z^m_+$,
\begin{multline}\label{e9-10}
\E_u\Bigl( \exp\bigl(-\eps \tau_\Lambda\bigr)
\prod_{i=1}^{m} \xi^i\bigl(\tau_\Lambda\bigr), \;
\tau_{ \{1,\ldots, m\}\setminus \Lambda} < \tau_\Lambda\Bigr) ~=~\\ 
\E_{u}\Bigl( \exp\bigl(-\eps \tau_\Lambda\bigr)
\prod_{i\in\Lambda}\xi^i\left(\tau_\Lambda\right)\times
\prod_{i\not\in\Lambda} \xi^i\left(\tau_{ \{1,\ldots, m\}\setminus \Lambda}\right), \;
\tau_{ \{1,\ldots, m\}\setminus \Lambda}< \tau_\Lambda\Bigr).
\end{multline}
\end{cor}

We are ready now to get  another equivalent representation of the function
\eqref{e9-7}. 
\begin{lemma}\label{lem9-4} If the conditions (C1)-(C4) are satisfied and 
  $\nu(0)=0$, then for any $u=(u^1,\ldots,u^m)\in\Z^m_+$, 
\be\label{e9-14}
 \prod_{i=1}^m \left(u^i - \E_{u^i}( \eta_i(T_i))\right) ~=~ \prod_{i=1}^m \left(u^i -
\E_{u^i}( \xi^i(\tau_i))\right) =~ \prod_{i=1}^m u^i - \E_u\left( \prod_{i=1}^m
\xi^i\left(\min_{1\leq j\leq m}\tau_j\right) \right) 
\ee
\end{lemma}
\begin{proof} The first equality of \eqref{e9-14} follows from \eqref{e9-9p}. To prove the
  second equality, it is convenient to use the induction with respect to $m$. For $m=1$,
  this equality  is trivial. 
Suppose now that \eqref{e9-14} holds for some $m=k\geq 1$. Then 
\begin{multline*}
\prod_{i=1}^{k+1} \left(u^i -
\E_{u^i}( \xi^i(\tau_i))\right) = \left(\prod_{i=1}^k u^i ~-~
\E_u\left( \prod_{i=1}^k \xi^i\Bigl(\min_{1\leq j\leq k}\tau_j\Bigr) \right)
\right)\\ \times \left(u^{k+1} ~-~ \E_{u^{k+1}}\left( \xi^{k+1}(\tau_{k+1})\right)\right) 
\end{multline*}
and hence, to get \eqref{e9-14} for $m=k+1$, it is sufficient to show that 
\begin{align}
\E_u\left( \prod_{i=1}^{k+1} \xi^i\left(\min_{1\leq j\leq k+1}\tau_j\right) \right) ~=~
u^{k+1} \times \E_u\left( \prod_{i=1}^k \xi^i\left(\min_{1\leq j\leq
  k}\tau_j\right) \right)  \hspace{1.5cm}\label{e9-15}\\ ~+~ \E_{u^{k+1}}(
\xi^{k+1}(\tau_{k+1})) \times\prod_{i=1}^k u^i  ~-~ \E_{u^{k+1}}(
\xi^{k+1}(\tau_{k+1}))   \times \E_u\left( \prod_{i=1}^k \xi^i\left(\min_{1\leq j\leq
  k}\tau_j\right) \right)  \nonumber 
\end{align}
To prove this equality let us notice that  because of the assumption (C4),
almost surely $\tau_i\not=\tau_j$ for all $1\leq  i < j\leq
k+1$. Hence, almost surely, only one of the coordinates $\xi^1(\min_{1\leq j\leq k+1}\tau_j)$,
\ldots, $\xi^{k+1}(\min_{1\leq j\leq k+1}\tau_j)$ is negative or zero, and consequently,  
\[
\prod_{i=1}^{k+1} \xi^i\left(\min_{1\leq i\leq k+1}\tau_i\right) ~\leq ~ 0. 
\]
The left hand side of \eqref{e9-15} is therefore negative or zero. 
While we do not yet know whether the left hand side of \eqref{e9-15} finite or equal to $-\infty$, the quantities 
\[
\E_u\left( \exp\Bigl(-\eps \min_{1\leq i\leq k+1}\tau_i\Bigr)
\prod_{i=1}^{k+1} \xi^i\Bigl(\min_{1\leq j\leq k+1}\tau_j\Bigr) \right) 
\] 
are well defined for any $\eps > 0$ by Lemma~\ref{lem9-1}, and by monotone
convergence theorem, 
\be\label{e9-16}
\E_u\left( \prod_{i=1}^{k+1} \xi^i\Bigl(\,\min_{1\leq j\leq k+1}\tau_j\Bigr) \right)  =
\lim_{\substack{\eps\to 0\\ \eps > 0}} \E_u\left( \exp\Bigl(-\eps \min_{1\leq j\leq k+1}\tau_j\Bigr)
\prod_{i=1}^{k+1} \xi^i\Bigl(\, \min_{1\leq j\leq k+1}\tau_j\Bigr) \right). 
\ee
Furthermore, using again the fact that almost surely $\tau_i\not=\tau_{k+1}$ for all
$i=1,\ldots,k$,  one gets 
\begin{align*}
\E_u&\left( \exp\Bigl(-\eps \min_{1\leq j\leq k+1}\tau_j\Bigr)
\prod_{i=1}^{k+1} \xi^i\Bigl(\min_{1\leq j\leq k+1}\tau_j\Bigr) \right) \\
&= \E_u\left( \exp\Bigl(-\eps \min_{1\leq j\leq k}\tau_j\Bigr)
\prod_{i=1}^{k+1} \xi^i\Bigl(\min_{1\leq j\leq k}\tau_j\Bigr), \; \min_{1\leq j\leq k}
\tau_j < \tau_{k+1}\right)\\
&\hspace{2.6cm}+ \E_u\left( \exp\Bigl(-\eps \tau_{k+1}\Bigr)
\prod_{i=1}^{k+1} \xi^i\left(\tau_{k+1}\right), \; \min_{1\leq j\leq k}
\tau_j > \tau_{k+1}\right) \nonumber\\
&= \E_u\left( \exp\Bigl(-\eps \min_{1\leq j\leq k}\tau_j\Bigr)
\prod_{i=1}^{k+1} \xi^i\Bigl(\min_{1\leq j\leq k}\tau_j\Bigr)\right) + \E_u\left( \exp\Bigl(-\eps \tau_{k+1}\Bigr)
\prod_{i=1}^{k+1} \xi^i\left(\tau_{k+1}\right)\right) \\
&\hspace{2.6cm} - \E_u\left( \exp\Bigl(-\eps \min_{1\leq j\leq k}\tau_j\Bigr)
\prod_{i=1}^{k+1} \xi^i\Bigl(\min_{1\leq j\leq k}\tau_j\Bigr), \; \min_{1\leq j\leq k}
\tau_j > \tau_{k+1}\right)\\
&\hspace{3.8cm} - \E_u\left( \exp\Bigl(-\eps \tau_{k+1}\Bigr)
\prod_{i=1}^{k+1} \xi^i\left(\tau_{k+1}\right), \; \min_{1\leq j\leq k}
\tau_j < \tau_{k+1}\right)  
\end{align*} 
For the right hand side of the last equality,  Lemma~\ref{lem9-2} applied with
$\Lambda=\{1,\ldots,k\}$ proves that the first term if  is equal to  
\[
u^{k+1}\times \E_u\left( \exp\Bigl(-\eps \min_{1\leq j\leq k}\tau_j\Bigr)
\prod_{i=1}^{k} \xi^i\Bigl(\min_{1\leq j\leq k}\tau_j\Bigr)\right), 
\]
again by Lemma~\ref{lem9-2} but now with $\Lambda=\{t_{k+1}\}$ the
second term is equal to 
\[
\E_u\left( \exp\Bigl(-\eps \tau_{k+1}\Bigr) \xi^{k+1}\left(\tau_{k+1}\right)\right) \times
\prod_{i=1}^{k} u^i,
\]
Corollary~\ref{cor9-4} applied with $\Lambda=\{1,\ldots,k\}$ shows that the third term is equal to 
\[
\E_u\left( \exp\Bigl(-\eps \min_{1\leq j\leq k}\tau_j\Bigr)
\xi^{k+1}(\tau_{k+1})\times \prod_{i=1}^{k} \xi^i\Bigl(\min_{1\leq j\leq k}\tau_j\Bigr), \; \min_{1\leq j\leq k}
\tau_j > \tau_{k+1}\right) 
\]
and again by Corollary~\ref{cor9-4} but now with $\Lambda=\{t_{k+1}\}$, the fourth term is
equal to 
\[
\E_u\left( \exp\Bigl(-\eps \tau_{k+1}\Bigr)
\xi^{k+1}(\tau_{k+1})\times \prod_{i=1}^{k} \xi^i\Bigl(\min_{1\leq j\leq k}\tau_j\Bigr), \; \min_{1\leq j\leq k}
\tau_j > \tau_{k+1}\right).  
\]
Since the sum of the last two terms is equal to 
\[
\E_u\left( \exp\Bigl(-\eps \min_{1\leq j\leq k}\tau_j\Bigr)
\xi^{k+1}(\tau_{k+1})\times \prod_{i=1}^{k} \xi^i\Bigl(\min_{1\leq j\leq k}\tau_j\Bigr)\right) 
\]
letting $\eps\to 0$ and using monotone convergence theorem one gets 
\begin{align*}
\E_u\left( \prod_{i=1}^{k+1} \xi^i\left(\min_{1\leq j\leq k+1}\tau_j\right) \right) ~=~
u^{k+1} \times \E_u\left( \prod_{i=1}^k \xi^i\left(\min_{1\leq j\leq
  k}\tau_j\right) \right)  \hspace{1.5cm}\\ ~+~ \E_{u^{k+1}}(
\xi^{k+1}(\tau_{k+1})) \times\prod_{i=1}^k u^i  ~-~   \times \E_u\left(
\xi^{k+1}(\tau_{k+1}) \times \prod_{i=1}^k \xi^i\left(\min_{1\leq j\leq
  k}\tau_j\right) \right).
\end{align*}
The last relation combined with Lemma~\ref{lem9-3} proves \eqref{e9-14}. 
\end{proof}

\begin{proof}[Proof of Proposition~\ref{pr9-3}] Lemma~\ref{lem9-4} and 
  Corollary~\ref{cor9-3} prove this proposition in the case where $\nu(0)=0$. 
Suppose now that $\nu(0)\not= 0$ and let us notice that  a positive function $f> 0$
on $\Z^m_+$ is harmonic for the transition kernel $P_+$ if and only if it is harmonic for
the new transition kernel $\tilde{P}_+ = (\tilde{P}(u,u') ~=~ \tilde{\nu}(u'-u), \;
u,u'\in\Z^m_+)$ with  
\[
\tilde{\nu}(u) ~=~ \begin{cases} (1-\nu(0))^{-1} \nu(u) &\text{ if $u\not= 0$},\\
0 &\text{ if $u= 0$}. 
\end{cases}
\]
If the original probability measure $\nu$ satisfies the conditions (C1)-(C4), the modified
probability measure also satisfies the conditions (C1)-(C4). Since $\tilde\nu(0)=0$, the only positive harmonic
functions of the modified transition kernel $\tilde{P}_+$, as well as of the original
transition kernel $P_+$, are therefore the constant multiples of the function 
\be\label{e9-18}
f(u) ~=~ \prod_{i=1}^m u^i ~-~ \E_u\left( \prod_{i=1}^m \tilde\xi^i\Bigl(\min_{1\leq i\leq
  m}\tilde\tau_i\Bigr) \right), \quad u\in\Z_+^m
\ee
where $\tilde\xi(t) = (\tilde\xi^1(t),\ldots,\tilde\xi^m(t))$ is the new random walk on
$\Z^m$ with transition probabilities $\tilde{p}(u,u') ~=~ \tilde\nu(u'-u)$ and
$\tilde\tau_i = \inf\{t\geq 0 : \tilde\xi^i(t)\leq 0\}$ for $1\leq
i\leq m$. Consider finally a sequence of independent identically distributed Bernoulli random
variables $(\theta_n)$ with $\P(\theta_n = 1) = 1-\nu(0)$. The random walk $(\xi(t)$ can
be represented in terms of the random walk $(\tilde\xi(t))$ as follows : if the random
walk $(\tilde\xi(t))$ and the sequence $(\theta_n)$ are independent, then letting $N(n) =
\theta_1 +\cdots + \theta_n$ one gets $\xi(n) = \tilde\xi(N(n))$. With such a 
representation, almost surely, 
\[
\prod_{i=1}^m \tilde\xi^i\Bigl(\min_{1\leq i\leq
  m}\tilde\tau_i\Bigr) ~=~ \prod_{i=1}^m \xi^i\Bigl(\min_{1\leq i\leq m}\tau_i\Bigr)
\]
and consequently, the right hand side of \eqref{e9-18} is equal to the right hand side of
\eqref{e9-8}. Proposition~\ref{pr9-3} is therefore proved.
\end{proof}

\subsection{Proofs of Proposition~\ref{pr1-1} and Proposition~\ref{pr1-2}}

Recall that for $\Lambda=\Lambda(M)~\dot~=\{i :~ M^i = 0\}$, the random process $S^{\Lambda}(t) = (S^i(t))_{i\in\Lambda}$
  is a homogeneous random walk on $\Z^{\Lambda}$ with transition
  probabilities 
\[
\P_u\left(S^{\Lambda} = u'\right) ~=~ \mu_\Lambda(u'-u) ~\dot=~  \sum_{x\in\Z^d : x^{\Lambda} = u'-u}
\mu(x), \quad u,u'\in\Z^{\Lambda(q)}
\]
and zero mean 
\[
\E_0(S^\Lambda(1)) ~=~ \sum_{u\in\Z^{\Lambda}} \mu_\Lambda(u) u ~=~ \sum_{x\in\Z^d} \mu(x)
\, x^\Lambda ~=~ M^{\Lambda} ~=~ 0.
\]
The induced Markov chain $(X_M(t))$  is identical to the random walk $S^{\Lambda(M)}(t)$
  for $t < \tau_{\Lambda(M)} = \inf\{ t \geq 0 : S^i(t) \leq 0 \; \text{ for some } \;
i\in\Lambda(M)\}$ and killed at the time $\tau_{\Lambda(M)}$. 
The transition kernel of the Markov chain $(X_M(t))$ is $$P_+ ~=~ \left(p(u,u') ~=~ \mu_{\Lambda(M)}(u'-u), \;
u,u'\in\Z^{\Lambda(M)}_+\right).$$ 
Under the hypotheses (A1)-(A4), the probability measure $\mu_{\Lambda(M)}$ satisfies
the conditions
(C1)-(C4) of Proposition~\ref{pr9-3}, and consequently, every
minimal harmonic function $f > 0$ of  $(X_M(t))$ is a constant multiple of the function 
\[
f(u) ~=~ \prod_{i\in\Lambda(M)} u^i ~-~ \E_{u}\left( \prod_{i\in\Lambda(M)} S^i\left(\tau_{\Lambda(M)}\right)\right). 
\] 
By Theorem~\ref{th1}, from this it follows that every sequence $x_n\in\Z^d_+$ with $\lim_n
|x_n| = \infty$ and $\lim_n x_n/|x_n| ~=~ M$ is fundamental for the Markov chain $(Z(t))$
and for any $x\in\Z^d_+$, 
\[
\lim_n G(x,x_n)/G(x_0,x_n) ~=~ h(x)/h(x_0)
\]
with 
\[
h(x) ~=~ f(x^{\Lambda(M)}) - \E_x\Bigl( f\bigl(S^{\Lambda(M)}(\tau)\bigr), \; \tau <
\tau_{\Lambda(M)}\Bigr). 
\]
Remark finally  that for any $x\in\Z_+^d$, 
\begin{align*}
h(x) ~=~ \prod_{i\in\Lambda(M)} x^i  -~ \E_{x}\left( \prod_{i\in\Lambda(M)}
S^i\left(\tau_{\Lambda(M)}\right)\right) -~ \E_x\left(\prod_{i\in\Lambda(M)}
S^i\left(\tau\right), \; \tau < \tau_{\Lambda(M)}  \right) \\ + \E_x\left(
\E_{S^{\Lambda(M)}(\tau)}\Bigl( \,\prod_{i\in\Lambda(M)}
S^i\left(\tau_{\Lambda(M)}\right)\Bigr), \; \tau < \tau_{\Lambda(M)} \right).
\end{align*}
Since by strong Markov property, the last term is equal to 
\[
 \E_x\Bigl(\;\prod_{i\in\Lambda(M)}
S^i\left(\tau_{\Lambda(M)}\right), \; \tau < \tau_{\Lambda(M)} \Bigr).
\]
we conclude that for any $x\in\Z_+^d$, 
\begin{align*}
h(x) &~=~ \prod_{i\in\Lambda(M)} x^i  ~-~ \E_{x}\Bigl( \; \prod_{i\in\Lambda(M)}
S^i\left(\tau_{\Lambda(M)}\right), \; \tau \geq \tau_{\Lambda(M)} \Bigr) \\&\hspace{6cm}
-~ \E_x\Bigl(\; \prod_{i\in\Lambda(M)}
S^i\left(\tau\right), \; \tau < \tau_{\Lambda(M)}  \Bigr) \\
&~=~ \prod_{i\in\Lambda(M)} x^i  ~-~ \E_{x}\Bigl( \; \prod_{i\in\Lambda(M)}
S^i\left(\tau\right), \; \tau < \infty \Bigr). 
\end{align*}

\providecommand{\bysame}{\leavevmode\hbox to3em{\hrulefill}\thinspace}
\providecommand{\MR}{\relax\ifhmode\unskip\space\fi MR }
\providecommand{\MRhref}[2]{%
  \href{http://www.ams.org/mathscinet-getitem?mr=#1}{#2}
}
\providecommand{\href}[2]{#2}

\end{document}